\documentclass[12pt]{dmathesis}

\usepackage[T1]{fontenc}
\usepackage{lmodern}

\usepackage{amsmath,amssymb,amsthm,graphicx,color}
\usepackage{amsfonts}
\usepackage{verbatim}
\usepackage{cite,multirow}
\usepackage[british]{babel}
\usepackage{enumerate, multicol}
\usepackage{enumitem}
\usepackage{hyperref} 

\usepackage{float}

\usepackage{tikz}
\usetikzlibrary{calc}
\usepackage{xcolor,pgf} 
\definecolor{DBlack}{RGB}{35,31,32}
\definecolor{DPurple}{RGB}{126,49,123}

\hypersetup{
    colorlinks=false,
    pdfborder={0 0 0},
}

\newtheorem{theorem}{Theorem}[section]

\newtheorem{proposition}[theorem]{Proposition}
\newtheorem{lemma}[theorem]{Lemma}
 \newtheorem{conjecture}[theorem]{Conjecture}

\numberwithin{equation}{section}

\theoremstyle{definition}
\newtheorem{definition}[theorem]{Definition}
\newtheorem{examplex}[theorem]{Example}

\theoremstyle{remark}
\newtheorem{remark}[theorem]{Remark}
\newtheorem{remarks}[theorem]{Remarks}
\newtheorem*{remark*}{Remark}

\medskip
\allowdisplaybreaks

\newcommand{\Z}{\mathbb{Z}}
\newcommand{\ZP}{\mathbb{Z}_+}
\newcommand{\R}{\mathbb{R}}
\newcommand{\RP}{\mathbb{R}_+}
\newcommand{\N}{\mathbb{N}}
\newcommand{\Sp}{\mathbb{S}}

\newcommand{\1}[1]{{\mathbf 1}{\{#1\}}}
\newcommand{\2}[1]{{\mathbf 1}{(#1)}}

\newcommand{\eps}{\varepsilon}

\newcommand{\blob}{\mkern2mu\raisebox{2pt}{\scalebox{0.4}{$\bullet$}}\mkern2mu}

\renewcommand{\Pr}{\mathbb{P}}

\newcommand{\bpi}{\mbox{\boldmath${\pi}$}}

\DeclareMathOperator{\E}{\mathbb{E}}
\DeclareMathOperator{\Var}{\mathbb{V}ar}

\newcommand{\as}{{\ \mathrm{a.s.}}}

\DeclareMathOperator{\trace}{tr} 
\newcommand{\tra}{{\scalebox{0.6}{$\top$}}}

\newcommand{\ba}{\mathbf{a}}
\newcommand{\bb}{\mathbf{b}}
\newcommand{\bd}{\mathbf{d}}
\newcommand{\bx}{\mathbf{x}}
\newcommand{\by}{\mathbf{y}}
\newcommand{\0}{\mathbf{0}}
\newcommand{\re}{{\mathrm{e}}}
\newcommand{\rc}{{\mathrm{c}}}
\newcommand{\ud}{{\mathrm d}}

\newcommand{\be}{\mathbf{e}}
\newcommand{\bz}{\mathbf{z}}
\newcommand{\bk}{\mathbf{k}}
\newcommand{\bt}{\mathbf{t}}
\newcommand{\bu}{\mathbf{u}}
\newcommand{\bs}{\mathbf{s}}
\newcommand{\bmu}{\boldsymbol{\mu}}
\newcommand{\vo}{\mathbf{1}}

\newcommand{\tX}{\widetilde{X}}

\newcommand{\tod}{\stackrel{d}{\longrightarrow}}

\newcommand{\toas}{\overset{\textup{a.s.}}{\longrightarrow}}

\newcommand{\eqd}{\stackrel{d}{=}}

\newcommand{\cF}{{\mathcal{F}}}
\newcommand{\cH}{{\mathcal{H}}}

\newcommand{\cN}{{\mathcal{N}}}

\newcommand{\cL}{{\mathcal{L}}}

\newcommand{\cX}{{\mathcal{X}}}

\newcommand{\calN}{\mathcal{N}}

\DeclareMathOperator*{\supp}{supp}

\makeatletter
\def\namedlabel#1#2{\begingroup  
    (#2)%
    \def\@currentlabel{#2}%
    \phantomsection\label{#1}\endgroup
}
\makeatother

\begin{document}

\title{On some random walk problems}
\author{Chak Hei Lo}
\researchgroup{Probability group}

\date{October 2017}


\begin{abstract*}

We consider several random walk related problems in this thesis. In the first part, we study a Markov chain on $\RP \times S$, where $\RP$ is the non-negative real numbers and $S$ is a finite set, in which when the $\RP$-coordinate is large, the $S$-coordinate of the process is approximately Markov with stationary distribution $\pi_i$ on $S$. Denoting by $\mu_i(x)$ the mean drift of the $\RP$-coordinate of the process at $(x,i) \in \RP \times S$, we give an exhaustive recurrence classification in the case where $\sum_{i} \pi_i \mu_i (x) \to 0$, which is the critical regime for the recurrence-transience phase transition. If $\mu_i(x) \to 0$ for all $i$, it is natural to study the \emph{Lamperti} case where $\mu_i(x) = O(1/x)$; in that case the recurrence classification is known, but we prove new results on existence and non-existence of moments of return times. If $\mu_i (x) \to d_i$ for $d_i \neq 0$ for at least some $i$, then it is natural to study the \emph{generalized Lamperti} case where $\mu_i (x) = d_i + O (1/x)$. By exploiting a transformation which maps the generalized Lamperti case to the Lamperti case, we obtain a recurrence classification and an existence of moments result for the former. The generalized Lamperti case is seen to be more subtle, as the recurrence classification depends on correlation terms between the two coordinates of the process.

In the second part of the thesis, for a random walk $S_n$ on $\R^d$ we study the asymptotic behaviour of the associated centre of mass process $G_n = n^{-1} \sum_{i=1}^n S_i$. For lattice distributions we give conditions for a local limit theorem to hold. We prove that if the increments of the walk have zero mean and finite second moment, $G_n$ is recurrent if $d=1$ and transient if $d \geq 2$. In the transient case we show that $G_n$ has diffusive rate of escape. These results extend work of Grill, who considered simple symmetric random walk. We also give a class of random walks with symmetric heavy-tailed increments for which $G_n$ is transient in $d=1$.


\pagenumbering{gobble}

\end{abstract*}

\pagenumbering{roman}


\maketitlepage*


\tableofcontents*

\chapter*{List of Assumptions}
\addcontentsline{toc}{chapter}{List of Assumptions}
\begin{description}
\item
[\namedlabel{}{A}]
Suppose that $(X_n , \eta_n)$, $n \in \ZP$, is a time-homogeneous, irreducible Markov chain on $\Sigma$,
a
locally finite subset of  $\RP \times S$. Suppose that for each $k \in S$ the line $\Lambda_k$ is unbounded.
\item
[\namedlabel{}{B$_\textit{p}$}]
There exists a constant $C_p< \infty$ such that for all $n \in \ZP$,
\[
\E [|X_{n+1}-X_n|^p \mid X_n = x, \, \eta_n = i ] \le C_p , \text{ for all }  (x,i) \in \Sigma.
\]  
\item
[\namedlabel{}{D$_\text{C}$}]
For each $i \in S$ there exists $d_i \in \R$ such that $\mu_i(x) = d_i + o(1)$ as $x \to \infty$.
\item
[\namedlabel{}{D$_\text{G}$}]
 For $i, j \in S$ there exist  $d_i \in \R$, $e_i \in \R$, $d_{ij} \in \R$ and $t^2_i \in \RP$, with at least one $t^2_i$ non-zero, such that  
\begin{itemize}
\item[(a)] for all $i \in S$, $\mu_i(x) = d_i+ \frac{e_i}{x} + o(x^{-1})$ as $x \to \infty$;
\item[(b)] for all $i \in S$, $\sigma^2_i(x) = t^2_i + o(1)$ as $x \to \infty$;
\item[(c)] for all $i,j \in S$, $\mu_{ij}(x) = d_{ij} + o(1)$ as $x \to \infty$; and
\item[(d)] $\sum_{i \in S} \pi_i d_i = 0$.
\end{itemize}
\item
[\namedlabel{}{D$_\text{G}^+$}]
There exist $\delta_2 \in (0,1)$, $d_i \in \R$, $e_i \in \R$, $d_{ij} \in \R$ and $t^2_i \in \RP$, with at least one $t^2_i$ non-zero, such that  
\begin{itemize}
\item[(a)] for all $i \in S$, $\mu_i(x) = d_i+ \frac{e_i}{x} + o(x^{-1-\delta_2})$ as $x \to \infty$;
\item[(b)] for all $i \in S$, $\sigma^2_i(x) = t^2_i + o(x^{-\delta_2})$ as $x \to \infty$; and
\item[(c)] for all $i,j \in S$, $\mu_{ij}(x) = d_{ij} + o(x^{-\delta_2})$ as $x \to \infty$.
\end{itemize}
\item
[\namedlabel{}{D$_\text{L}$}]
For each $i \in S$ there exist $c_i \in \R$ and $s^2_i \in \RP$, with at least one $s^2_i$ non-zero, such that, as $x \to \infty$,
 $\mu_i(x) = \frac{c_i}{x} + o(x^{-1})$ and $\sigma_i^2(x) = s^2_i + o(1)$.
\item
[\namedlabel{}{D$_\text{L}^+$}]
Suppose that there exist $\delta_1 \in (0,1)$, $c_i \in \R$, and $s^2_i \in \RP$, with at least one $s^2_i$ non-zero, such that for all $i \in S$, as $x \to \infty$,
 $\mu_i(x) = \frac{c_i}{x} + o(x^{-1-\delta_1})$ and $\sigma_i^2(x) = s^2_i + o(x^{-\delta_1})$.
\item
[\namedlabel{}{L}]
Suppose that the minimal subgroup of $\R^d$ associated with $X$ is $L := H \Z^d$ with $h := | \det H | > 0$.
\item
[\namedlabel{}{M}]
Suppose that $\E [ \| X \|^2 ] < \infty$ and $M$ is positive-definite. 
\item
[\namedlabel{}{Q$_\infty$}]
Suppose that $\lim_{x \to \infty} q_{ij}(x)=q_{ij}$ exists for all $i,j \in S$, and $(q_{ij})$ is an irreducible stochastic matrix. 
\item
[\namedlabel{}{Q$_\infty^+$}]
Suppose that there exists $\delta_0 \in(0,1)$ such that $\max_{i,j \in S}|q_{ij}(x)-q_{ij}|=O(x^{-\delta_0})$ as $x \to \infty$.
\item
[\namedlabel{}{Q$_\text{G}$}]
 For $i, j \in S$ there exist $\gamma_{ij} \in \R$ such that $q_{ij}(x)=q_{ij}+\frac{\gamma_{ij}}{x} +o(x^{-1})$, where $(q_{ij})$ is a stochastic matrix.
\item
[\namedlabel{}{Q$_\text{G}^+$}]
 There exist $\delta_3 \in (0,1)$ and $\gamma_{ij} \in \R$ such that $q_{ij}(x)=q_{ij}+\frac{\gamma_{ij}}{x} +o(x^{-1-\delta_3})$.
\item
[\namedlabel{}{S}]
Suppose that $X \eqd -X$ and $X$ is in the domain of normal attraction of a symmetric $\alpha$-stable  distribution with $\alpha \in (0,1)$.
\item
[\namedlabel{}{\textbf{V}}]
Suppose that $\E [\|\xi\|^2] < \infty$ and write $\Sigma := \E [(\xi-\mu)(\xi-\mu)^\tra]$.
Here $\Sigma$ is a nonnegative-definite, symmetric $d$ by $d$ matrix; 
we write $\sigma^2 := \trace \Sigma = \E [ \| \xi - \mu \|^2 ]$.
\item
[\namedlabel{}{\textbf{W}$_{\mu}$}]
Let $d \in \N$, and suppose that $Z, Z_1, Z_2, \ldots$ are i.i.d.~random variables with $\E \|Z\| < \infty$ and $\E Z=\mu \in \R^d$. 
The random walk $(S_n,n\in \ZP)$ is the sequence of partial sums $S_n := \sum_{i=1}^n Z_i$ with $S_0 := 0$.
\end{description}

\listoffigures

\begin{declaration*}
\addcontentsline{toc}{chapter}{Declaration}

The work in this thesis is based on research carried out in the Department of Mathematical Sciences, at the University of Durham, England. No part of this thesis has been submitted elsewhere for any other degree or qualification. It is all my own work unless referenced to the contrary in the text.

Parts of Chapters 2 to 5 are adapted from joint work with Andrew R. Wade \cite{LW1}. 

Parts of Chapters 6 to 9 are adapted from joint work with Andrew R. Wade \cite{LW2}. 

\end{declaration*}


\begin{acknowledgements*}
\addcontentsline{toc}{chapter}{Acknowledgements}

I am extremely thankful to my supervisors, Mikhail Menshikov and Andrew Wade, for all of their guidance and encouragement, which enabled me to develop a deep understanding of the theory of random walks. They gave me much inspiration and patiently listened to many of my presentations throughout the last three years. 

I am also grateful to Nicholas Georgiou and Ostap Hryniv for fruitful discussions on the topics of this thesis. 

Thank you to everyone in the Probability and Statistics group at Durham for creating an enjoyable working environment. Special thanks should be given to all of the postgraduate students in the department who broadened my view with discussions on a variety of Mathematical topics. 

Finally, words alone cannot express the sincere thanks I owe to my parents, for their moral support, especially in helping me to pass through some hard times.

\end{acknowledgements*}

\begin{dedication*}
		\aaa 
		\also 
		\bbb
\end{dedication*}

\cleardoublepage
\pagenumbering{arabic}

\chapter{Introduction}
\label{introduction}

Many \emph{stochastic processes} arising in applications exhibit a range of possible behaviours depending upon the values of certain key parameters. Investigating \emph{phase transitions} for such systems leads to interesting and challenging mathematics. Much progress has been made over the years, using various techniques. The most subtle case is when the system is near-critical in some sense (near a phase boundary). This thesis will study a few particular near-critical Markov models, with an aim to extend known criteria for classifying \emph{recurrence} and \emph{transience}.

Now we will start on some background knowledge and classical results on random walk theory, together with some new intuitions.

\section{Random walk}

\emph{Random walk} is one of the most important models in probability theory. It displays profound mathematical properties and has a wide range of application in many scientific fields and much more. It is a stochastic process which describes the random trajectory of a particle (or random walker) in space. The motion of the particle is explained with a succession of random increments or jumps at discrete instants in time. The \emph{long term asymptotic behaviour} of the particle or walker is of great interest and has stimulated extensive research in this field. It has a long and rich history across a variety of subjects. The classical one-dimensional random walk dates back to the `gambler's ruin' problem, addressed a few centuries ago by Fermat and Pascal\cite{IT}. The mathematical theory started to formalize as the French mathematician Louis Bachelier gave his insight to his stock prices model using the random walk reasoning in his Ph.D. thesis in 1900 \cite{LB}. The popularity of the term `random walk' gradually increased when a Professor of Economics at Princeton University, Burton Malkiel, published his book, A Random Walk Down Wall Street, in 1973 \cite{MB}. 

For the more general version of the model in \emph{several dimensions}, it was probably first studied in around 1880 in the form of Lord Rayleigh's theory of sound \cite{LR}. Shortly after, similar ideas from Albert Einstein's theory of Brownian motion (1905-1908) in statistical physics \cite{AE} and English statistician Karl Pearson's theory of random migration of species (1906) in biology \cite{PB3} arose. The term random walk is first suggested by Pearson in a letter to the journal Nature \cite{PK} and it is stated as a path with a succession of random steps, usually on a $d$-dimensional lattices in classical literature.

In 1920, the Hungarian mathematician George P\'olya confirmed the mathematical importance of this indispensable random walk model \cite{GP}. Numerous elegant connections and ideas in random walk blossom and propagate to other significant branches of mathematics such as combinatorics, harmonic analysis, potential theory, and spectral theory over the last century. The theory of random walk then continued to proliferate in lively realm of modern science. A broad range of studies can be found in \cite{MFS}. 

The popularity of researching the random walk model is due to its vast applications in different subjects such as, but no limited to the following.

\begin{itemize}
\item Chemistry: Polymer conformation in molecular chemistry \cite{MNB,GGL,GW};
\item Biology: Modelling of microbe locomotion in microbiology \cite{HCB,GW};
\item Economics: Financial systems, modelling stock prices\cite{EF};
\item Psychology: Human memory search in a semantic network \cite {JA}.
\end{itemize}

In this chapter, we will discuss some of the history and motivation behind the study of such random walk problems. We will also give some foundation material on random walk theory with some personal intuition. Let's discover these hidden gems through the exciting adventures of some random walk problems.

\section{Markov chains and recurrence classification}
\label{s:mcrc}

The \emph{Markov process}, named after the Russian mathematician Andrey Markov, has a characteristic property that it retains no memory of where it has been in the past. This property is sometimes known as the \emph{Markov property} or the \emph{memorylessness property}. In other words, where the process will go next only depends on the current state of the process. By conditioning on the current state of the process, its future and past states are independent. When the Markov process has a finite or countable set of states in particular, we would call it a \emph{Markov chain}.

Although Andrey Markov studied Markov chains and Markov processes, with his first paper on these topics in 1906, other specific models of Markov processes already existed. Random walk is an example of a Markov chain, and was studied hundreds of years earlier \cite{GHW}.

Compared to the usual use of the term random walk, which suggests that the process is on a regular lattice, Markov chains are usually more general in terms of describing a more complicated state space. As both of them are stochastic processes, we would not distinguish them specifically in the context of this thesis, and will use them interchangeably.

A very important property for Markov chains is the \emph{recurrence classification}. It gives us a general idea of how the process will evolve in the long term. Given a Markov chain $(X_n), n \ge 0$ on a countable state space $S$, a state $i \in S$ is called \emph{recurrent} if 
\begin{equation*}
\Pr (X_n = i \text{ for infinitely many } n | X_0 = i) =1.
\end{equation*}
A state $i \in S$ is called \emph{transient} if
\begin{equation*}
\Pr (X_n = i \text{ for infinitely many } n | X_0 = i) =0.
\end{equation*}
Although it is not immediate, standard Markov chain theory shows that any state can only be either recurrent or transient, see \cite[p.26, Theorem 1.5.3]{JN}.

We can also understand the idea of recurrence and transience by looking at the \emph{return time}, also know as the \emph{first passage time} and the \emph{hitting time}, defined as follows. For $i \in S$,
\begin{equation*}
\tau_i = \inf \{n \ge 1: X_n = i \},
\end{equation*}
with the convention that $\inf \emptyset := \infty$. Intuitively if $X_0=i$, $\tau_i$ is time it takes for the process to come back to its original position. Again, from standard Markov chain theory, we can easily see that a state $i$ is recurrent if and only if $\Pr (\tau_i <\infty | X_0 = i) =1$ and it is transient if and only if $\Pr (\tau_i <\infty | X_0 = i) <1$.

If a state is recurrent, it implies that the process will come back to this state with probability one, but it does not guarantee that the process will come back in finite time in expectation. Hence we could further classify the recurrent case into positive recurrent or null recurrent. We define a recurrent state $i$ to be \emph{positive recurrent} if 
\begin{equation*}
\E [\tau_i | X_0=i] < \infty
\end{equation*}
and \emph{null recurrent} if
\begin{equation*}
\E [\tau_i | X_0=i] = \infty.
\end{equation*}
This time, it is clear that it is a \emph{dichotomous} classification.

In order to understand the recurrence classification for the whole process, we should understand the structure of the walk first. Sometimes, it is possible to break a chain into smaller pieces, so that we can understand the behaviour of each piece separately in a relatively simple way, and group them all back together to get a result for the whole chain. This involves identification of \emph{communication classes} of the chain. 

Given a Markov chain $(X_n), n \ge 0$ on a countable state space $S$, for any states $i, j \in S$ we say that $i$ \emph{leads to} $j$ and write $i \to j$ if

\begin{equation*}
\Pr (X_n = j \text{ for some } n \ge 0 | X_0 = i) >0.
\end{equation*}

We also say that $i$ \emph{communicates with} $j$ and write $i \leftrightarrow j$ if both $i \to j$ and $j \to i$. It is clear that $i \leftrightarrow  i$ from the definition. Together with the fact that $i \leftrightarrow  j$ and $j \leftrightarrow  k$ implies $i \leftrightarrow  k$ for any states $i$, $j$ and $k \in S$, we conclude that $\leftrightarrow$ is an equivalence relation on $S$. So we can partition $S$ into \emph{communicating classes}. If a chain only consists of one class, then it is called \emph{irreducible}.

From standard Markov chain theory, the properties of positive recurrence, null recurrence and transience are all class properties. This means if a state in a certain class is transient, then any state in the class is also transient. 

In our context of random walks in this thesis, they are always irreducible Markov chains, hence the recurrence classification for the process (with certain fixed parameters) we considered as a whole is well defined.

Hence when we say recurrence classification in context of this thesis, we want to determine how the parameters in the model will affect the process to be positive recurrent, null recurrent, or transient.

\section{Simple symmetric random walk}

The most comprehensively studied random walk model is the \emph{simple symmetric random walk}. Formally, denote by $\{\be_1 , \be_2, \cdots, \be_d \}$ the standard orthonormal basis on $\R^d$, and let $U_d := \{\pm \be_1 , \pm \be_2, \cdots, \pm \be_d \}$ be the set of possible jumps of the random walk. Given a sequence of independent identically distributed (i.i.d.) random variables $Z, Z_1, Z_2, \ldots$, with
\begin{equation}
\Pr (Z = \be) = \frac{1}{2d} \quad \text{for } \be \in U_d,
\end{equation}
we define the \emph{simple symmetric random walk} as a discrete-time Markov process $(S_n, n \ge 0)$ on the $d$-dimensional integer lattice $\Z^d$ by 
\begin{equation}
S_n = \sum_{k=1}^n Z_k.
\end{equation}
Alternatively, we can think about this process in the natural way. To move from a certain point $S_n$ to the next point $S_{n+1}$ in $\Z^d$, we chose, uniformly at random, from all of the $2d$ neighbours of $S_n$, in other words, all the points which differ from $S_n$ by exactly $\pm 1$ in a single coordinate.  Here are some pictures of simple symmetric random walks in one, two and three dimensions.

\begin{figure}[H]
\color{blue}
\begin{center}
\includegraphics[width=9.5cm, angle=0, clip = true]{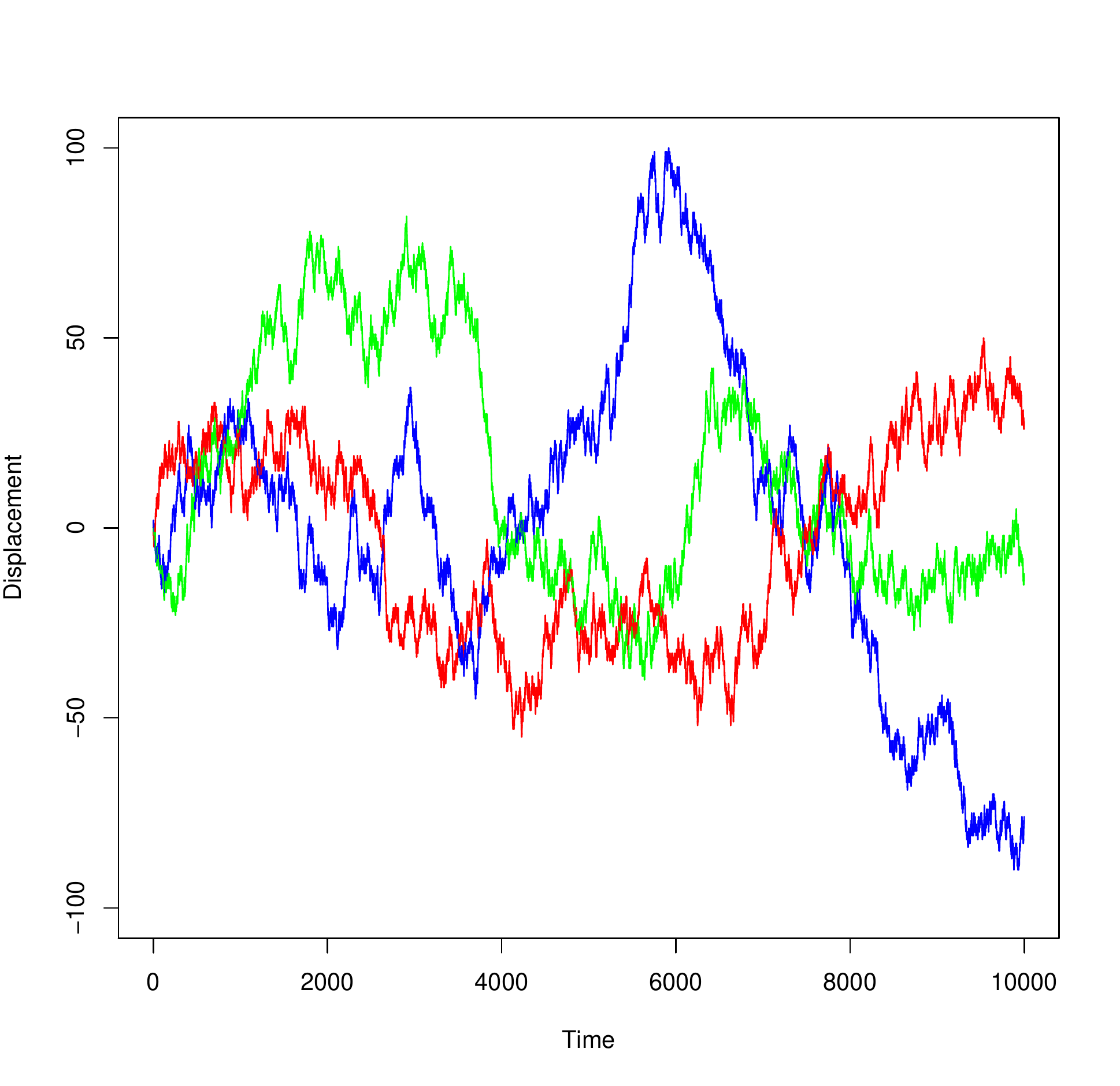}
\caption{Three simulated trajectories of 1D SSRW against time.}
\label{fig_RWS1}
\end{center}
\end{figure}
\vspace{-10mm}
\begin{figure}[H]
\color{blue}
\begin{center}
\includegraphics[width=9.5cm, angle=0, clip = true]{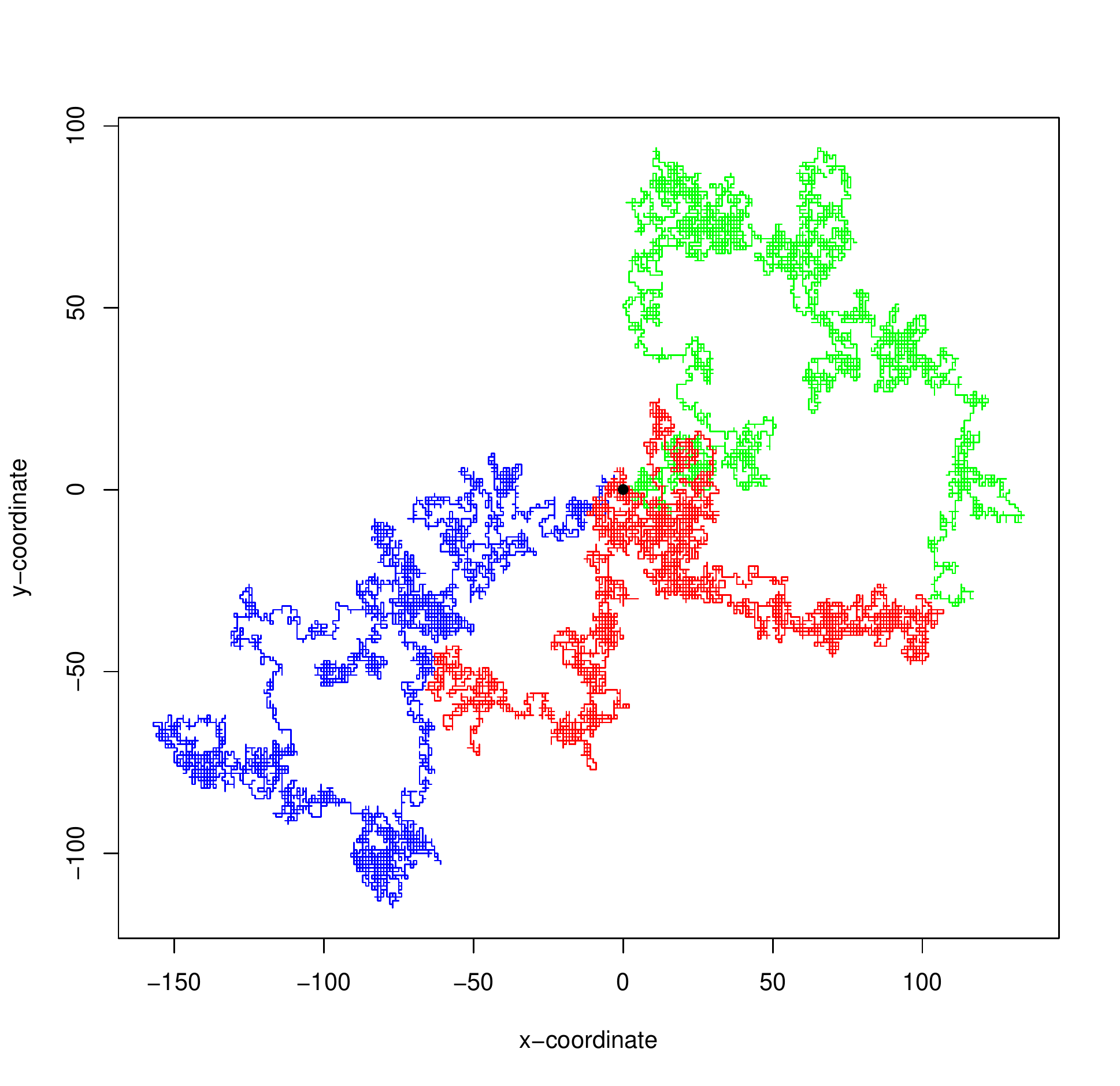}
\caption{Three simulated trajectories of 2D SSRW.}
\label{fig_RWS2}
\end{center}
\end{figure}

\vspace{-10mm}
\begin{figure}[H]
\color{blue}
\begin{center}
\includegraphics[width=12cm, angle=0, clip = true]{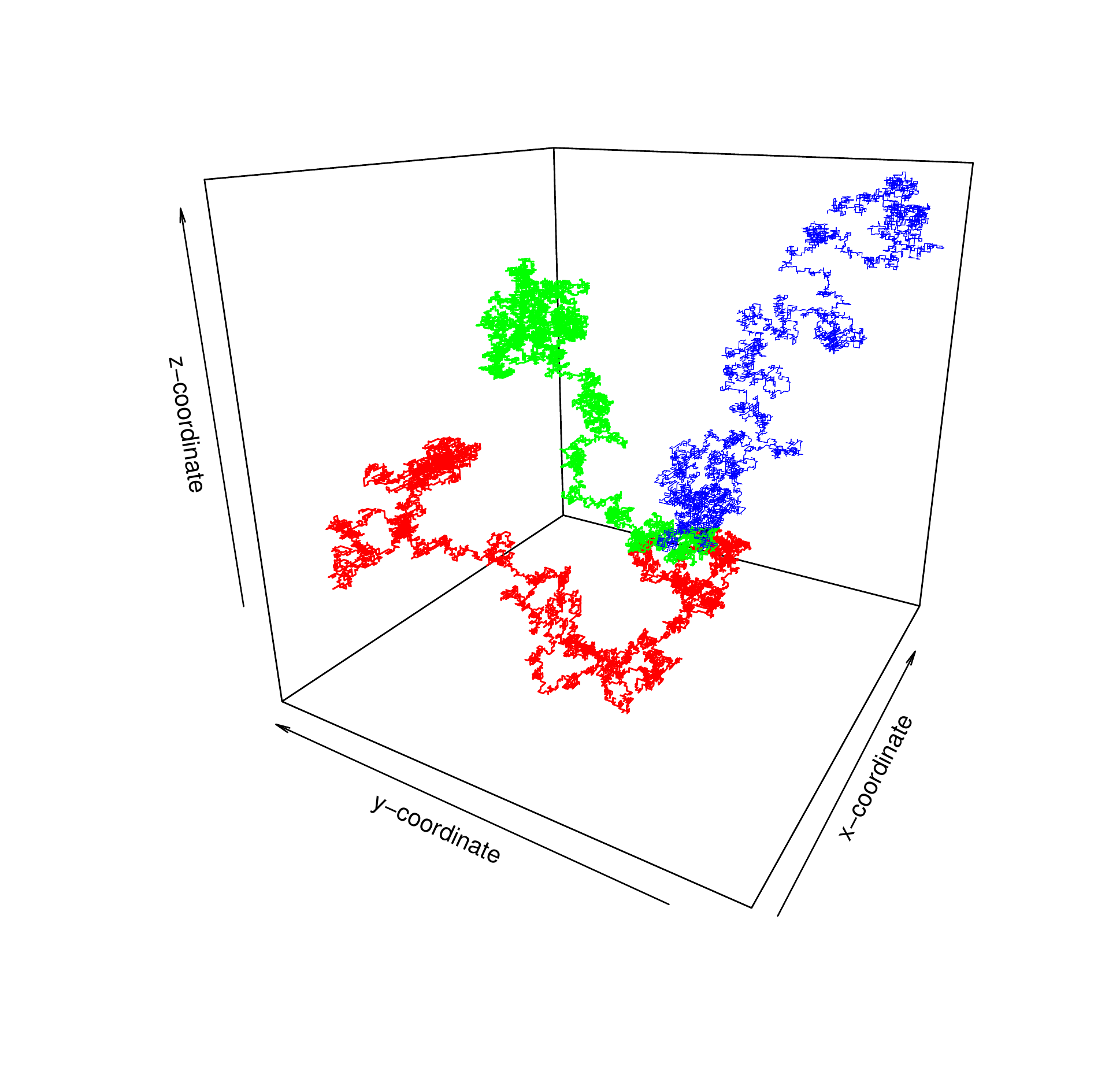}
\vspace{-10mm}
\caption{Three simulated trajectories of 3D SSRW.}
\label{fig_RWS3}
\end{center}
\end{figure}

\vspace{-10mm}
One of the most fundamental properties of a random walk is the recurrence property. The story goes back to 1920s. George P\'olya enjoyed to take random running paths in a big park as his daily exercise. Although his paths were completely random, he often met the same couple during his journey, who was also running around the area \cite{GP}. He realized that assuming the couple also takes a random path every day, then his relative position to the couple is also a random walk. This can be done by just combining the two steps of the random walks by P\'olya and the couple at every time point as one big step. Then they will meet each other whenever the combined random walk visits the origin. Now the real question is, what is the probability that the walk will eventually returns to $\mathbf{0}$? Mathematically, define $\tau_d := \min\{n \ge 1: S_n= \mathbf{0} \}$ to be the time needed for the first return to the origin. If the walk never comes back, then $\tau_d = \infty$, as with the usual convention that $\min \emptyset := \infty$. Now our interest is in the \emph{P\'olya's random walk constant} $p_d$, defined as
\begin{equation}
p_d := \Pr(\tau_d < \infty).
\end{equation}
Similar to the recurrence classification for Markov chains that discussed in Section~\ref{s:mcrc}, we call the random walk \emph{recurrent} if $p_d =1$, and \emph{transient} if $p_d <1$. Intuitively, a recurrent walk means that the random walk will visit the origin infinitely often with probability one while a transient walk means with probability one, it will only come back to the origin finitely many times, and never return again.

In general, finding this classification is very difficult due to the fact that the intrinsic properties of the state space or the movement of the walk is complicated to quantify for meaningful analysis. However, in the case of simple symmetric random walk, which is a pleasant model to study due to the simple and clean structure, there are a lot of well developed combinatorial techniques based on counting sample paths that give us elegant properties of the walk. We now present the following beautiful result by George P\'olya in 1921 \cite{GP}.

\begin{theorem}[P\'olya's Recurrence Theorem]
The simple symmetric random walk on $\Z^d$ is recurrent in one or two dimensions, but transient in three or more dimensions. Equivalently, $p_1=p_2=1$ but $p_d<1$ for all $d \ge 3$. 
\end{theorem}

The essence of this theorem can be easily understood by the aphorism credited to Shizuo Kakutani in a UCLA colloquium talk: `A drunk man will eventually find his way home, but a drunk bird may get lost forever' \cite[p.191]{RD}. My version to remember the critical dimension is by thinking of the sentence `Everyone but astronaut drinks'.

More precisely on the value of $p_d$, Montroll \cite{EWM} in 1956 showed that for $d \ge 3$, $p_d=1-u_d^{-1}$ where 	
\begin{equation}
u_d =	\int_0^\infty \left[I_0 \left(\frac{t}{d} \right) \right]^d e^{-t} \ud t,	
\end{equation}
and $I_0(z)$ is the modified Bessel function of the first kind. Numerically, $p_3 \approx 0.340537$, see \cite{GNW,MW,CD,GZ} and $p_4 \approx 0.193206$ \cite{EWM,SRF}.

The intuition behind this phenomenon is actually quite difficult to come up with. At first sight, one might think as the dimension increases, the number of points in the lattice increases and also more choices are available at each time point, that is why it is more difficult for the particle or the walker to jump back to the origin. This is not a very convincing argument since if you are away from the origin, you have many choices in higher dimensions, but a high proportion of them are `helping' you to get back in terms of shortening the distance from the starting point, then you should still have a lot of tendency to come back. In one dimension, except the starting point, we always have equal tendency to move to or away from the origin. In two dimensions, most of the points on the lattice have equal number of choices to help or not help you to come back, while on the axis there are actually more choices that push you away than those pull you back! However, both one or two dimensions fall into the recurrent case. This argument is unclear from the classification, and there is no hint for why the critical change is from two to three dimensions, but not, say, four to five dimensions. 

In fact, P\'olya's original argument was based on delicate path counting and is largely combinatorial, which the intuition remains hidden behind. Some other intuition is based on the proof by electric networks and potential theory technique. The end of the proof boils down to the convergence of harmonic series. The increase of dimension changes the convergence to divergence, and thus the critical point emerges from two to three dimensions, algebraically. Again, this is not a very satisfactory explanation due to the lack of explaining the physical meaning of how the dimension affects the series. 

If we want to generalize the above methods to more general random walks, they just completely break down due to the complicated structure or long distance correlation. We realized that not only the average drift in the model matters, but the variance of jumps is equally important.

One of the heuristic and intuitive arguments that I came across in the literature is the following. Consider the random walk in $\R^d$ then the probability of the random walk being within distance $O(1)$ of the origin after $n$ steps will become order $O(n^{-\frac{d}{2}})$ from the local limit theorem for random walk, that will be explained in Section 6.5. Now if we consider all possible $n$ and sum the probabilities up, we get an expression $\sum_{n=1}^{\infty} n^{-\frac{d}{2}}$ which is divergent when $d=2$ and convergent when $d=3$. By the Borel-Cantelli lemma this gives a sufficient condition for transience. However, this argument does not give both directions, i.e. the divergent sequence does not imply recurrence directly.

In my own opinion, the best and neatest argument is using the idea of Lyapunov functions which can be found in \cite{MPW}, which involves a version of Lamperti's fundamental recurrence classification \cite{L1}. We will delay this argument to Chapter~2.


\section{Homogeneous random walk on $\R^d$}

Simple symmetric random walk is a specific model that is very restrictive to the movement of the walk. It is natural to extend the theory to a more general class of random walks. A famous intermediate extension involves the Pearson-Rayleigh random walk on $\R^d$, which allows the walk to jump to any point on the unit circle/sphere centred at the current position, with uniform probability. Similar results to those for the simple symmetric random walk can also be obtained. In fact, we can do far more than this. Without any particular structure of the jump, we define a random walk as a discrete-time Markov process $(S_n; n \ge 0)$ on an unbounded state space $\Sigma \subseteq \R^d$. Throughout the whole thesis, we always assume the walk is time-homogeneous, i.e. the distribution of $S_{n+1}$ given $(S_0, S_1, \ldots, S_n)$ only depends on $S_n$ but not on $n$. 

A typical type of random walk that was studied extensively in the literature is the \emph{spatially homogeneous random walk}. We can define it as $S_n = \sum_{k=1}^n Z_k$ where $Z, Z_1, \ldots, Z_n$ are i.i.d. random variables, taking values in $\R^d$, so the law of the increment does not depend on the current position of the walk.

In the context of the general random walk, there are some results on the generalization of the seminal P\'olya's recurrence theorem for the continuous state space $\R^d$. However, we need to reconsider the definition of recurrence and transience again. The original definition of recurrence is not completely clear in a continuous state space. Do we insist of the walk going back to the exact same point or do we allow the walk just come back to a small neighbourhood of the point it visited in the past? These two situation exhibit a very different behaviour in critical situations. Hence we should separate them clearly. Without any specification on the structure of the walk, we will use the following definition.

\begin{definition}
A random walk $(S_n; n \ge 0)$ taking values in $\Sigma \subseteq \R^d$ is transient if $\lim_{n \to \infty} \| S_n\|  = \infty$, a.s. The walk is recurrent if, for some constant $r_0 \in \RP$, $\liminf_{n \to \infty} \|S_n\| \le r_0$.
\end{definition}


It is very important to know that the classification of recurrence and transience is not necessarily exhaustive in general, we will look deeper in this later in our specific model. In the special case of spatially homogeneous random walk, one can apply the Hewitt-Savage zero-one law to prove the dichotomy. We will explain in more details in Part II of the thesis. Also, even with these more general definitions, we need to make sure that the walk should not be `trapped' in part of the state space as the transient definition suggest the walk will go to infinity eventually, but here the walk can just go to a finite limiting point, breaking the dichotomy. So the classification is not properly defined in this case. The easiest way is to assume the state space $\Sigma$ to be locally finite to get some form of irreducibility so we can avoid the ambiguity on the recurrence classification.

Now we are ready to generalize the influential result of P\'olya's recurrence theorem. In 1951, Two mathematicians Kai-lai Chung and Wolfgang Heinrich Johannes Fuchs (see \cite{KLC} and \cite[Chapter 9]{OK}) extended the result to non-degenerate homogeneous random walks whose increments have finite second moments as follows.

\begin{theorem} [Chung-Fuchs Theorem]
Let $S_n$ be a random walk in $\R^d$. Then we have the following statements.
\begin{enumerate}
\item When $d = 1$, if $\E \left[ |Z| \right] < \infty$ and $\E [Z] = 0$, then $S_n$ is recurrent. 
\item When $d = 2$, if $\E \left[ Z^2 \right] < \infty$ and $\E [Z] = 0$, then $S_n$ is recurrent. 
\item If $d \ge 3$ and the random walk is not contained in a lower-dimensional sub-space, then it is transient.
\end{enumerate}
\end{theorem}

Notably, the Brownian motion, as a continuous version of the simple symmetric random walk, exhibits similar behaviour. However, the proof does not follow by the theorem above.

Compared to the classic path counting proof of P\'olya's theorem, the proof of the Chung-Fuchs theorem is based on Fourier analysis. Although the methods are different, they both retain the unsatisfactory fact that intuition is still hidden behind the calculations. 

In the early 1960s, John W. Lamperti made a momentous breakthrough on developing the approach of Lyapunov functions \cite{L1}. This method can be applied to a broader variety of random walks than the combinatorial and analytical approaches. Just as importantly, it is probably the first method which clarifies the probabilistic intuition behind the recurrence classification problem. We will see more about this in the next chapter.

At the end of this section we will provide some pictures of homogeneous random walks in two dimensions. The behaviour can vary a lot depending on the properties of the walk.

\begin{figure}[H]
\color{blue}
\begin{center}
\includegraphics[width=12cm, angle=0, clip = true]{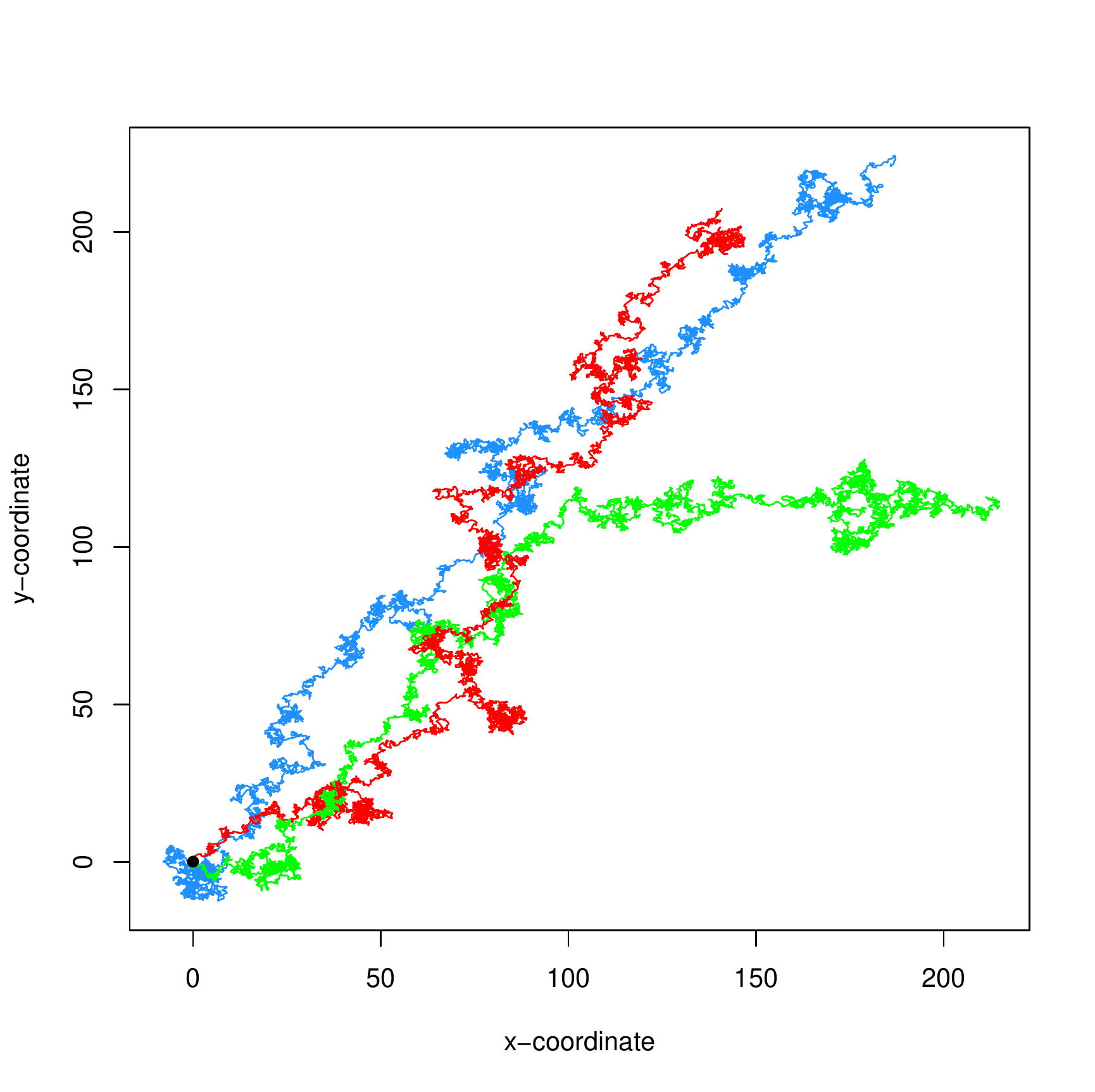}
\caption{Three simulated paths of two dimensional random walk with drift.}
\label{fig_RWSD2}
\end{center}
\end{figure}

\begin{figure}[H]
\color{blue}
\begin{center}
\includegraphics[width=12cm, angle=0, clip = true]{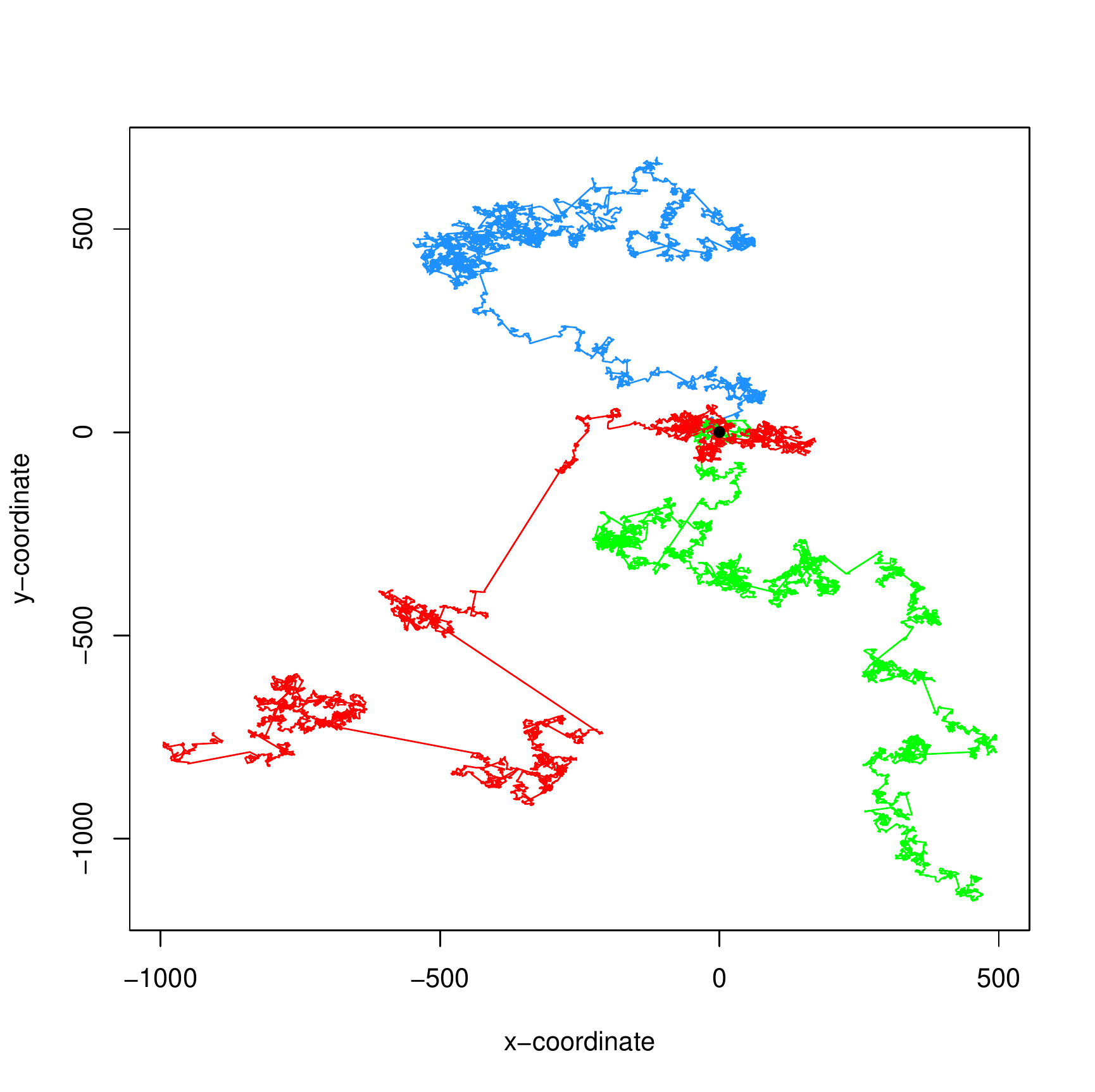}
\caption{Three simulated paths of two dimensional random walk with heavy-tailed distributions.}
\label{fig_RWSH2}
\end{center}
\end{figure}

\section{Non-homogeneous random walk on $\R^d$}

Now we would like to go a step further to ease the restriction of spatial homogeneity. What will happen if we allow the jump distribution to depend on the current location? This means in particular that $\mu(x) := \E [S_{n+1} - S_n | S_n = x]$ becomes a function of the current position $x \in \R^d$. First we should just consider the case that $\mu(x)=\mu$ is a constant (vector) not depending on $x$. Again if this constant (vector) is not zero (vector), then we will still have the trivial case that the walk will be transient for any dimensions. The interesting case is if we have zero drift. Is this condition enough to determine the recurrence classification? Are we able to draw similar conclusions as the Chung-Fuchs theorem?

For one dimension, the answer is already quite complicated. See the discussion in \cite[p.50]{MPW}. A zero drift non-homogeneous random walk must be recurrent on $\RP$, but not on $\R$. Details and a counter example, which is a particular case of Kemperman's oscillating random walk\cite{JHBK}, can be found in \cite{RF}. The increment law is one of two distributions (with mean zero but finite variance) depending on the walk's present sign. In contrast, for a spatially homogeneous random walk on $\R$, the zero drift condition does imply recurrence, see \cite[Chapter 9]{OK}.

In higher dimensions, the situation is even more subtle. Either recurrent or transient behaviour is possible even for walks with uniformly bounded increments. As a result we quote the following Theorem, as in \cite[Theorem 1.5.3]{MPW}.

\begin{theorem}
There exist non-homogeneous random walks with uniformly bounded jumps and $\mu(x) =0$ for all $x \in \R^d$ that are
\begin{itemize}
\item transient in $d = 2$;
\item recurrent in $d \ge 3$.
\end{itemize}
\end{theorem}

A recent paper in 2015 \cite{GMMW} gave some examples with elliptical random walks related to this theorem. They showed that the key property for the recurrence classification is the \emph{increment covariance}. It can be shown that if the increment covariance is \emph{fixed} throughout space, then one recovers the same conclusion as the Chung-Fuchs theorem (recurrence if and only if $d \le 2$), see Thm 1.5.4 in \cite{MPW}.

Here are some examples of the non-homogeneous elliptic random walks.

\newpage

\begin{figure}[H]
\color{blue}
\vspace{-10mm}
\begin{center}
\includegraphics[width=12cm, angle=0, clip = true]{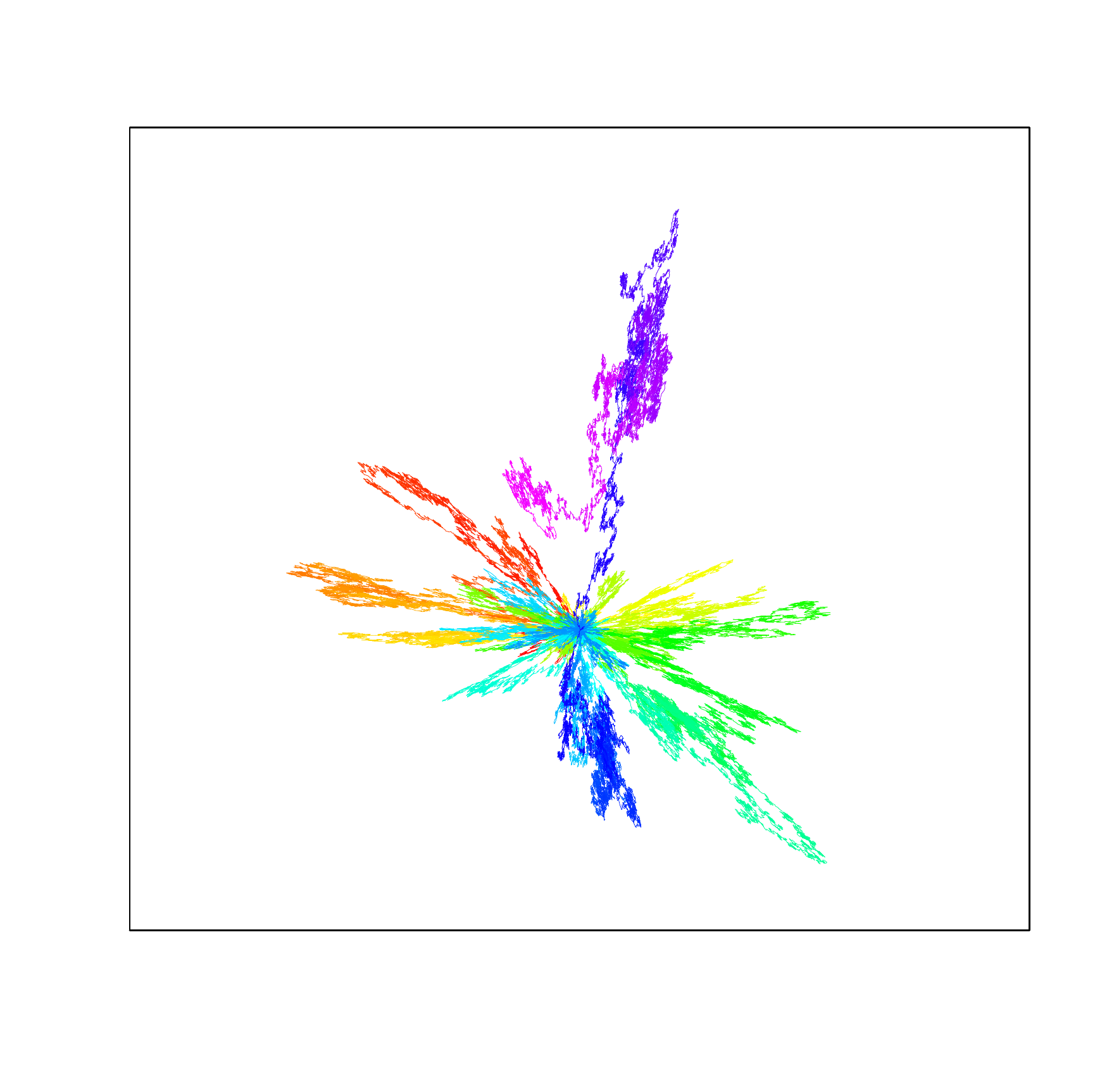}
\vspace{-15mm}
\caption{A 2D elliptic random walk with comparatively large radial component.}
\label{fig_RWSE2R}
\end{center}
\end{figure}
\vspace{-20mm}
\begin{figure}[H]
\color{blue}
\begin{center}
\includegraphics[width=12cm, angle=0, clip = true]{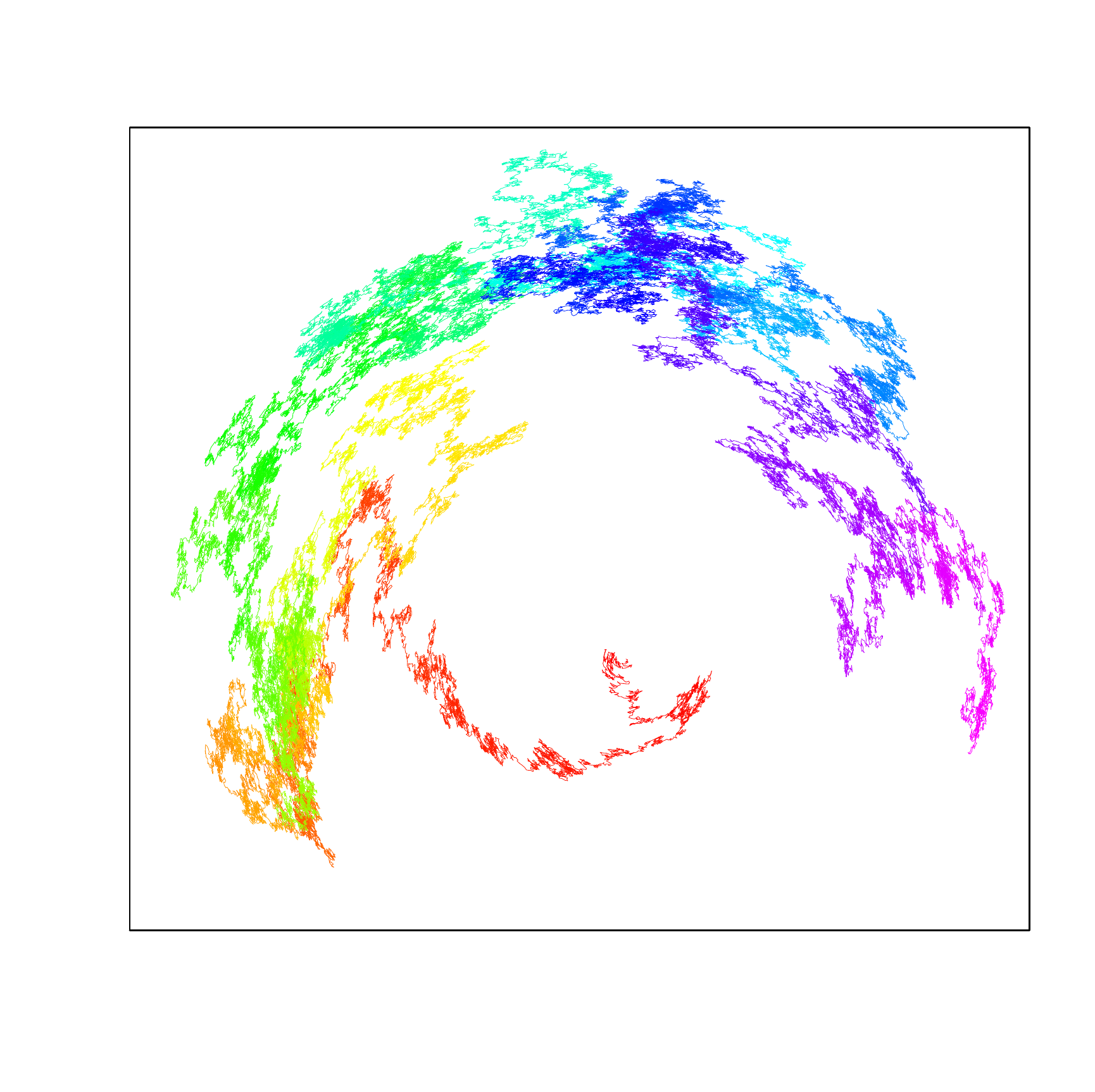}
\vspace{-20mm}
\caption{A 2D elliptic random walk with comparatively large transversal component.}
\label{fig_RWSE2T}
\end{center}
\end{figure}

The general classification for non-homogeneous random walk in $\R^d$ is a long standing open problem. Despite this fact, we are going to present you a full classification on a specially structured state space in Part I of this thesis. 

\section{Law of large numbers and central limit theorem}

In this section, we will state the classical results of the law of large numbers and the central limit theorem for homogeneous random walk. This will provide us with a rough idea of how the walk behaves in long term. 

In the past, these limit theorems started with the form of a `law of averages'. It first appeared in a theorem of Bernoulli \cite{JB} on the sums of binary random variables, but it was only stated in 1713 over a century after comments of Cardano in his work on dice games \cite{GC}. Fifty years later, Halley's treatise of mortality rates \cite{EH} clearly expressed a knowledge of decreasing errors in large samples. The term `law of large numbers' itself wasn't coined until one of Poisson's late works on probability theory in 1837 \cite{JN2}, in which the sum of Bernoulli random variables with varying probabilities of success was shown to converge to the sum of the probabilities; the theorem was only rigorously proved
by Chebyshev in 1867 \cite{PLC}.

The first description of a law for more general random variables was produced in 1929 by Kinchin \cite{AYK} and this became the weak law of large numbers. In the succeeding couple of years, Kolmogorov \cite{ANK} improved the result to establish the well known strong law, which we will present shortly after in this section.

Now we should formally define the random walk that we are considering and set up the assumptions.
\begin{description}
\item
[\namedlabel{ass:walk}{\textbf{W}}]
Let $d \in \N$, and suppose that $Z, Z_1, Z_2, \ldots$ are i.i.d.~random variables with $\E \|Z\| < \infty$ and $\E Z=\mu \in \R^d$. 
The random walk $(S_n,n\in \ZP)$ is the sequence of partial sums $S_n := \sum_{i=1}^n Z_i$ with $S_0 := 0$.
\end{description}
The first moment condition is not required in the setting of a general homogeneous walk, but it is necessary for our law of large numbers and central limit theorem to hold. 
Here is our formal statement for the law of large numbers.
\begin{theorem}[Law of large numbers of a random walk]
Suppose that (\ref{ass:walk}) holds, then
\begin{align}
\label{eqn:lln}
\frac{1}{n} (S_n - n \mu) \toas 0, \text{ as } n \to \infty.
\end{align} 
\end{theorem}
The symbol $\toas$ stands for almost sure convergence. The proof of this theorem can be found in \cite[p.73, Theorem 2.4.1]{RD}, which follows the classical lines of Etemadi's proof in 1981 \cite{NE}. More background material can be found in \cite{WW}.

To have more control of the walk, in addition to~\eqref{ass:walk}, we will sometimes assume the following:
\begin{description}
\item
[\namedlabel{ass:Sigma}{\textbf{V}}]
Suppose that $\E [\|\xi\|^2] < \infty$. We write $\Sigma := \E [(\xi-\mu)(\xi-\mu)^\tra]$ and $\sigma^2~:=~\trace~\Sigma~=~\E [ \| \xi~-~\mu \|^2$, where $\Sigma$ is a nonnegative-definite, symmetric $d$~by~$d$ matrix. 
\end{description}
Again, we may not always have this for the general setting, but have to assume this for the central limit theorem. Now we are ready for another classical result, the Lindeberg--L\'{e}vy central limit theorem: 
\begin{theorem}[Central limit theorem of a random walk]
Suppose that (\ref{ass:walk}) and (\ref{ass:Sigma}) hold; then 
\begin{align}
\label{eqn:clt}
\frac{1}{\sqrt{n}} (S_n - n \mu) \tod \calN_d( 0, \Sigma ), \text{ as } n \to \infty,
\end{align}
where $\calN_d( 0, \Sigma )$ is the $d$-dimensional normal distribution with mean $0$ and covariance matrix $\Sigma$.
\end{theorem}
Again, this theorem is an adaptation from \cite[p.124, Theorem 3.4.1]{RD}, and the proof can be found therein.

\section{Thesis outline}

The essence of this thesis consists of three directions of generalization of the classical theory, namely spatial non-homogeneity, structured state space, and derived processes.

First, a considerable amount of literature including books such as \cite{BDH,LL,PR,PRW} is devoted to spatially homogeneous random walks. The spatial homogeneity provides a well behaved model to first consider a difficult problem. However, it restricts the random movement of the particle to be the same in any location in the space, which is often not very realistic due to the underlying environment. This suggests us to study \emph{non-homogeneous random walks}. Compared to the homogeneous random walks, non-homogeneous random walks provide a better understanding of phase transitions and near-critical behaviour. See \cite{MPW} for a systematic account of non-homogeneous random walks on $\R^d$.

Second, random walks on the standard multidimensional integer lattice are common in the literature. Motivated by certain applications (see Section 2.1), it is also of interest to consider state spaces with additional structure. We include the strip and half strip models, and a generalization of the lattice distribution, in the first and second part of the thesis respectively.

Third, of interest is not only the random walk, but certain other processes derived from the random walk. For example, the Wiener process, also known as the standard Brownian motion, is a limit of random walk. It further expands the universe of random walk to various continuous models including the study of eternal inflation in physical cosmology and the Black-Scholes option pricing model in the mathematical theory of finance \cite{HK}. Although Brownian motion has been extensively studied, other simple derived processes remain hidden in the literature as they are very difficult to understand and investigate. 

It is a very difficult task to implement all these three new ideas into one model of random walk. Non-homogeneous walks and some derived processes from random walk are quite rarely investigated due to their complexity and difficulty in the treatment of the mathematical structure. 

Instead, with these ideas in mind, we hand picked two interesting models in the two main parts of the thesis. The first part will focus on the half strip model. This model consists of the first two elements of generalization of the classical theory. First, instead of the traditional state space on $\Z^d$, we considered a Markov chain on a specially structured state space. This state space gives more useful structure to the model, in particular to apply in certain specific applications, which are impossible to analyse with the traditional state space. Second, instead of restricting the walk to be spatially homogeneous, we allow the walk to be more flexible and only require the walk to converge to a (different) drift on each line. This suggests an extremely general model, to the extent that it is usually more general than all of the situations that most of the applications would need to apply to. Our analysis of the recurrence classification is complete with any sensible parameters for the applications we considered.

The second model is on the centre of mass of homogeneous random walk. It is a simple derived process of the random walk by taking the average of the sum over its past trajectory. Despite the simplicity of the model, almost no literature can be found concerning this process except one in the very special case of simple symmetric random walk.

The material in this thesis is aimed to be as self-contained as possible. After this chapter on general introduction and some basics of random walk theory, this thesis will divided into two parts for two different problems. The first part is about a model with non-homogeneous random walks on an unusual state space called the half strip. Our main focus of this part will be the recurrence classification around the critical region of phase change, and the moment existence or non existence problems of the model, which quantify the degree of recurrence. Our first group of main results includes a complete classification depending on various parameters including the drift and variability of each line, the interactions between the lines, and the probability to change or stay on the same line. The second group of main results provides the necessary and sufficient conditions for the moment existence or non existence depending on the same set of parameters. 

The second part of the thesis is about the centre of mass process of the random walk in $d$-dimensions. We want to investigate the change of the recurrence property when we increase the dimensions. The main results include a local limit theorem, which help us to prove that the process is transient for dimension $2$ or higher. Explicitly, we show that the centre of mass process has diffusive rate of escape in the transient case. On the other hand, we proved that the process is recurrent in one dimension. We also give a class of random walks with symmetric heavy-tailed increments for which the centre of mass process is transient in one dimension.


A journey of thoughts starts here.



%
%
%
%
%

\chapter*{Part I: \\ Non-Homogeneous Walks on a Half Strip}
\addcontentsline{toc}{chapter}{Part I: Non-Homogeneous Walks on a Half Strip}

\chapter{Notation, preliminaries and prerequisites}

\section{Literature review}

Markov processes $(X_n, \eta_n)$ on structured state-spaces $\Sigma$ contained in $\mathbb{X} \times S$ are of interest in many applications. In this part of the thesis, we are interested in the case where $X_n \in \mathbb{X} = \RP$ and $\eta_n \in S$  a finite set, in which case $\Sigma$ is a \emph{half strip}. Motivating applications include
\begin{itemize}
\item modulated queues~\cite{MN}, where $X_n$ represents the queue length and $\eta_n$ tracks the state of a service regime
or buffer;
\item regime-switching processes in mathematical finance, where $\eta_n$ tracks a state of the market;
\item physical processes with internal degrees of freedom~\cite{AK}, where $\eta_n$ tracks internal momentum states of a particle.
\end{itemize}

\vspace{5mm}

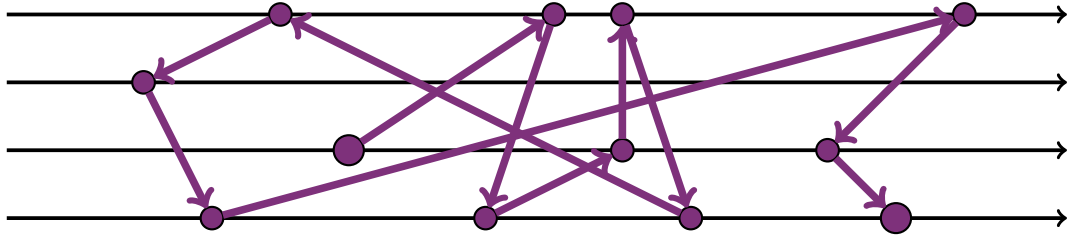
\begin{figure}
\begin{center}

\begin{tikzpicture}[scale=.9,thick]
\path[use as bounding box] (0,-2) rectangle (15,4);
\draw[->,line width=0.5mm] (0,0) to (15.5,0);
\draw[->,line width=0.5mm] (0,1) to (15.5,1);
\draw[->,line width=0.5mm] (0,2) to (15.5,2);
\draw[->,line width=0.5mm] (0,3) to (15.5,3);

\node[circle,fill=DPurple,draw,outer sep=0,inner sep=4] (X0) at (5,1) {};
\node[circle,fill=DPurple,draw,outer sep=0,inner sep=3] (X1) at (8,3) {};
\node[circle,fill=DPurple,draw,outer sep=0,inner sep=3] (X2) at (7,0) {};
\node[circle,fill=DPurple,draw,outer sep=0,inner sep=3] (X3) at (9,1) {};
\node[circle,fill=DPurple,draw,outer sep=0,inner sep=3] (X4) at (9,3) {};
\node[circle,fill=DPurple,draw,outer sep=0,inner sep=3] (X5) at (10,0) {};
\node[circle,fill=DPurple,draw,outer sep=0,inner sep=3] (X6) at (4,3) {};
\node[circle,fill=DPurple,draw,outer sep=0,inner sep=3] (X7) at (2,2) {};
\node[circle,fill=DPurple,draw,outer sep=0,inner sep=3] (X8) at (3,0) {};
\node[circle,fill=DPurple,draw,outer sep=0,inner sep=3] (X9) at (14,3) {};
\node[circle,fill=DPurple,draw,outer sep=0,inner sep=3] (X10) at (12,1) {};
\node[circle,fill=DPurple,draw,outer sep=0,inner sep=4] (X11) at (13,0) {};

\draw[->,color=DPurple,line width=1mm] (X0) --(X1);
\draw[->,color=DPurple,line width=1mm] (X1) --(X2);
\draw[->,color=DPurple,line width=1mm] (X2) --(X3);
\draw[->,color=DPurple,line width=1mm] (X3) --(X4);
\draw[->,color=DPurple,line width=1mm] (X4) --(X5);
\draw[->,color=DPurple,line width=1mm] (X5) --(X6);
\draw[->,color=DPurple,line width=1mm] (X6) --(X7);
\draw[->,color=DPurple,line width=1mm] (X7) --(X8);
\draw[->,color=DPurple,line width=1mm] (X8) --(X9);
\draw[->,color=DPurple,line width=1mm] (X9) --(X10);
\draw[->,color=DPurple,line width=1mm] (X10) --(X11);

\end{tikzpicture}
\vspace{-10mm}
\caption{An illustration of the half strip model.}
\label{fig_0}
\end{center}
\end{figure}

In much of the literature, $\eta_n$ is itself a Markov chain; in this case $(X_n, \eta_n)$ is known as a \emph{Markov-modulated Markov chain} or a \emph{Markov random walk}~\cite{GA, hp}; in the contexts of strips, study of these models goes back to Malyshev~\cite{VM}. The case where $\eta_n$ is Markov also includes processes that can be represented as \emph{additive functionals of Markov chains} \cite{rogers}. Such models pose a variety of mathematical questions, which have been studied rather deeply over several decades using various techniques that take advantage of the additional Markov structure, and much is now known. 
 
Much less is known when $\eta_n$ is \emph{not} Markov. In this part of the thesis, following~\cite{IF,AW}, we are interested in the case where $\eta_n$ is not Markov but is, roughly speaking, approximately Markov when $X_n$ is large, with stationary distribution $\pi_i$ on $S$. This relaxation is necessary to probe more intimately the recurrence/transience phase transition for these models. If $\mu_i(x)$ is the mean drift of the $\RP$-coordinate of the process at $(x,i) \in \Sigma$, then crucial to the asymptotic behaviour of the process are the asymptotics of the $\mu_i$ in comparison to the $\pi_i$. If $\mu_i (x) \to d_i \in \R$ for each $i \in S$, then the process is transient if $\sum_i \pi_i d_i > 0$ and positive recurrent if $\sum_i \pi_i d_i < 0$ \cite{IF,AW}. The critical case $\sum_i \pi_i d_i  =0$ is more subtle, and to investigate the
recurrence/transience phase transition it is natural, by analogy with classical work of Lamperti on $\RP$~\cite{L1,L2}, to study the case where $\sum_i \pi_i \mu_i (x)  = O (1/x)$. In particular, the law of the increments is \emph{non-homogeneous} in $X_n$, which typically precludes  $\eta_n$ from being Markovian, but admits our weaker conditions. 

The \emph{Lamperti drift} case in which every line has $\mu_i(x) = O(1/x)$ was studied in~\cite{AW}, and we will state the results in Section 3.1, with some new techniques to prove the results. The main focus of this part of the thesis is the \emph{generalized Lamperti drift} case where $\mu_i (x) = d_i + O (1/x)$ with $\sum_{i \in S} \pi_i d_i = 0$.

We obtain a recurrence classification for the generalized Lamperti drift case, and in the recurrent case we obtain results on existence and non-existence of passage-time moments, quantifying the recurrence. We obtain these results by use of a transformation of the process into one with \emph{Lamperti drift}, and so we establish new results on existence and non-existence of passage-time moments in that setting first. Our method is different from that of~\cite{AW}, which relied on an analysis of an embedded Markov chain, in that we make use of some Lyapunov functions for the half-strip model.

\section{The state space $\Sigma$}

Let us start with the traditional model in the literature first. We define $(X_n,\eta_n)$ as a time-homogeneous \emph{irreducible} Markov chain on $\ZP \times S$. We need the irreducibility here because we want to keep the recurrence classification as a class property for the whole problem rather than a property in some states. Although all the results in this part of the thesis will be applicable to this model, we would like to first do some modification of the state space. There are technical reasons for this change that we will explain later in Section 3.2, see Remark 3.2.5(a). However, we should now provide some intuition why we should make such a change.

Originally, the Markov chain $(X_n, \eta_n)$ is on $\ZP \times S$. This is very restrictive in terms of the mean drift that you can get from this model. Later in this part, we would like to have a more general non-homogeneous drift. If we stick with this model, then we can only assign a complicated probability on each point in order to achieve the right drift, rather than having the flexibility to assign a simple probability for a point with non-integer horizontal coordinate. In reality it is very tricky to achieve the drift we want: one must carefully pick all those integer-valued jumps to obtain such a subtle drift. This is the reason we want to extend the state space from $\ZP \times S$ to $\Sigma$, as the following,

\begin{itemize}
\item $\Sigma$ is a locally finite subset of $\RP \times S$, where $\RP$ is the set of positive real numbers and $S$ is a finite and non-empty set.
\item $\Lambda_k := \{x \in \RP: (x,k) \in \Sigma\}$. 
\item $\Lambda := \bigcup_{k \in S} \Lambda_k$. 
\item $S_x := \{ i \in S: (x,i) \in \Sigma\}$.
\end{itemize}

In here, we call $\Lambda_k$ a line, where $k \in S$ and also $\Lambda$ as the projection of $\Sigma$. $S_x$ stores the information of which line has an accessible state that can project to $\Lambda$ at a certain horizontal reference point $x$.

We need to assume $\Lambda_k$ \emph{unbounded} for each $k \in S$ to make sure that the model is allowed to go to infinity, i.e. be transient, on any line whenever possible to preserve the structure of the model, so that the classification make sense. 

Recall that $\Sigma$ being a locally finite subset of $\RP \times S$ means that for any $c \in \RP$, $\Sigma \cap ([0,c] \times S)$ has finite number of points. Notice here the locally finite property is inherited by each line from the state space.

The local finiteness condition is to ensure that $\Sigma$ has no finite limit points, so that if $(X_n, \eta_n)$ is transient, then $X_n \to \infty$. Consider the following example when the local finiteness condition is not satisfied. First we define the state space to be
\begin{equation*}
\Sigma=\left(\ZP \cup \left\{\frac{k}{k+1}:k \in \Z\right\} \right) \times \{1\}.
\end{equation*}
Then we assign the transition probabilities as follows,
\begin{itemize}
\item $\Pr(X_{n+1}=\frac{k+1}{k+2}|X_n=\frac{k}{k+1})=p$, $\Pr(X_{n+1}=\frac{k-1}{k}|X_n=\frac{k}{k+1})=1-p$ for all $k \in \ZP$, 
\item $\Pr(X_{n+1}=k-1|X_n=k)=\Pr(X_{n+1}=k+1|X_n=k)=\frac{1}{2}$ for all $k \in \ZP$,
\item $\Pr(X_{n+1}=\frac{1}{2}|X_n=0)=\Pr(X_{n+1}=1|X_n=0)=\frac{1}{2}$.
\end{itemize}
When $p$ is close to $1$, we can see that whenever the walk goes into the state $0$, it has half probability to go to state $\frac{1}{2}$, and then the process has very high tendency not to go back to $0$ and keep on increasing, while it does not go to infinity as it would not be greater than $1$.

From now we extend and replace the definition of half strips or semi-infinite strips from the state space $\ZP \times S$ to $\Sigma$ unless otherwise specified. Here is our model formally.

\begin{description}
\item
[\namedlabel{ass:basic}{A}]
Suppose that $(X_n , \eta_n)$, $n \in \ZP$, is a time-homogeneous, irreducible Markov chain on $\Sigma$,
a
locally finite subset of  $\RP \times S$. Suppose that for each $k \in S$ the line $\Lambda_k$ is unbounded.
\end{description}

Notice that all the results in this part are also applicable to the more restricted state space $\ZP \times S$.

\section{Recurrence classification for the half strip}

As described earlier, one of the most important properties to understand for a random walk or a Markov chain is the recurrence classification. Intuitively, as we saw in the introduction, recurrent means that the random walk will always come back to any state in long-run, while transient means the random walk will to go to infinity in some direction and never come back. Some thought is required to see how this applies to the present state space. First, in the vertical direction, $S$ is finite and thus the walk cannot actually escape in this direction. On the other hand, in the horizontal direction $\RP$, the process cannot escape to the left, but only to the right side. It can escape via any line due to the fact that $\Lambda_k$ is unbounded for all $k \in S$ when we set up the model. Here is the formal definition for our half strip model.


\begin{lemma}
\label{lem:Class3}
Let $(X_n,\eta_n)$ be a time-homogeneous irreducible Markov chain on the state-space $\Sigma$. Exactly one of the following holds:
\begin{enumerate}
\item[(i)] If $(X_n,\eta_n)$ is recurrent, then $\Pr [X_n=x  \textit{ } i.o.]=1$ for any $x \in \Lambda$.
\item[(ii)] If $(X_n,\eta_n)$ is transient, then $\Pr [X_n=x  \textit{ } i.o.]=0$ for any $x \in \Lambda$, and $X_n \to \infty \textit{ } a.s.$
\end{enumerate}
In the former case, we call $(X_n)$ recurrent, and in the latter case, we call $(X_n)$ transient. 
\end{lemma}

Notice that the process $(X_n)$ is not a Markov chain so this is different from our usual definition. This is a lemma but not a definition because it is not trivial that the dichotomy of recurrence and transience holds, i.e. the probability must be 0 or 1 rather than other values. Now we are going to prove Lemma \ref{lem:Class3}. 

\begin{proof}
As $(X_n,\eta_n)$ is an irreducible Markov chain, the states of $(X_n,\eta_n)$ are either all recurrent or all transient. In the former case, for any $x \in \Lambda$, where $\Lambda=\bigcup_k \Lambda_k$, we have $x \in \Lambda_k$ for some $k \in S$. Then we get $(x,k) \in \Sigma$. That $(X_n,\eta_n)$ is recurrent means $(X_n,\eta_n)=(x,k)$ \textit{ i.o. a.s.}, thus we have $X_n=x$ \textit{ i.o. a.s.} This gives $\Pr(X_n=x \textit{ }  i.o.)=1$.

On the other hand, if $(X_n,\eta_n)$ is transient, for any $x \in \Lambda$, $(X_n,\eta_n)=(x,k)$ only \textit{f.o.} for any $k$ such that $(x,k) \in \Sigma$. Summing over $k$, of which there are finitely many, we have $X_n=x$ only \textit{f.o.} So we have $\Pr(X_n=x \textit{ }  i.o.)=0$. 

This implies $X_n \in R$ \textit{f.o.} for any finite non-empty set $R \in \Lambda$. As $\Sigma$ is locally finite, we know $\Lambda_k$ is also locally finite. With the knowledge that $S$ is finite, we get that $\Lambda$ is locally finite. For any $L \in \ZP$, denote $R_L=\Lambda \cap [0,L]$, which is finite and non-empty for $L$ large enough. Summing over $X_n=i$ \textit{f.o.} for $i \in R_L$, we have $X_n \in R_L$ \textit{f.o.} as $R_L$ is finite. Hence we have $\liminf_{n \to \infty}X_n \ge L$. As $L$ was arbitrary, we conclude that $\liminf_{n \to \infty}X_n = \infty $. So we have $\lim_{n \to \infty}X_n = \infty$. 
\end{proof}

As in the usual random walk, recurrence in the half strip can be further classified as null recurrence or positive recurrence. Again, we have to properly define these concepts due to the complication of the state space. Intuitively, null recurrence means the expected time of return to any point is infinite while it is finite if the random walk is positive recurrent. We also define \emph{null} to be null recurrent or transient. Here are the formal definitions.

\begin{lemma}
\label{lem:Class4}
Let $(X_n,\eta_n)$ be a time-homogeneous irreducible Markov chain on the state-space $\Sigma$. There exists a unique measure $\nu : \Lambda \to \RP$ such that
\[
\lim_{n \to \infty} \frac{1}{n} \sum_{k=0}^{n-1} \1{X_k=x} = \nu (x), \quad a.s.
\]
Exactly one of the following holds.
\begin{enumerate}
\item[(i)] If $(X_n,\eta_n)$ is null, then $\nu (x) =0$ for all $x \in \Lambda$.
\item[(ii)] If $(X_n,\eta_n)$ is positive recurrent, then $\nu (x) >0$ for all $x \in \Lambda$ and $\sum_{x \in \Lambda}~\nu(x)~=~1$.
\end{enumerate}
If $X_n$ is recurrent, then we say that it is null recurrent if (i) holds and positive recurrent if (ii) holds.
\end{lemma}
This is again a lemma because it is not trivial that the case that $\nu (x) =0$ for some $x$ and $\nu (x) >0$ for some other $x$ would not happen. The proof relies on careful separation of the two coordinates of the state space.

\begin{proof}
By standard Markov chain theory, e.g. \cite{JN}, P.35, Theorem 1.7.5 and 1.7.6, there exists a (unique) measure $\phi (x,i):\Sigma \to \RP$ such that
\[
\lim_{n \to \infty} \frac{1}{n} \sum_{k=0}^{n-1} \1{X_k=x,\eta_k=i} = \phi (x,i), \quad a.s.
\]
Define $\nu(x)$ as the projection of $\phi (x,i)$ on the second component, i.e.
\[
\nu(x)=\sum_{i \in S_x}\phi (x,i)
\]
for any $x \in \Lambda$. Then we get, a.s.,
\begin{align*}
\nu (x) =& \sum_{i \in S_x} \lim_{n \to \infty} \frac{1}{n} \sum_{k=0}^{n-1} \1{X_k=x, \eta_k=i} \\
=& \lim_{n \to \infty} \frac{1}{n} \sum_{k=0}^{n-1} \sum_{i \in S_x} \1{X_k=x, \eta_k=i} \\
=& \lim_{n \to \infty} \frac{1}{n} \sum_{k=0}^{n-1} \1{X_k=x}. 
\end{align*}

It is very important to notice that the sum for $i$ here is finite so that it can be taken out of the other sum and limit without causing any extra problem. The set $S_x$ is also non-empty because given the fact that $x \in \Lambda$, there exist some $i \in S$ such that $(x,i) \in \Sigma$. So the set $S_x \ne \emptyset$ for $x \in \Lambda$.

Now when $(X_n,\eta_n)$ is null, then $\phi (x,i)=0$ for all $(x,i) \in \Sigma$, so $\nu (x)=\sum_{i \in S_x} \phi (x,i)=0$, always bearing in mind that we are doing a finite sum.

For $(X_n,\eta_n)$ positive recurrent, $\phi (x,i)>0$ for all $(x,i) \in \Sigma$ and hence $\nu(x)>0$ since as $\nu (x)=\sum_{i \in S_x} \phi (x,i)>0$ and the sum is not empty. With the fact that $\sum_{(x,i) \in \Sigma} \phi (x,i)=1$, we can separate the sum across the two coordinates and get $\sum_{x \in \Lambda} \sum_{i \in S_x} \phi (x,i)=1$. This is the same as saying $\sum_{x \in \Lambda}\nu(x)=1$. Hence all of the claims in the lemma are proved.
\end{proof}

\section{Assumptions of the model}
\label{s:hscd}

To solve our recurrence classification problem, we also need the following technical assumptions. First, to be realistic, we first need to assume the displacement of the $X$-coordinate has bounded $p$-moments for some $p< \infty$. This is a crucial but weak assumption because without this, there will be no control of the size of jumps. We do not want the walk have an increasing size of boundless jumps when it is at the position far on the right side. In this bad behaviour the walk can suddenly jump back to the far left or have a very big jump on the right in one step, so that all the steps that the walk had before are negligible. So we would like to impose this uniform bound for the walk to get some regularity to predict the long term behaviour. 

\begin{description}
\item
[\namedlabel{ass:p-moments}{B$_\textit{p}$}]
There exists a constant $C_p< \infty$ such that for all $n \in \ZP$,
\[
\E [|X_{n+1}-X_n|^p \mid X_n = x, \, \eta_n = i ] \le C_p , \text{ for all }  (x,i) \in \Sigma.
\]  
\end{description}

We will need $p>2$ most of the time in this part of the thesis, which we sometimes refer to as demanding that `two moments exist'. However, for some of the results, $p>1$, i.e. `one moment exists' is already sufficient.

We define $p(x,i,y,j)$ as the transition probabilities of our irreducible Markov chain $(X_n,\eta_n) \in \Sigma$, i.e.
\[
\Pr[(X_{n+1},\eta_{n+1})=(y,j) \mid (X_n,\eta_n)=(x,i)] =p(x,i,y,j).
\] 
For the sake of reasonable behaviour of the probabilities so that we can have the unique stationary distribution \mbox{\boldmath${\pi}$} from the embedded process in the vertical, i.e. $\eta$, direction, we need to assume that $\eta_n$ is approximately Markov when $X_n$ is large. First, we define
\begin{equation}
q_{ij}(x)=\sum_{y \in \Lambda_j}p(x,i,y,j)
\end{equation}
as we do not need the information of the exact point that the walk is jumping to, but only which line it jumps to and which point it starts from. Here is our assumption:
\pagebreak

\begin{description}
\item
[\namedlabel{ass:q-lim}{Q$_\infty$}]
Suppose that $\lim_{x \to \infty} q_{ij}(x)=q_{ij}$ exists for all $i,j \in S$, and $(q_{ij})$ is an irreducible stochastic matrix. 
\end{description}

Now if we assume $(Q_\infty)$, then we can define a new process $(\eta^*_n)$, $n \in \ZP$, as a Markov chain on $S$ with transition probabilities given by $q_{ij}$. As $(\eta^*_n)$ is irreducible and finite, we know that there exists a \emph{unique} stationary distribution \mbox{\boldmath${\pi}$}$=(\pi_1,\pi_2, \ldots, \pi_{|S|})^\top$ on $S$ with $\pi_j>0$ for all $j \in S$ and satisfying $\pi_j=\sum_{i \in S}\pi_i q_{ij}$ for all $j \in S$. $(Q_\infty)$ is very important here because if \mbox{\boldmath${\pi}$} does not exist, then we cannot define the total average drift of the whole system, which determines the recurrence classification. 

Naturally, we want to specify the movement of the chain by the one-step mean (horizontal) drift at each point on each line, i.e., its first moment in the $X$-coordinate on line $i$. This is:
\[
\mu_i(x) := \E[X_{n+1}-X_n \mid X_n=x, \eta_n=i]=\sum_{j \in S}\sum_{y \in \Lambda_j}(y-x) p(x,i,y,j);
\]
notice that $\mu_i(x)$ is finite if $(B_p)$ holds for some $p \ge 1$. In the simplest case, we suppose that each line has an asymptotically \emph{constant drift}, and we assume

\begin{description}
\item
[\namedlabel{ass:drift-constant}{D$_\text{C}$}]
For each $i \in S$ there exists $d_i \in \R$ such that $\mu_i(x) = d_i + o(1)$ as $x \to \infty$.
\end{description}

Although this is called the constant drift, from the term $o(1)$ we actually allow $\mu_i(x)$ to fluctuate around the constants, as long as the fluctuation converges to zero when $x \to \infty$. In some sense, only the behaviour when $x$ is big matters.

Instead of stating the original theorems by Malyshev\cite{VM} or Falin \cite{IF}, we shall state a slightly generalised and polished result in a paper of Georgiou and Wade \cite{AW}, for the model that we are using now. 

\begin{theorem}[Georgiou, Wade, 2014, amended]
\label{t:drift-constant}
Suppose that~\eqref{ass:basic} holds, and that~\eqref{ass:p-moments} holds for some $p>1$. Suppose also that~\eqref{ass:q-lim} and~\eqref{ass:drift-constant} hold.
Then the following classification applies.
\begin{itemize}
\item If $\sum_{i \in S}d_i\pi_i>0$, then $(X_n,\eta_n)$ is transient.
\item If $\sum_{i \in S}d_i\pi_i<0$, then $(X_n,\eta_n)$ is positive recurrent.
\end{itemize}
\end{theorem}

Theorem~\ref{t:drift-constant} is a minor generalization of Theorem~2.4 of~\cite{AW}, which took $\Sigma = \ZP \times S$; the proof there readily extends to the statement here. We give an alternative proof, using Lyapunov functions, in Section 4.4. Earlier versions of the result, which had the extra assumption that $q_{ij}(x)=q_{ij}$ not depending on $x$, are Theorem~3.1.2 of~\cite{MM} and the results of~\cite{IF}. The proof in \cite{AW} is based on the investigation of the embedded process $(Y_n)$, which records the $X$-coordinate of the chain when it returns to a given line. They use increment moment estimates together with some Foster-Lamperti conditions to classify the process $(Y_n)$, and then deduce the classification for $(X_n)$ from the equivalence results. 

Intuitively, $\sum_{i \in S}d_i\pi_i$ stands for the total average drift of the system, as it is summing over all lines with the average drift on each line multiplied by the proportion of time spent on the line. So if the total average drift is positive, the walk has the tendency to go to the right on average, thus it is difficult for the process to return to the points on the left in long term, and the walk is transient. On the other hand, if we have a negative total average drift, then the walk will have the tendency to go to the left, and keep coming back to the left boundary, thus the walk is (positive) recurrent. 

As you can see, Theorem~\ref{t:drift-constant} has nothing to say about the much more subtle case where $\sum_{i \in S} d_i \pi_i = 0$. One natural guess would just be null recurrence whenever the condition is satisfied but this is not always true.  In fact, the model can fall into any classification, i.e., it can be positive recurrent, null recurrent or transient. Here further assumptions are required to reach any conclusion.

One way to achieve $\sum_{i \in S} d_i \pi_i = 0$ is to have $d_i = 0$ for all $i \in S$. In this case, by analogy with the classical one-dimensional work of Lamperti~\cite{L1,L2}, the natural setting in which to probe the recurrence-transience phase transition is that of \emph{Lamperti drift}, as studied in~\cite{AW}, which we present in Section~\ref{s:ld}. In this setting we give new results on existence and non-existence of moments of passage times.

The second possibility and the most subtle case, in which $d_i \ne 0$ for some $i \in S$ but nevertheless $\sum_{i \in S} d_i \pi_i = 0$, leads to what we call \emph{generalized Lamperti drift}, which is the main focus of this part of the thesis and is presented in Section~\ref{s:gld}. Here we establish a recurrence classification as well as results on passage-time moments.

The proof of the theorems introduced in these sections will be delayed until Chapter \ref{ch:pfhs}, after we introduce various techniques related to Lyapunov functions method, martingale theory and some well known linear algebra results.

\section{The Lamperti problem}
\label{s:lp}

For the first step to probe the recurrence classification for the Lamperti drift case in our half strip problem,  we should recall the origin of the name, i.e., the Lamperti problem, see \cite{MPW}, Section 1.3 and Chapter 3. 

We start again with the simple symmetric random walk $S_n$ on $\Z^d$, and start the walk at the origin. This time instead of going through the standard proof of P\'olya's recurrence theorem to get the recurrence classification, we will try a different method. First we reduce this $d$-dimensional problem into a one dimensional one by the \emph{Lyapunov function}, a transformation process given by
\begin{equation*}
X_n := \| S_n \|,
\end{equation*}
where $\| \blob \|$ is the Euclidean norm in $\R^d$. Hence $X_n$ is just the distance between the origin and the particle at time $n$. So now the stochastic process will take values in $S:=\{ \|x\| : x \in \Z^d \}$, a countable subset of the half line $\RP$. Notice that the recurrence classification property will transfer from $S_n$ to $X_n$, since $S_n=0$ if and only if $X_n=0$. However the Markov property was sacrificed for the reduction in dimensionality. One can easily observe, say in two dimension,  for the same value of $X_n$ on different positions for $S_n$ may give different distributions, thus the Markov property will not hold for $X_n$, see the example in \cite{MPW}, Section 1.3. Hence from this point, we need to have a method to find the recurrence classification of $X_n$, which does not heavily depend on the Markov property. 

This topic leads to a more general area called the \emph{Lamperti problem}, introduced by John Lamperti \cite{L1,L2} in early 1960s. Informally, let us begin with a discrete-time time-homogeneous Markov process $X_n$ with well-defined increments moment functions
\begin{equation*}
m_k(x) = \E \left[ (X_{n+1}-X_n)^k | X_n=x \right]
\end{equation*}
for all $k \ge 0$. Having a uniform bound on the increments can easily guarantee this condition, but is is not necessary. The Lamperti problem is asking if we are given the first few moments, especially the first two, $\mu_1$ and $\mu_2$, how to determine the asymptotic behaviour of $X_n$. If we indeed impose the uniform bound condition formally,
\begin{equation}
\Pr \left( |X_{n+1}-X_n| \le B \right) =1 \label{2b}
\end{equation}
for some $B \in \RP$, then we can have a slightly modified version of Lamperti's fundamental recurrence classification, see Theorem 1.3.1 of \cite{MPW}.

\begin{theorem}[Lamperti, 1960]
\label{thm:lam}
Suppose that $X_n$ is a Markov process on $S$ satisfying \eqref{2b}. Under mild conditions on irreducibility, the following recurrence classification holds. Let $\varepsilon >0$.
\begin{itemize}
\item If $2x m_1(x) + m_2(x) < -\varepsilon$, then $X_n$ is positive recurrent;
\item If $2x |m_1(x)| \le m_2(x) + O(x^{-\varepsilon})$, then $X_n$ is null recurrent;
\item If $2x m_1(x) - m_2(x) > \varepsilon$, then $X_n$ is transient;
\end{itemize}
\end{theorem}

Notice that the null recurrence classification is slightly sharper than Lamperti's original results. This theorem states that if the absolute value of the first moment is large enough compared to the second moment in the tail (infinite side) of the walk, then the process will have enough force to go in the specific direction, left or right, depending on the sign of the drift, resulting in transience or positive recurrence. Otherwise, if the absolute value of (twice) the drift is not large enough compared to the variance, then the walk does not have enough force to go in a specific direction, as the variance dominates the effect of the drift, resulting in the null-recurrent case.

Although this version of the theorem does not directly give us the P\'olya's Theorem because of the lost of Markov property stated before, this method is still applicable by slight modification of the definition of $\mu_k$. By computing the first and second moment of $X_n$ explicitly for this simple symmetric random walk $S_n$, we get
\begin{align*}
\E\left[X_{n+1}-X_n | S_n=x \right] &= \left( \frac{d-1}{2d} \right) \frac{1}{\|x\|} + O(\|x\|^{-2}) \\ 
\E\left[(X_{n+1}-X_n)^2 | S_n=x \right] &= \frac{1}{d} + O(\|x\|^{-1}). 
\end{align*}
So the corresponding terms in the theorem will be
\begin{align*}
2xm_1(x) &= 1 - \frac{1}{d} + O(x^{-1}) \\
m_2(x) &= \frac{1}{d} + O(x^{-1}).
\end{align*}
Hence using the theorem we get $S_n$ is transient if and only if 
\begin{equation*}
1 - \frac{1}{d} >  \frac{1}{d},
\end{equation*}
which is equivalent to $d>2$. For the technical details see \cite{MPW} Section 3.5. As you can see, this is a potent way to prove the P\'olya's Theorem. With the sole and elementary computations of the increment moments of $X_n$ using Taylor's theorem, the method can generalize to a broad range of random walks, and does not require any special structure on the original process. 

Finally, back to our half strip model, if we take the special case that $S$, the vertical component of $\Sigma$ to be a singleton, it reduces back to the model in the Lamperti problem. So one might see the half strip model is actually a generalization of the Lamperti problem. One may think we can easily push the Lamperti's fundamental recurrence classification result through the half strip model. However, the real situation is much more difficult than that. There is no doubt that if all of the lines have the same classification, say transient, then the whole system of the half strip will also be transient, because no matter which line the process is on, we still have the tendency to go to infinity on the right side. However, what if some of the lines are recurrent and some of the lines are transient? Then the result is not clear, as it depends on how much time the process spends on each line and how recurrent or transient each line is. In Section \ref{s:hscd}, we gave the result when we have a non-zero total average drift, and in the next chapter we will discuss the subtle case when we have zero total average drift, starting with the special case of \emph{Lamperti drift}, and complete the classification with \emph{generalised Lamperti drift}.

\chapter{Main results}

\section{Lamperti drift on a half strip}
\label{s:ld}

\subsection{Recurrence classification}

For the remainder of this part of the thesis we introduce the following shorthand to simplify notation:
\[
\E_{x,i}[\,\blob\,] = \E[\, \blob \,\mid X_n=x, \, \eta_n=i].
\]

Continuing with our half strip model, we would like to probe the classification in the special case with zero total average drift, i.e. $\sum_{i \in S}d_i\pi_i=0$. To proceed with more complicated drifts, as in the Lamperti's fundamental recurrence classification, we need to have some control on the variance, i.e. the second moment of the increments. So we define, for $(x,i) \in \Sigma$,
\[
\sigma_i^2(x) := \E_{x,i} [ (X_{n+1}-X_n)^2 ];
\]
note that $\sigma_i^2(x)$ is finite if~\eqref{ass:p-moments} holds for some $p \ge 2$. The formal definition for the Lamperti drift case of the half strip model is as follows:
\begin{description}
\item
[\namedlabel{ass:drift-lamperti}{D$_\text{L}$}]
For each $i \in S$ there exist $c_i \in \R$ and $s^2_i \in \RP$, with at least one $s^2_i$ non-zero, such that, as $x \to \infty$,
 $\mu_i(x) = \frac{c_i}{x} + o(x^{-1})$ and $\sigma_i^2(x) = s^2_i + o(1)$.
\end{description}
The reason that we named this case the Lamperti drift is because the problem has a very similar structure and result as in the Lamperti problem. And in fact for our half strip state space $\Sigma$, if we take $S$ to be a singleton, it returns to the well-known Lamperti problem. Results in this chapter hence cover the results from Lamperti. 

In this case, comparing to $(D_C)$, we have $d_i=0$ for all $i \in S$. We specify the error in $o(1)$ can be in the natural form $\frac{c_i}{x} + o(x^{-1})$, but it is possible to impose the drift in other forms such as $\frac{c_i}{\sqrt{x}}$. The exact form of the drift does not actually affect the theory here but the calculation would be different. So for the time being we will stick with the traditional drift type coinciding with the representation in the Lamperti problem. 

To obtain results at the critical point for the phase transition we will need to strengthen the assumptions~\eqref{ass:q-lim} and~\eqref{ass:drift-lamperti} by imposing additional assumptions:
\begin{description}
\item
[\namedlabel{ass:q-lim+}{Q$_\infty^+$}]
Suppose that there exists $\delta_0 \in(0,1)$ such that $\max_{i,j \in S}|q_{ij}(x)-q_{ij}|=O(x^{-\delta_0})$ as $x \to \infty$. 
\item
[\namedlabel{ass:drift-lamperti+}{D$_\text{L}^+$}]
Suppose that there exist $\delta_1 \in (0,1)$, $c_i \in \R$, and $s^2_i \in \RP$, with at least one $s^2_i$ non-zero, such that for all $i \in S$, as $x \to \infty$,
 $\mu_i(x) = \frac{c_i}{x} + o(x^{-1-\delta_1})$ and $\sigma_i^2(x) = s^2_i + o(x^{-\delta_1})$.
\end{description}
We need these assumptions in the critical case to have slightly more control on the error terms of the transition probability and the mean and variance of the horizontal increments. In the Lamperti drift setting, we have the following recurrence classification.

\begin{theorem} 
\label{thm:GW}
Suppose that~\eqref{ass:basic} holds, and that~\eqref{ass:p-moments} holds for some $p>2$. Suppose also that~\eqref{ass:q-lim} and~\eqref{ass:drift-lamperti} hold.
Then the following classification applies.
\begin{itemize}
\item If $\sum_{i \in S}(2c_i-s_i^2)\pi_i>0$, then $(X_n,\eta_n)$ is transient.
\item If $|\sum_{i \in S}2c_i\pi_i|<\sum_{i \in S}s_i^2\pi_i$, then $(X_n,\eta_n)$ is null recurrent.
\item If $\sum_{i \in S}(2c_i+s_i^2)\pi_i<0$, then $(X_n,\eta_n)$ is positive recurrent.
\end{itemize}
If, in addition,~\eqref{ass:q-lim+} and~\eqref{ass:drift-lamperti+} hold, then the following condition also applies (yielding an exhaustive classification):
\begin{itemize}
\item If $|\sum_{i \in S}2c_i\pi_i|=\sum_{i \in S}s_i^2\pi_i$, then $(X_n,\eta_n)$ is null recurrent.
\end{itemize}
\end{theorem}

Theorem~\ref{thm:GW} is a slight generalization of Theorem~2.5 of~\cite{AW}, which took $\Sigma = \ZP \times S$. The proof in~\cite{AW}, which made use of Lamperti's~\cite{L1,L2} results applied to the embedded process obtained by observing the $X$-coordinate on each visit to a reference line, extends readily to the statement here. We give an alternative proof in Section 4.5 of the first three points in the theorem (not the critical case).

We can use similar intuition behind Theorem~\ref{thm:lam} to understand the theorem here. Instead of considering only one line, we consider the weighted average of the total drift with the weighted average of the total variance in the system, weighting on the proportion of time spent on each line. If the absolute value of the former is large enough compared to the latter, then it will give the system a strong enough push to a direction either right or left in average, depending on the sign of the drift, resulting in transience or positive recurrence accordingly. However, if the absolute value of the former is not big enough, the walk will not be able to generate enough force to overcome the second moment, thus giving the null-recurrent case.

In the next subsection, we will quantify these two forces from the first and second moment. Comparing the size of these will give us the knowledge of the degree of recurrence of the process.

\subsection{Existence and non-existence of moments}

In the case of recurrence, we can actually quantify how recurrent the process is. Instead of just having the classification of positive recurrent and null recurrent, one way to obtain quantitative information on the nature of recurrence is to study moments of \emph{passage times}. For $x \in \RP$, 
define the stopping time 
\begin{equation}
\label{eqn_tau}
\tau_x :=\min\{n \ge 0 : X_n \le x\}.
\end{equation}
In the positive-recurrent situation, we have that $\E[\tau_x]< \infty$ a.s., for all $x$ sufficiently large. In the case of null, $\E[\tau_x^s]= \infty$ a.s., for all $s \ge 1$, and $x$ sufficiently large.

First we state a result that gives conditions for $\E [ \tau_x^s]$ to be finite. 
\begin{theorem}
\label{thm:L2}
Suppose that~\eqref{ass:basic} holds, and that~\eqref{ass:p-moments} holds for some $p>2$. Suppose also that~\eqref{ass:q-lim} and~\eqref{ass:drift-lamperti} hold.
If  for some $\theta>0$,
\begin{equation}
\sum_{i \in S} \left[ 2c_i+(2\theta-1)s_i^2 \right] \pi_i <0, 
\label{eqn:L2}
\end{equation}
then for any $s \in \left[0,\theta \wedge \frac{p}{2} \right]$, we have $\E[\tau_x^s]<\infty$ 
for all $x$ sufficiently large.
\end{theorem}

We have the following result in the other direction.
\begin{theorem}
\label{thm:L3}
Suppose that~\eqref{ass:basic} holds, and that~\eqref{ass:p-moments} holds for some $p>2$. Suppose also that~\eqref{ass:q-lim} and~\eqref{ass:drift-lamperti} hold.
If for some $\theta \in (0, \frac{p}{2}]$, 
\begin{equation}
\sum_{i \in S} \left[ 2c_i+(2\theta-1)s_i^2 \right] \pi_i > 0, 
\label{eqn:L3}
\end{equation}
then for any $s \in \left[\theta , \frac{p}{2} \right]$, we have
$\E [ \tau_x^s ] =\infty$ for all $x$ sufficiently large.
\end{theorem}

In the case where $S$ is a singleton, Theorems~\ref{thm:L2} and~\ref{thm:L3} reduce to versions of Propositions~1 and~2, respectively, of~\cite{AIM} on passage-time moments for Markov chains on $\RP$.

Using these two theorems, by plugging in different values of $\theta$ in the expression $\sum_{i \in S} \left[ 2c_i+(2\theta-1)s_i^2 \right] \pi_i$, we can pinpoint which moments of the passage times exist or not. In short, if more moments exist then the process is more recurrent, and we should expect a smaller scale of time for the process to return.

We also see that if we put $\theta=1$ in Theorems~\ref{thm:L2}, we can see the moments of the passage time exists for all $s \in [0,1]$, implying that the process is positive recurrent. if we put $\theta \to 0^+$ in Theorems~\ref{thm:L3}, we can see that the moments of the passage time does not exists for all $s \in [0,\frac{p}{2}]$, implying that the process is null. (This does not directly imply transient unfortunately because some null-recurrent random walk can also have no moment exist, e.g. simple random walk on $\Z^2$.) 

Intuitively, these two theorems add an extra parameter $\theta$ in the equation, comparing to Theorem~\ref{thm:GW}, which gives some extra flexibility on how tolerant is the drift size comparing to the variance. For Theorems~\ref{thm:L2}, the stronger the restriction on $c_i$, i.e. imposing a larger $\theta$, the more moments you can get from the passage time. This means if there is a larger $\theta$ that satisfies the equation in the theorem, the process is more `recurrent' in some sense. On the opposite hand, if we impose a smaller $\theta$, giving more flexibility to $c_i$, you will get fewer moments as a result.

Theorem~\ref{thm:L3} is essentially the opposite consideration of Theorem~\ref{thm:L2}. Its use is to pinpoint the critical value of $s$ which gives you the existence-non-existence phase transition.

The proofs of Theorem \ref{thm:L2} and Theorem \ref{thm:L3} will be presented in Chapter~\ref{ch:pfhs}, with the use of some specific Lyapunov functions and some semi-martingale methods. Notice that we need to use different functions for the proofs of Theorems~\ref{thm:L2} and Theorems~\ref{thm:L3}, and there is no direct relation between them.

The next section will discuss the most subtle case when $d_i \ne 0$ for some $i \in S$ but nevertheless $\sum_{i \in S} d_i \pi_i = 0$, which is what we call the \emph{generalized Lamperti drift}.

\section{Generalized Lamperti drift on a half strip}
\label{s:gld}

\subsection{Recurrence classification}

Now we turn to the main topic of this part of the thesis. The last case is when some (or all) of the lines have non-zero constant drift, but the total average drift is zero. This case is the most subtle, as it is possible to construct some examples with the same $\mu_i(x)$ and $\sigma_i(x)$ but which fall into different classifications. We will show some explicit examples in Chapter \ref{ch:exphs}. We discovered that the asymptotic properties of the process depend not only on $\mu_i(x)$
and $\sigma_i^2(x)$ but also on
the quantities
\[
\mu_{ij}(x) :=\E_{x,i} \left[(X_{n+1}-X_n) \1 {\eta_{n+1}=j} \right] ;
\]
this alerts us to the fact that correlations between the components of the increments are now crucial.
The case of generalized Lamperti drift is the following. To avoid confusion with the Lamperti drift case, we changed the symbols for $c_i$ and $s_i$ to $e_i$ and $t_i$.
\begin{description}
\item
[\namedlabel{ass:drift-gl}{D$_\text{G}$}]
 For $i, j \in S$ there exist  $d_i \in \R$, $e_i \in \R$, $d_{ij} \in \R$ and $t^2_i \in \RP$, with at least one $t^2_i$ non-zero, such that  
\begin{itemize}
\item[(a)] for all $i \in S$, $\mu_i(x) = d_i+ \frac{e_i}{x} + o(x^{-1})$ as $x \to \infty$;
\item[(b)] for all $i \in S$, $\sigma^2_i(x) = t^2_i + o(1)$ as $x \to \infty$;
\item[(c)] for all $i,j \in S$, $\mu_{ij}(x) = d_{ij} + o(1)$ as $x \to \infty$; and
\item[(d)] $\sum_{i \in S} \pi_i d_i = 0$.
\end{itemize}
\end{description}
Note that necessarily we have the relation $d_i =\sum_{j \in S} d_{ij}$.

As in the Lamperti drift case, we need to have an additional condition at the phase boundary.
\begin{description}
\item
[\namedlabel{ass:drift-gl+}{D$_\text{G}^+$}]
There exist $\delta_2 \in (0,1)$, $d_i \in \R$, $e_i \in \R$, $d_{ij} \in \R$ and $t^2_i \in \RP$, with at least one $t^2_i$ non-zero, such that  
\begin{itemize}
\item[(a)] for all $i \in S$, $\mu_i(x) = d_i+ \frac{e_i}{x} + o(x^{-1-\delta_2})$ as $x \to \infty$;
\item[(b)] for all $i \in S$, $\sigma^2_i(x) = t^2_i + o(x^{-\delta_2})$ as $x \to \infty$; and
\item[(c)] for all $i,j \in S$, $\mu_{ij}(x) = d_{ij} + o(x^{-\delta_2})$ as $x \to \infty$.
\end{itemize}
\end{description}

We also must impose refined forms of the condition~\eqref{ass:q-lim}, where now further terms come into play. 
\begin{description}
\item
[\namedlabel{ass:q-lim-gl}{Q$_\text{G}$}]
 For $i, j \in S$ there exist $\gamma_{ij} \in \R$ such that $q_{ij}(x)=q_{ij}+\frac{\gamma_{ij}}{x} +o(x^{-1})$, where $(q_{ij})$ is a stochastic matrix.
\item
[\namedlabel{ass:q-lim-gl+}{Q$_\text{G}^+$}]
 There exist $\delta_3 \in (0,1)$ and $\gamma_{ij} \in \R$ such that $q_{ij}(x)=q_{ij}+\frac{\gamma_{ij}}{x} +o(x^{-1-\delta_3})$. 
\end{description}
The fact that $\sum_{j \in S}q_{ij}(x)=1$ implies, after the following calculation, that $\sum_{j \in S}\gamma_{ij}=0$ for all $i \in S$.

First as the sum of all the transition probabilities on a line is 1, we have
\[
\sum_{j \in S}q_{ij}(x)=1.
\]
Plugging in the condition ${\bf(Q_G)}$, we get
\[
\sum_{j \in S}\left(q_{ij}+ \frac{\gamma_{ij}}{x}+o(x^{-1})\right)=1.
\]
Simplifying,
\[
\sum_{j \in S}\frac{\gamma_{ij}}{x}=o(x^{-1})
\]
for all $x \in \Lambda$. By choosing appropriate $x \in \Lambda$, we have
\[
\left|\sum_{j \in S}\gamma_{ij}\right| \le \epsilon.
\]
Since $\epsilon >0$ was arbitrary, we get
\[
\sum_{j \in S}\gamma_{ij}=0.
\]

The underlying intuition of how many terms we should consider before the error term for each parameter is quite interesting. In principle, we need to take the same order on every basic variable to get the balance of the estimation. That is if we take the first two order terms on the drift of each line, it is sensible to take the first two terms of the transition probabilities. However because the second moment and the interaction between the lines is already on one higher level of the model, as they are like the first level, i.e. pairwise interaction between the basic variables, we only need the first term of the estimation. So now we can have every parameter on the same accuracy of consideration, and it turns out that this accuracy level is enough for determining our classification.

This time, for understanding the statement of our recurrence classification in the generalized Lamperti case, we need the following preliminary result on solutions $\ba = (a_1, \ldots, a_{|S|})^\tra$
to the system of equations
\begin{equation}
\label{e:d-system}
d_i+\sum_{j \in S}(a_j-a_i)q_{ij}=0, \text{ for all } i \in S ;
\end{equation}
we say that a solution $\ba = (a_1, \ldots, a_{|S|})^\tra$ is \emph{unique up to translation}
if all solutions $\ba' = (a'_1, \ldots, a'_{|S|})^\tra$ have $a'_j - a_j$ constant for all $j \in S$.
\begin{lemma}
\label{lem:L4}
Let $d_i \in \R$ and $(q_{ij})$ be an irreducible stochastic matrix with stationary distribution $\pi$. Then the following statements are equivalent.
\begin{itemize}
\item $\sum_{i \in S}d_i \pi_i=0$.
\item There exists a solution $\ba = (a_1, \ldots, a_{|S|})^\tra$ to~\eqref{e:d-system} that is unique up to translation.
\end{itemize}
\end{lemma}

For the proof of Lemma \ref{lem:L4}, see Section 4.3. 

Next we give our main recurrence classification for the model with generalized Lamperti drift. The criteria involve solutions to~\eqref{e:d-system}; as described in Lemma \ref{lem:L4} such solutions are not unique, but nevertheless the expressions in which they appear in Theorem~\ref{thm:L1} are invariant under translations (see Remark~\ref{remarks}(c)), and so the statement makes sense.

\begin{theorem} 
\label{thm:L1}
Suppose that~\eqref{ass:basic} holds, and that~\eqref{ass:p-moments} holds for some $p>2$. Suppose also that~\eqref{ass:q-lim-gl} and~\eqref{ass:drift-gl} hold.
Define $\ba = (a_1, \ldots, a_{|S|})^\tra$ to be a solution to~\eqref{e:d-system}
whose existence is guaranteed by Lemma~\ref{lem:L4}. Define
\begin{equation}
\label{eq:UV}
 U := \sum_{i \in S} \left( 2 e_i + 2 \sum_{j \in S} a_j \gamma_{ij} \right) \pi_i , 
~\text{and}~
  V := \sum_{i \in S} \left( t_i^2 + 2 \sum_{j\in S} a_j d_{ij} \right) \pi_i .\end{equation}
Then the following classification applies.
\begin{itemize}
\item If $U > V$ then $(X_n,\eta_n)$ is transient.
\item If $|U| < V$ then $(X_n,\eta_n)$ is null recurrent.
\item If $U < - V$ then $(X_n,\eta_n)$ is positive recurrent.
\end{itemize}
If, in addition,~\eqref{ass:q-lim-gl+} and~\eqref{ass:drift-gl+} hold, 
then the following condition also applies (yielding an exhaustive classification):
\begin{itemize}
\item If $|U|= V$ then $(X_n,\eta_n)$ is null recurrent.
\end{itemize}
\end{theorem}

From this complicated theorem, you can see that each of the parameters has its own role in controlling the recurrence classification. The $a_i$'s here are actually a key element to the proof of the theorem. They give the shift on each line in the state space, resulting in a transformation to the system. In this way, the system is aligned in a way that the constant term $d_i$'s in the drift are eliminated and we can recover the Lamperti drift after the transformation. When all $a_i$'s are zero, it actually implies all $d_i$'s are zero, and Theorem \ref{thm:L1} recovers the Lamperti drift case as in Theorem \ref{thm:GW}. 

After this transformation on $a_i$'s, the effects of $d_i$'s  transfer to the $a_i$'s, so similar to the Lamperti drift type, we can just compare the size of the Lamperti component of the drift, $e_i$'s to the second moment $t_i$'s, with the proportion of time spent on each line, given by $\pi_i$'s, and most importantly, the effect on the shifting of lines. That is the reason why now we have got some extra terms, with the interactions, $\gamma_{ij}$ and $d_{ij}$ coming into play, depending on the weight that how much we shift the line. Focusing on a single line $i$, the larger value of $\gamma_{ij}$ from any point on any line $j$ in the same direction of the Lamperti drift component, $e_i$'s , with the same direction of the shift $a_i$, (decrease in the other direction) will help to increase the total of the drift, thus giving more force to walk on that line to go either transient or positive recurrent depending on the direction. If the increase on the second term of the transition probability is either opposite to the direction of the Lamperti component of the drift, or the direction of the shift (not both), then they will cancel out each other. So it will have a counter effect on the drift thus lower the force to go through the fluctuation of the variance of the line, giving a higher tendency to go to the case of null recurrence. In the last case that the the transition probability is increases in both the opposite direction of the Lamperti drift component and the direction of shift, these two opposing signs will work together thus increase the force on the line to go to either transient or positive recurrent depending on the direction of the Lamperti drift component. Vice versa for the case of decreasing the transition probabilities.

The other quantity $d_{ij}$, on the other hand, would affect the power of the second moment of the walk. Again, it depends also on the fact if the sign of $a_i$ is the same as the interacting drift $d_{ij}$ or not. The sign of the variance plays no role here because it is always positive. This means for a specific line, if $a_i$ is positive, i.e., shifting to the right, then if $d_{ij}$ is also positive (same direction), then increasing the interacting drift $d_{ij}$ would also increase the fluctuation of the walk. This will help to increase the corrected variance and the walk on this line will need more drift in order to go pass the effect of the second moment. So this increase the tendency for the walk to go to the null-recurrent case. The same happens when both $a_i$ and $d_{ij}$ is negative as they also help each other in the same way. On the contrary, if they have a different sign, increasing $d_{ij}$ would decrease the fluctuation of the walk, thus shorten the tolerance gap for small drifts. This would mean that the walk now need a smaller drift to go though the variance and result in transient or positive recurrent, depending on the sign of the Lamperti drift component.

Weighting these tendencies with the right proportion of time spent on each line, it will adjust the right comparison with the corrected drift and variance in the whole system on average, thus giving you the right classification.

The proof of this theorem will be the main focus of Chapter~\ref{ch:pfhs}.

\subsection{Existence and non-existence of moments}

As in Section~\ref{s:ld}, we quantify the degree of recurrence by establishing existence and non-existence of moments of the passage times $\tau_x$ as defined at \eqref{eqn_tau}. First we give conditions for existence of moments.
\begin{theorem}
\label{thm:L4}
Suppose that~\eqref{ass:basic} holds, and that~\eqref{ass:p-moments} holds for some $p>2$. Suppose also that~\eqref{ass:q-lim-gl} and~\eqref{ass:drift-gl} hold.
Define $\ba = (a_1, \ldots, a_{|S|})^\tra$ to be a solution to~\eqref{e:d-system}
whose existence is guaranteed by Lemma~\ref{lem:L4}.
If for some $\theta >0$, with $U$ and $V$ as given by~\eqref{eq:UV},
\begin{equation}
U + (2\theta -1 ) V < 0,
\end{equation}
then for any $s \in \left[0,\theta \wedge \frac{p}{2} \right]$, we have
$\E[\tau_x^s]<\infty$ for all $x$ sufficiently large.
\end{theorem}

Finally, we give conditions for non-existence of moments. 
\begin{theorem}
\label{thm:L5}
Suppose that~\eqref{ass:basic} holds, and that~\eqref{ass:p-moments} holds for some $p>2$. Suppose also that~\eqref{ass:q-lim-gl} and~\eqref{ass:drift-gl} hold.
Define $\ba = (a_1, \ldots, a_{|S|})^\tra$ to be a solution to~\eqref{e:d-system}
whose existence is guaranteed by Lemma~\ref{lem:L4}.
If for some $\theta \in (0, \frac{p}{2}]$,  with $U$ and $V$ as given by~\eqref{eq:UV},
\begin{equation}
U + ( 2 \theta - 1 ) V > 0,
\end{equation}
then for any $s \geq \theta$, we have $\E[\tau_x^s]=\infty$ for all sufficiently large $X_0 > x$.
\end{theorem}
\begin{remarks}
\label{remarks}
\begin{enumerate}[label=(\alph*),leftmargin=0pt,itemindent=20pt,nosep]
\item The generalization of the state-space $\Sigma$ from $\ZP \times S$ considered in~\cite{AW} and previous work is not merely for the sake of generalization; it is necessary for the technical approach of the generalized Lamperti drift case, whereby we find a transformation $\phi : \Sigma \to \Sigma'$ such that if $(X_n,\eta_n)$ has generalized Lamperti drift, then $\phi (X_n, \eta_n)$ has Lamperti drift (i.e., the constant components of the drifts are eliminated). We then apply the results of Section~\ref{s:ld} to deduce the results in Section~\ref{s:gld}. Even if $\Sigma = \ZP \times S$, the state-space $\Sigma'$ obtained after the transformation $\phi$ will not be (lines are translated in a certain way).
\item The local finiteness assumption ensures that transience of the Markov chain $(X_n, \eta_n)$ is equivalent to $\lim_{n \to \infty} X_n = +\infty$, a.s., and hence all our conditions on $\mu_i(x)$ etc.~are asymptotic conditions as $x \to \infty$.
\item As mentioned above, the non-uniqueness of solutions to~\eqref{e:d-system} is not a problem for the statement of the theorems in this section, because the quantities in our conditions are unchanged under translation of the $a_i$. The variables $a_i$ are well defined here in a non-trivial way. Indeed, Lemma~\ref{lem:L4} shows that if $(a_i,i \in S)$ is a solution then so is $(c+a_i, i \in S)$ for any $c \in \R$, and, furthermore, every solution is of this form. Moreover,  the facts that $\sum_{j \in S}\gamma_{ij}=0$ and $\sum_{i \in S}\sum_{j \in S}d_{ij}\pi_i=\sum_{i \in S}d_i\pi_i=0$ guarantee that replacing every $a_i$ by $c+a_i$
does not change the conditions in our theorems. Another way to go around this is to choose a particular line $0 \in S$ and set $a_0=0$, then $a_i$ is now forced to be unique. There is no loss of generality if $a_0 \ne 0$, we can also obtain a new set of solutions by a translation $\tilde{a}_i=a_i-a_0$. 
\end{enumerate}
\end{remarks}

\chapter{Proofs and technical details}
\label{ch:pfhs}

\section{Semi-martingale criteria for recurrence classification}

In this section we will present some of the fundamental results on the semi-martingale criteria for recurrence classification. These results on discrete-time martingales are due to Doob \cite{JD}. More of these results and their proofs can also be found in \cite{RD,ANS}. First we recall the definitions of martingales, submartingales and supermartingales. 

\begin{definition}[Martingales, submartingales, supermartingales]
A real-valued stochastic process $X_n$ adapted to a filtration $\mathcal{F}_n$ is a \emph{martingale} (with
respect to the given filtration) if, for all $n \ge 0$,
\begin{enumerate}
\item[(i)] $\E | X_n | < \infty$, and
\item[(ii)] $\E \left[ X_{n+1} - X_n | \mathcal{F}_n  \right] = 0$.
\end{enumerate}
If in (ii) ` $=$ ' is replaced by ` $\ge$ ' (respectively, ` $\le$ '), then $X_n$ is called a submartingale (respectively, supermartingale).
\end{definition}

For the term semimartingale, it does not just includes martingales, submartingales and supermartingales. We will use it in a broader context with some stochastic process which drift is of similar structure, on the whole space or just locally on some tail set.

We use the standard notation

\begin{equation}
x^+ := \max \{0, x \}.
\end{equation}

Recall the follow fundamental result from martingale theory.

\begin{theorem}[Martingale convergence theorem]
\label{thm:mct}
Assume that $X_n$ is a submartingale such that $\sup_n \E[X_n^+ ] < \infty$. Then there is an integrable random variable $X$ such that $X_n \to X$ a.s. as $n \to \infty$.
\end{theorem}

For the proof please see \cite{RD}, Therem 5.2.8. Now we give an important corollary to Theorem \ref{thm:mct} and Fatou's lemma.

\begin{theorem}[Convergence of non-negative supermartingales]
Assume that $X_n \ge 0$  is a supermartingale. Then there is an integrable random variable $X$ such that $X_n \to X$ a.s. as $n \to \infty$, and $\E [X] \le \E [X_0] $.
\end{theorem}

For the proof please see \cite{RD}, Therem 5.2.9. Based on the previous convergence, we give the following recurrence and transience criteria, which are central to our analysis of the half strip model. The statements here are taken from Section 2.5 of \cite{MPW}.

\begin{theorem}[Recurrence criterion]
\label{thm:mcrc}
An irreducible Markov chain $X_n$ on a countably infinite state space $\Sigma$ is recurrent if and only if there exist a
function $f: \Sigma \to \RP$ and a finite non-empty set $A \subset \Sigma$ such that
\begin{equation}
\E \left[f(X_{n+1}) - f(X_n) \text{ } | \text{ } X_n = x  \right] \le 0, \text{ for all } x \in \Sigma \text{ } \backslash \text{ } A,
\end{equation}
and $f(x) \to \infty $ as $x \to \infty$.
\end{theorem}

\begin{theorem}[Transience criterion]
\label{thm:mctc}
An irreducible Markov chain $X_n$ on a countably infinite state space $\Sigma$ is transient if and only if there exist a
function $f: \Sigma \to \RP$ and a non-empty set $A \subset \Sigma$ such that
\begin{equation}
\E \left[f(X_{n+1}) - f(X_n) \text{ } | \text{ } X_n = x  \right] \le 0, \text{ for all } x \in \Sigma \text{ } \backslash \text{ } A,
\end{equation}
and
\begin{equation}
\label{con:mctc}
f(y) < \inf_{x \in A} f(x), \text{ for at least one site } y \in \Sigma \text{ } \backslash \text{ } A.
\end{equation}
\end{theorem}

These two criterion can be trace back to the work of F.G. Foster \cite{FGF}. He proved the `if' part of Theorem \ref{thm:mcrc} in the case where the exceptional set $A$ is a singleton. For the finite set version for this direction can be found in Pakes \cite{AGP}. The `only if' part of Theorem \ref{thm:mcrc} is due to Mertens \emph{et al.} \cite{MSZ}. Foster \cite{FGF} also proved Theorem \ref{thm:mctc} for the case where $A$ is a single point. The finite set version is due to Harris and Marlin \cite{HM} and Mertens \emph{et al.} \cite{MSZ}.

\section{Lyapunov function estimates for the half strip}

Recall in Section \ref{s:lp} we proved the P\'olya's Theorem with a Lyapunov function using the technique of reduction of dimensionality. We took $X_n := \| S_n \|$ as our function and one critical bit to apply the semi-martingale criteria is the calculation of expectations. Although it is pretty straightforward in the model of simple symmetric random walk, it can take a bit of effort in general models.

The main difficulty in applying the theorems in the previous section for the classification is to find a good Lyapunov function which gives suitable $\E \left[f(X_{n+1}) - f(X_n) \text{ } | \text{ } X_n = x  \right]$. Depending on the model, these functions can be sometimes simple and easy to find, while sometimes it is very difficult to come up with the right function and calculate the expectation stated. In our half strip problem, we will give a Lyapunov function for each of the constant drift case and the Lamperti drift case. The formulation and the calculation of the former one is straightforward, while the latter one requires a lot more effort. They show both the strength and weakness of this Lyapunov function method. Although the method is very robust and constructive, it is tricky to start with the right function without any experience. Also, without explicit calculation of the expectation, it is hard to tell if the function that we picked is indeed the right one. The Lyapunov function for a specific model is usually not unique and it can be in various forms. To pick a good Lyapunov function that enables simplier calculation among all those which will satisfy the conditions in the theorems is a skill derived from experience.

\subsection{Lyapunov function for constant drift}

Our analysis for the constant drift case is based on two Lyapunov functions $g : \Sigma \to (0,\infty)$ and $h_\nu : \Sigma \to (0,\infty)$ for $\nu>0$ for the recurrent case and transient case respectively, defined by

\begin{equation}
\label{lya_con1}
g(x, i) := x + b_i
\end{equation}
for some $b_i \in \R$, and

\begin{equation}
\label{lya_con2}
h_\nu(x, i) :=
\begin{cases}
 x^{-\nu}-\nu b_{i} x^{-\nu-1} & \text{if } x \ge x_0,\\
x_0^{-\nu}-\nu b_{i} x_0^{-\nu-1} & \text{if } x < x_0,
\end{cases}
\end{equation}
where 
$b_i \in \R$ and $x_0 := 1+ 2 \nu  \max_{i \in S}|b_{i}| $.

We will need the following increment moment estimates for our Lyapunov function in the constant drift case. For the function $g$, we have the following lemma.

\begin{lemma}
\label{lem:calf0}
Suppose that~\eqref{ass:basic} holds, and that~\eqref{ass:p-moments} holds for some $p>1$. Suppose also that~\eqref{ass:q-lim} and~\eqref{ass:drift-constant} hold.
Then we have, as $x \to \infty$,
\begin{equation}
\E_{x,i} \left[  g ( X_{n+1}, \eta_{n+1} ) - g (X_n, \eta_n) \right] = d_i + \sum_{j \in S} (b_j - b_i) q_{ij} + o(1).
\end{equation}
\label{e:lyapunov0}
\end{lemma}

\begin{proof}
Using the condition \eqref{ass:drift-constant} that $\E_{x,i} \left[ X_{n+1} - X_n \right] = d_i + o(1) $, we get
\begin{align*}
\E_{x,i} \left[  g ( X_{n+1}, \eta_{n+1} ) - g (X_n, \eta_n) \right] &=  \E_{x,i} \left[X_{n+1} - X_n \right] + \E_{x,i} \left[ b_{\eta_{n+1}} - b_{\eta_n}   \right] \\
&=  \left[ d_i + o(1) \right] + \sum_{j \in S} q_{ij} (b_j-b_i),
\end{align*}
by applying~\eqref{ass:q-lim} in the last step. Hence we have the result as stated.
\end{proof}

On the other hand for the function $h_\nu$, we have a slightly more complicated situation.

\begin{lemma}
\label{lem:calf01}
Suppose that~\eqref{ass:basic} holds, and that~\eqref{ass:p-moments} holds for some $p>1$. Suppose also that~\eqref{ass:q-lim} and~\eqref{ass:drift-constant} hold.
Then for any $\nu \in (0, p]$, we have, as $x \to \infty$,
\begin{align}
\label{e:lyapunov01}
& {} 
\E_{x,i} \left[  h_\nu ( X_{n+1}, \eta_{n+1} ) - h_\nu (X_n, \eta_n) \right] = -\nu x^{-1-\nu} \left( d_i + \sum_{j \in S} (b_j - b_i) q_{ij} + o(1) \right) .
\end{align}
\end{lemma}

\begin{proof}
Denote $\Delta_n :=X_{n+1}-X_n$, and consider the event $E_n := \{ | \Delta_n | \leq X_n^\zeta \}$ where $\zeta \in (0,1)$. Then we choose $\zeta \in (0,1)$ and $x_1 \in \RP$ such that $x_1 - x_1^\zeta \geq x_0$; then on the event $E_n \cap \{ X_n \geq x_1\}$ we have $X_{n+1} \geq x_1 - x_1^\zeta \geq x_0$. Thus, for all $x \geq x_1$, we may write
\begin{align}
& \quad \E_{x,i} \left[  h_\nu ( X_{n+1}, \eta_{n+1} ) - h_\nu (X_n, \eta_n) \right] \nonumber \\
&= \E_{x,i} \left[ \left(  X_{n+1}^{-\nu}- X_n^{-\nu}\right) \2{E_n} \right] -\nu \E_{x,i} \left[ \left(  b_{\eta_{n+1}} X_{n+1}^{-\nu-1} - b_{\eta_{n}} X_n^{-\nu-1} \right) \2{E_n} \right] \nonumber \\& \quad + \E_{x,i} \left[ \left( h_\nu ( X_{n+1}, \eta_{n+1} ) - h_\nu (X_n, \eta_n) \right) \2{E_n^\rc} \right]  \label{eqnc1}
\end{align}

For the first term in equation \eqref{eqnc1}, we apply Taylor's expansion and get
\begin{align}
\E_{x,i} \left[ \left(  X_{n+1}^{-\nu}- X_n^{-\nu}\right) \2{E_n} \right] &= x^{-\nu} \E_{x,i} \left[ \left( \left( 1+ \frac{\Delta_n}{X_n} \right)^{-\nu} -1 \right) \2{E_n} \right] \nonumber \\
&= x^{-\nu} \E_{x,i} \left[ -\nu \left( \frac{\Delta_n}{X_n} \right) \2{E_n} + Z \right] \nonumber
\end{align}
where $|Z| \le C X_n^{-2} |\Delta_n|^2 \2{E_n}$, $C \in \R$ a constant. As $|Z| \le C X_n^{-2} |\Delta_n| \cdot |\Delta_n| \2{E_n}\le C X_n^{-2} |\Delta_n| X_n^{\zeta}$, we have
\begin{equation}
\left| \E_{x,i} \left[ Z \right] \right| \le \E_{x,i} \left[ |Z| \right] \le C x^{\zeta-2} \E_{x,i} \left[ |\Delta_n| \right] = O(x^{\zeta-2}), \nonumber
\end{equation}
using~\eqref{ass:p-moments} in the last step. Thus we get $\E_{x,i} \left[ Z \right]= o(x^{-1})$. So we get 
\begin{equation}
\E_{x,i} \left[ \left(  X_{n+1}^{-\nu}- X_n^{-\nu}\right) \2{E_n} \right] = -\nu x^{-1-\nu} \left( \E_{x,i} \left[ \Delta_n \2{E_n} \right] + o(1) \right). \nonumber
\end{equation}
Observing 
\begin{equation}
|\Delta_n| \2{E_n^c} = |\Delta_n|^p |\Delta_n|^{1-p} \2{E_n^\rc} \le |\Delta_n|^p X_n^{\zeta(1-p)}, \nonumber
\end{equation}
we have 
\begin{equation}
\left|  \E_{x,i} \left[ \Delta_n \2{E_n^\rc} \right] \right| \le \E_{x,i} \left[ |\Delta_n| \2{E_n^\rc} \right] \le C_p x^{\zeta(1-p)}, \nonumber
\end{equation}
by~\eqref{ass:p-moments}, where $C_p$ is a constant depending on $p$. As $\zeta \in (0,1)$ and $p>1$, we obtain
\begin{equation}
\E_{x,i} \left[ \Delta_n \2{E_n} \right] = \E_{x,i} \left[ \Delta_n  \right] - \E_{x,i} \left[ \Delta_n \2{E_n^\rc} \right] = \E_{x,i} \left[ \Delta_n  \right] + o(1) \nonumber
\end{equation}
Using~\eqref{ass:drift-constant}, we get
\begin{equation}
\E_{x,i} \left[ \left(  X_{n+1}^{-\nu}- X_n^{-\nu}\right) \2{E_n} \right] = -\nu x^{-1-\nu} \left[ d_i + o(1) \right].
\end{equation}

For the second term in equation \eqref{eqnc1}, first observe that
\begin{align}
\label{eqk}
& \quad \E_{x,i} \left[ \left(  b_{\eta_{n+1}} X_{n+1}^{-\nu-1} - b_{\eta_{n}} X_n^{-\nu-1} \right) \2{E_n} \right] \nonumber\\
& =\E_{x,i} \left[   b_{\eta_{n+1}} \left( X_{n+1}^{-\nu-1} -  X_n^{-\nu-1} \right) \2{ E_n} \right] + \E_{x,i} \left[ \left( b_{\eta_{n+1}}  - b_{\eta_{n}} \right)  X_n^{-\nu-1} \2{ E_n} \right] .
\end{align}
We deal with the two terms on the right-hand side of~\eqref{eqk} separately. First,
\begin{align*}
\left| \E_{x,i} \left[   b_{\eta_{n+1}} \left( X_{n+1}^{-\nu-1} -  X_n^{-\nu-1} \right) \2{ E_n} \right] \right| \leq \Bigl( \max_{j \in S} | b_j |  \Bigr) 
\E_{x,i} \left[  \left| X_{n+1}^{-\nu-1} -  X_n^{-\nu-1} \right| \2{ E_n} \right] ,
\end{align*}
where, by Taylor's formula, given $X_n =x$,
\[ \left| X_{n+1}^{-\nu-1} -  X_n^{-\nu-1} \right| \2{ E_n} = O ( x^{-\nu-2 +\zeta } ) = o (x^{-\nu-1} ) .\]
On the other hand,
\begin{align*}
\E_{x,i} \left[ \left(  b_{\eta_{n+1}}  - b_{\eta_{n}} \right)  X_n^{-\nu-1} \2{ E_n} \right] & = x^{-\nu-1} \sum_{j \in S}  ( b_j  - b_i )  \Pr_{x,i} \left[ \{ \eta_{n+1} = j \} \cap E_n \right] .
\end{align*}
Here we have
\begin{align*}
\left| \Pr_{x,i} \left[ \{ \eta_{n+1} = j \} \cap E_n \right] - q_{ij} (x) \right| &\leq \Pr_{x,i} \left[ E^\rc_n \right] =  \E_{x,i} \left[ |\Delta_n|^p |\Delta_n|^{-p} \1{E^\rc_n} \right] \\ &\le x^{-p\zeta} \E_{x,i} \left[ |\Delta_n|^p \right] \to 0.
\end{align*}
Hence we get
\[ \E_{x,i} \left[ \left( b_{\eta_{n+1}} X_{n+1}^{-\nu-1} - b_{\eta_n}  X_n^{-\nu-1} \right) \2{ E_n} \right] = x^{-\nu-1} \sum_{j \in S} (b_j - b_i ) q_{ij}(x) + o (x^{-\nu-1}) .\]
For the third term in equation \eqref{eqnc1}, we observe that
\[
0 \le h_\nu(x,i) \le C, \text{ for all } x \geq 0 \text{ and all } i \in S.
\]
for some $C$,  depending on $\nu$  and $(b_i, i \in S)$.
As $\nu \in [0, p-1)$. For all $x$ and $i$,  we have
\begin{align*}
  \left| \E_{x,i} \left[\left(h_{\nu}(X_{n+1},\eta_{n+1})-h_{\nu}(X_n,\eta_n)\right) \2{ E^\rc_n} \right]\right| \le C \Pr_{x,i} \left( E^\rc_n \right) \\  =  C \E_{x,i} \left[ |\Delta_n|^p |\Delta_n|^{-p} \2{E^\rc_n} \right] \le C x^{-p\zeta} \E_{x,i} \left[ |\Delta_n|^p \right]= O(x^{-p\zeta}),
\end{align*}
Since $-\nu > 1-p$, we can choose $\zeta$ such that $0<\frac{1+\nu}{p}<\zeta<1$, which gives $-\zeta p < -1-\nu$. Finally, grouping all three terms together gives the desired result.

\end{proof}

\subsection{Lyapunov function for Lamperti drift}

Our analysis for the Lamperti drift case is a lot more complicated. It is based on the Lyapunov function $f_\nu : \Sigma \to (0,\infty)$ defined for $\nu \in \R$ by
\begin{equation}
\label{lya_lam}
f_\nu(x, i) :=
\begin{cases}
 x^\nu+\frac{\nu}{2} b_{i} x^{\nu-2} & \text{if } x \ge x_0,\\
x_0^\nu+\frac{\nu}{2} b_{i} x_0^{\nu-2} & \text{if } x < x_0,
\end{cases}
\end{equation}
where 
$b_i \in \R$ and $x_0 := 1+ \sqrt{|\nu|  \max_{i \in S}|b_{i}| }$.

First we want to establishes some bounds on~$f_\nu$.

\begin{lemma}
\label{cal21}
Suppose $\nu \in \R$. There exist positive constants $k_1,k_2 \in (0, \infty)$,
 depending on $\nu$  and $(b_i, i \in S)$, such that 
\[
k_1 (1+x)^\nu \le f_\nu(x,i) \le k_2(1+x)^\nu, \text{ for all } x \geq 0 \text{ and all } i \in S.
\]
\end{lemma}

\begin{proof}

To start with, we consider the case when $x \ge x_0$, with $x_0=1+\sqrt{|\nu| B}$, where $B = \max_{i \in S}|b_{i}|$, we have
\begin{equation}
\left|\frac{\nu}{2} b_{i} x^{\nu-2} \right| \le \left| \frac{\frac{\nu}{2} b_{i} x^{\nu}}{x_0^2} \right| 
= \left| \frac{\frac{\nu}{2} b_{i} x^{\nu}}{\left(1+\sqrt{|\nu| B}\right)^2} \right| 
\le\frac{\frac{|\nu|}{2} \left|b_{i}\right| x^{\nu}}{|\nu| B}  
\le \frac{1}{2} x^\nu \label{cal24}
\end{equation}
So we have 
\begin{equation}
\frac{1}{2} x^\nu \le f_\nu(x,i) \le \frac{3}{2} x^\nu \label{cal22}
\end{equation}
for all $x \ge x_0$. Noticing $x \ge x_0>1$, which implies $1<\frac{x+1}{2}<x$, we have 
\[
\min(1,2^{-\nu}) (1+x)^\nu \le x^\nu \le \max(1,2^{-\nu}) (1+x)^\nu
\]
Together with the inequality \eqref{cal22}, we have
\begin{equation}
\frac{1}{2} \min(1,2^{-\nu}) (1+x)^\nu \le f_\nu(x,i) \le \frac{3}{2} \max(1,2^{-\nu}) (1+x)^\nu \label{cal23}
\end{equation}
for all $x \ge x_0$. 

On the other hand, Suppose $x < x_0$. Then $f_\nu(x,i)=f_\nu(x_0,i)$, which is a case of \eqref{cal22} when $x=x_0$, so we have 
\begin{equation}
\frac{1}{2} x_0^\nu \le f_\nu(x,i) \le \frac{3}{2} x_0^\nu \label{cal25}
\end{equation}
for all $x < x_0$.
Now consider the fact that for $x < x_0$, 
\begin{equation*}
\min(1,(1+x_0)^{-\nu}) \le (1+x)^{-\nu} \le \max(1,(1+x_0)^{-\nu}).
\end{equation*}
Together with the inequality \eqref{cal25}, we get for $x < x_0$,
\begin{equation}
\frac{1}{2}x_0^\nu \min(1,(1+x_0)^{-\nu})(1+x)^\nu \le f_\nu(x, i) \le \frac{3}{2}x_0^\nu \max(1,(1+x_0)^{-\nu})(1+x)^\nu.
\end{equation}
Hence the proof is completed by taking appropriate positive constants $k_1$ and $k_2$ in different cases as just shown. 
\end{proof}

The next result, which is central to what follows, provides increment moment estimates for our Lyapunov function
in the Lamperti drift case.

\begin{lemma}
\label{lem:calf}
Suppose that~\eqref{ass:basic} holds, and that~\eqref{ass:p-moments} holds for some $p>2$. Suppose also that~\eqref{ass:q-lim} and~\eqref{ass:drift-lamperti} hold.
Then for any $\nu \in (2-p, p]$, we have, as $x \to \infty$,
\begin{align}
\label{e:lyapunov1}
& {} 
\E_{x,i} \left[  f_\nu ( X_{n+1}, \eta_{n+1} ) - f_\nu (X_n, \eta_n) \right] 
\nonumber\\
  &  \qquad {}
= \frac{\nu}{2} x^{\nu-2} \left( 2 c_i + (\nu-1) s_i^2 + \sum_{j \in S} (b_j - b_i) q_{ij} + o(1) \right) 
 .\end{align}
\end{lemma}

The rest of this section is devoted to the proof of Lemma~\ref{lem:calf}.
Denote $\Delta_n :=X_{n+1}-X_n$, and again consider the event $E_n := \{ | \Delta_n | \leq X_n^\zeta \}$ where $\zeta \in (0,1)$.
The basic idea behind the proof of Lemma~\ref{lem:calf} is to use a Taylor's formula expansion. Such an expansion is valid only if $\Delta_n$ is not too large;
to handle various truncation estimates we will thus need the following result.

\begin{lemma}
\label{lem:indicator}
Suppose that~\eqref{ass:p-moments} holds for some $p>2$. Then for any $\zeta \in (0,1)$ and any $q \in [0,p]$, we have 
\begin{equation}
\label{col0}
\E_{x,i} \left[ |\Delta_n|^q \2 { E^\rc_n}  \right] \le C_p x^{\zeta(q-p)}. 
\end{equation}
Furthermore, if $\zeta \in (\frac{1}{p-1},1)$, we have
\begin{align}
\E_{x,i}\left[\Delta_n \2{ E_n} \right] &= \E_{x,i}\left[ \Delta_n\right] + o \left(x^{-1}\right), 
\label{col1} \\
\E_{x,i}\left[\Delta_n^2 \2{ E_n} \right] &= \E_{x,i}\left[\Delta_n^2 \right] + o \left(1 \right). 
\label{col2}
\end{align}
\end{lemma}
\begin{proof}
For $q \in [0,p]$,
\begin{align}
|\Delta_n|^q \2{ E^\rc_n}   = |\Delta_n|^p |\Delta_n|^{q-p} \2{ E^\rc_n} \le |\Delta_n|^p X_n^{\zeta(q-p)} \label{cal3}
\end{align}
The inequality follows as $q-p \le 0$, so under the condition that $|\Delta_n| > X_n^\zeta$, we have $|\Delta_n|^{q-p} \le (X_n^{\zeta})^{(q-p)} $. 
Taking the conditional expectation on both sides of~\eqref{cal3} and using the condition~\eqref{ass:p-moments}, we obtain~\eqref{col0}.
For the second statement, we use the fact that for $r \in \{1,2\}$, 
\[
\E_{x,i}\left[\Delta_n^r\right] = \E_{x,i}\left[\Delta_n^r \2{ E_n} \right] + \E_{x,i}\left[\Delta_n^r \2{ E^\rc_n} \right],
\]
where, by the $q=r$ case of~\eqref{col0},
\begin{align}
\left| \E_{x,i}\left[\Delta_n^r \2{ E^\rc_n} \right]\right| \le  \E_{x,i}\left[ |\Delta_n|^r \2{ E^\rc_n} \right] \le  C_p x^{\zeta(r-p)}. \label{cal5}
\end{align}
When $r=1$, we choose $\zeta \in (\frac{1}{p-1}, 1)$, so we have $\zeta (1-p)<-1$, and then~\eqref{cal5} gives~\eqref{col1}. When $r=2$, we know $r-p<0$, and then~\eqref{cal5} gives~\eqref{col2}.
\end{proof}

To obtain Lemma~\ref{lem:calf}, we decompose the increment of $f_\nu$. First note that, for $\zeta \in (0,1)$,
\begin{align}
\label{e:two_terms1}
 &\E_{x,i}  \left[f_\nu(X_{n+1},\eta_{n+1})-f_\nu(X_n,\eta_n) \right]
   = \E_{x,i} \left[  ( f_\nu(X_{n+1},\eta_{n+1})-f_\nu(X_n,\eta_n) ) \2{ E_n} \right] \nonumber\\
 & \qquad \qquad \qquad {} + \E_{x,i} \left[  (f_\nu(X_{n+1},\eta_{n+1})-f_\nu(X_n,\eta_n) ) \2{ E^\rc_n} \right] .
\end{align}
Choose $\zeta \in (\frac{1}{p-1},1)$ and $x_1 \in \RP$
such that $x_1 - x_1^\zeta \geq x_0$; then on the event $E_n \cap \{ X_n \geq x_1\}$ we have $X_{n+1} \geq x_1 - x_1^\zeta \geq x_0$.
Thus, for all $x \geq x_1$, we may write
\begin{align}
\label{e:two_terms2}
   & \quad {}  \E_{x,i} \left[ \left( f_\nu(X_{n+1},\eta_{n+1})-f_\nu(X_n,\eta_n)\right) \2{ E_n} \right]    \nonumber\\
 & = \E_{x,i}\left[\left(X_{n+1}^\nu-X_{n}^\nu \right) \2{ E_n} \right]  +\frac{\nu}{2} \E_{x,i}\left[\left(b_{\eta_{n+1}}X_{n+1}^{\nu-2}-b_{\eta_n}X_{n}^{\nu-2}\right)\2{ E_n} \right] .
\end{align}
We proceed to estimate the terms on the right-hand sides of~\eqref{e:two_terms1} and~\eqref{e:two_terms2} separately, via a series of lemmas. 

\begin{lemma}
\label{cal4}
Suppose that~\eqref{ass:p-moments} holds for some $p>2$. Suppose also that~\eqref{ass:drift-lamperti} holds.
Let  $\zeta \in (\frac{1}{p-1},1)$. Then for any $r \in \R$, we have, as $x \to \infty$,
\[ \E_{x,i} \left[ \left( X_{n+1}^r - X_n^r \right) \2{ E_n} \right] = r x^{r-2} \left( c_i + \frac{r-1}{2} s_i^2 + o(1) \right) .\]
\end{lemma}
\begin{proof} 
By Taylor's formula we have that
\begin{align}
& \E_{x,i}\left[\left(X_{n+1}^r-X_{n}^r \right) \2{ E_n} \right]  
 = x^r \E_{x,i} \left[ \left( ( 1 + x^{-1} \Delta_n )^r - 1 \right) \2 { E_n } \right] \nonumber\\
& \qquad \qquad \qquad = x^r \E_{x,i} \left[\left(r \left(\frac{\Delta_n}{X_{n}}\right)+\frac{r(r-1)}{2}\left(\frac{\Delta_n}{X_{n}}\right)^2 \right)\2{ E_n} +Z \right], 
\label{cal01}
\end{align}
where $|Z| \le C X_n^{-3}|\Delta_n|^3 \2{ E_n}$ for some fixed constant $C \in \RP$. To bound the term $\E_{x,i}[Z]$, first we observe that
\[
|Z| \le CX_n^{-3} |\Delta_n|^2|\Delta_n|\2{ E_n} \le C |\Delta_n|^2 X_n^{\zeta-3}. \]
 Taking expectations on both sides of the last inequality, we obtain 
\begin{align*}
\left| \E_{x,i}[Z] \right| \le \E_{x,i} |Z| \le C x^{\zeta-3} \E_{x,i}[|\Delta_n|^2] = O ( x^{\zeta-3} ),
\end{align*}
using~\eqref{ass:p-moments}. Since $\zeta < 1$ this implies  $\E_{x,i}[Z]= o(x^{-2})$, so  the expression~\eqref{cal01} becomes 
\begin{align}
& \quad {} \E_{x,i}\left[\left(X_{n+1}^r-X_{n}^r \right) \2{ E_n} \right] \nonumber \\
& = r x^{r-1}\E_{x,i}\left[ \Delta_n \2{ E_n}\right] + \frac{r (r-1)}{2}x^{r-2}\E_{x,i}\left[\Delta_n^2 \2{ E_n}\right]+o\left(x^{r-2}\right) .
\label{cal02} 
\end{align}
Then by  Lemma~\ref{lem:indicator} together with the facts that, under~\eqref{ass:drift-lamperti},
 \begin{align*}
\E_{x,i}[\Delta_n] = \mu_i(x) = \frac{c_i}{x}+o(x^{-1}), \text{ and } \E_{x,i}[\Delta_n^2] =\sigma^2_i(x) =s_i^2 + o(1),
\end{align*}
we obtain
 \begin{align*}
& \quad \E_{x,i}\left[\left(X_{n+1}^r-X_{n}^r \right) \2{ E_n} \right]  \\
&= r x^{r-1}\left(\frac{c_i}{x}+o\left(x^{-1}\right)\right) + \frac{r(r-1)}{2}x^{r-2}\left(s_i^2 + o(1)\right)+o\left(x^{r-2}\right) , 
\end{align*}
from which the statement in the lemma follows.
\end{proof}

\begin{lemma}
\label{cal11}
Suppose that~\eqref{ass:p-moments} holds for some $p>0$.  
Let $r \in \R$ and $\zeta \in (0,1)$, and let $g : S \to \R$. Then, as $x \to \infty$,
\[ \E_{x,i} \left[ \left( g (\eta_{n+1} ) X_{n+1}^r - g (\eta_n ) X_n^r \right) \2{ E_n} \right] = x^r \sum_{j \in S} (g(j) - g(i) ) q_{ij}(x) + o (x^r) .\]
\end{lemma}
\begin{proof} 
First observe that
\begin{align}
\label{eq1}
& \quad \E_{x,i} \left[ \left( g (\eta_{n+1} ) X_{n+1}^r - g (\eta_n ) X_n^r \right) \2{ E_n} \right] \nonumber\\
& =\E_{x,i} \left[   g (\eta_{n+1} ) \left( X_{n+1}^r -  X_n^r \right) \2{ E_n} \right] + \E_{x,i} \left[ \left( g (\eta_{n+1} )  - g (\eta_n ) \right)  X_n^r \2{ E_n} \right] .
\end{align}
We deal with the two terms on the right-hand side of~\eqref{eq1} separately. First,
\begin{align*}
\left| \E_{x,i} \left[   g (\eta_{n+1} ) \left( X_{n+1}^r -  X_n^r \right) \2{ E_n} \right] \right| \leq \Bigl( \max_{j \in S} | g(j) |  \Bigr) 
\E_{x,i} \left[  \left| X_{n+1}^r -  X_n^r \right| \2{ E_n} \right] ,
\end{align*}
where, by Taylor's formula, given $X_n =x$,
\[ \left| X_{n+1}^r -  X_n^r \right| \2{ E_n} = O ( x^{r+\zeta -1} ) = o (x^r ) .\]
On the other hand,
\begin{align*}
\E_{x,i} \left[ \left( g (\eta_{n+1} )  - g (\eta_n ) \right)  X_n^r \2{ E_n} \right] & = x^r \sum_{j \in S}  ( g (j )  - g ( i ) )  \Pr_{x,i} \left[ \{ \eta_{n+1} = j \} \cap E_n \right] .
\end{align*}
Here
$\left| \Pr_{x,i} \left[ \{ \eta_{n+1} = j \} \cap E_n \right] - q_{ij} (x) \right| \leq \Pr_{x,i} \left[ E^\rc_n \right] \to 0$, 
by the $q=0$ case of Lemma~\ref{lem:indicator}. Combining these calculations gives the result.
\end{proof}

Combining the last two results we obtain the following estimate for the first term on the right-hand side
of~\eqref{e:two_terms1}.

\begin{lemma}
\label{cal6}
Suppose that~\eqref{ass:p-moments} holds for some $p>2$. Suppose also that~\eqref{ass:drift-lamperti} and~\eqref{ass:q-lim} hold.
Let $\zeta \in (\frac{1}{p-1},1)$. Then for any $r \in \R$, we have, as $x \to \infty$,
\begin{align*}
& \E_{x,i} \left[ \left( f_r (X_{n+1},\eta_{n+1} ) - f_r (X_n,\eta_n) \right) \2{ E_n} \right] \\
& \qquad = \frac{r}{2} x^{r-2} \left( 2 c_i +(r-1) s_i^2 + \! \sum_{j \in S} (b_j - b_i ) q_{ij} + o(1) \right).
\end{align*}

\end{lemma}
\begin{proof}
In the equation~\eqref{e:two_terms2}
we apply Lemma~\ref{cal4} and  Lemma~\ref{cal11} with $g(y)=b_y$ and $r-2$ in place of $r$;
 together with~\eqref{ass:q-lim} we obtain the result.
\end{proof}

We have the following estimate for the second term on the right-hand side
of~\eqref{e:two_terms1}.

\begin{lemma}
\label{cal17}
Suppose that~\eqref{ass:p-moments} holds for some $p>2$. 
Then for any $r \in (2-p, p]$, we can choose $\zeta \in (0,1)$ for which, as $x \to \infty$, 
\begin{align*}
\E_{x,i} \left[ \left|  f_r ( X_{n+1}, \eta_{n+1} ) - f_r (X_n, \eta_n) \right|  \2{ E^\rc_n}  \right] = o (x ^{r-2} ).
\end{align*}
\end{lemma}
\begin{proof} 
 We may suppose throughout this proof that $X_n \ge 1$. 
First suppose that $r \in (0,p]$. 
If $|\Delta_n|\le\frac{X_n}{2}$, then $\frac{X_n}{2}\le X_n+ \Delta_n \le \frac{3X_n}{2}$. Thus with Lemma~\ref{cal21} we have, on $\{ |\Delta_n|\le\frac{X_n}{2} \}$,
\begin{align}
f_r(X_{n+1}, \eta_{n+1}) &\le k_2 \left(1+\frac{3X_n}{2}\right)^r \le C_1 \left(1+X_n \right)^r , \label{cal28}
\end{align}
for 
some constant $C_1 \in \RP$. On the other hand, if $|\Delta_n|> \frac{X_n}{2}$, then $0 \le X_{n+1} = X_n + \Delta_n  \le 3|\Delta_n|$. So with Lemma~\ref{cal21}, 
we have, on $\{ |\Delta_n| > \frac{X_n}{2} \}$,
\begin{align}
\label{cal27}
f_r(X_{n+1}, \eta_{n+1}) &\le k_2 \left(1+3|\Delta_n|\right)^r \le C_2\left|\Delta_n \right|^r , 
\end{align}
for some constant $C_2 \in \RP$. Combining the results of \eqref{cal28} and \eqref{cal27}, 
we obtain for $r >0$,
\begin{equation}
f_r(X_{n+1}, \eta_{n+1}) \le C_3 \left(1+X_n \right)^r + C_3 |\Delta_n|^r, \label{cal32}
\end{equation}
for some $C_3 \in \RP$. Hence, for $r >0$, for some $C \in \RP$, 
\begin{align*}
&\quad \left|\E_{x,i} \left[ (f_r(X_{n+1},\eta_{n+1})-f_r(X_n,\eta_n)) \2{ E^\rc_n} \right] \right|  \\
& \le f_r (x, i) \Pr_{x,i} \left[ E^\rc_n \right] + \E_{x,i} \left[ | f_r(X_{n+1},\eta_{n+1}) | \2{ E^\rc_n} \right]  \\
& \leq C \left(1+x \right)^r \Pr_{x,i} \left[ E^\rc_n \right] + C \E_{x,i} \left[ |\Delta_n|^r \2{ E^\rc_n} \right] ,
\end{align*}
where we have used Lemma~\ref{cal21} and inequality~\eqref{cal32}.
Then, by the $q=0$ and $q=r \in (0,p]$ cases of Lemma~\ref{lem:indicator} we have
\[ \left|\E_{x,i} \left[ (f_r(X_{n+1},\eta_{n+1})-f_r(X_n,\eta_n)) \2{ E^\rc_n} \right] \right| = O(x^{r-p\zeta})+O(x^{\zeta(r-p)}) = O(x^{r-p\zeta}) ,\]
since $\zeta<1$. This last term is $o(x^{r-2})$ provided $r-p\zeta<r-2$, i.e., $\zeta>\frac{2}{p}$.

Finally, suppose that $r \in (2-p, 0]$. Now by Lemma~\ref{cal21}, we have
$0 \leq f(x, i) \leq C$ for some $C \in \RP$ and all $x$ and $i$, 
so that
\begin{align*}
  \left| \E_{x,i} \left[\left(f_r(X_{n+1},\eta_{n+1})-f_r(X_n,\eta_n)\right) \2{ E^\rc_n} \right]\right| \le C \Pr_{x,i} \left[ E^\rc_n \right] = O(x^{-p\zeta}),
\end{align*}
by the $q=0$ case of Lemma~\ref{lem:indicator}. Since $r > 2-p$, we can choose $\zeta$ such that $0<\frac{2-r}{p}<\zeta<1$, which gives $-\zeta p < r -2$. 
\end{proof}

Now we are ready to complete the proof Lemma~\ref{lem:calf}. 

\begin{proof}[Proof of Lemma~\ref{lem:calf}]
The expression for the first moment in~\eqref{e:lyapunov1}
is simply a combination of the $r=\nu$ cases of Lemmas~\ref{cal6} and~\ref{cal17}.
\end{proof}

\section{Some consequences of the Fredholm alternative}

This section serves two purposes. The first aim for this section is to prepare for the proofs in the next section, which we need to understand the term with $b_i$'s in the Lyapunov function estimate for the case of Lamperti drift. The other purpose is to show existence of $a_i$ in Lemma~\ref{lem:L4} for our translation in the generalized Lamperti drift case.

\subsection{Fredholm alternative}

The following well-known algebraic result will enable us to show that suitable $b_i$ exist to construct the Lyapunov function $f_\nu$ as defined at \eqref{lya_lam} under appropriate conditions involving $\pi_j$.  

In this section, vectors are column vectors on $\R^{|S|}$, $\0$ denotes the column vector whose components are all zero, and $I$ denotes the $|S| \times |S|$ identity matrix. We will need the following well-known algebraic result.

\begin{lemma}[Fredholm alternative]
\label{lem:fa}
Given an $|S| \times |S|$ matrix $A$ and a  column vector $\bb$, the equation $A\ba = \bb$ has a solution $\ba$ if and only if  any  column vector $\by$
for which $A^\tra \by = \0$ satisfies $\by^\tra \bb = 0$.
\end{lemma}
See \cite{AR} for other formulations and the proof of this theorem. 
First of all, we shall give the proof of Lemma~\ref{lem:L4}.
\begin{proof}[Proof of Lemma~\ref{lem:L4}]
First we write the system of equations~\eqref{e:d-system} in matrix form.
To this end, denote by $Q = (q_{ij})_{i,j\in S}$ the transition matrix for the Markov chain $\eta^*_n$ on $S$, 
and denote column vectors $\ba =(a_1, a_2, \ldots, a_{|S|})^\tra$ and $\bd =(d_1, d_2, \ldots, d_{|S|})^\tra$. Then~\eqref{e:d-system} is equivalent to
\[
(Q-I) \ba = -\bd.
\]
Setting $A=Q-I$ and $\bb = -\bd$, Lemma~\ref{lem:fa} shows that~\eqref{e:d-system} has a solution $\ba$
if and only if any  column vector $\by$ such that $(Q-I)^\tra \by = \0$ satisfies $\by^\tra \bd = 0$. But $(Q-I)^\tra \by = \0$
is equivalent to $\by^\tra Q = \by^\tra$, which implies that $\by = \alpha \bpi$ ($\alpha \in \R$) is a scalar multiple of the (unique) stationary
distribution for $Q$. Thus~\eqref{e:d-system}  
has a solution $\ba$ if and only if $\bpi^\tra \bd = 0$, i.e., $\sum_{i \in S}d_i \pi_i=0$, the special case that $\alpha=0$ is contributing nothing to the condition. 

Finally, we show that any solution $\ba$ to~\eqref{e:d-system} 
  is unique up to translation. 
	Suppose there are two solutions, $\ba'$ and $\ba''$, so that $(Q-I)\ba' = (Q-I)\ba'' = -\bd$;
	 thus $(Q-I)(\ba'-\ba'')=\0$. In other words, $Q (\ba'-\ba'') =  \ba'-\ba''$. As $Q$ is a stochastic matrix, this means that $\ba'-\ba''$
	is a scalar multiple of the eigenvector $(1,1, \ldots, 1)^\tra$ corresponding to eigenvalue~$1$. Thus the components of $\ba'$ and $\ba''$
	differ by a fixed amount.
\end{proof}

\subsection{Corollaries}

A modification of the above argument yields the following statements, with inequalities instead of equality, which will enable us to show that, under appropriate conditions involving $\pi_j$, suitable $b_i$ exist to construct the Lyapunov function $f_\nu$ satisfying appropriate supermartingale conditions.

\begin{lemma}
\label{lem:L2}
Let $u_i \in \R$ for each $i \in S$.
\begin{itemize}
\item[(i)]
Suppose $\sum_{i \in S} u_i \pi_i<0$.
Then there exist $(b_i, i \in S)$ such that
\[ u_i + \sum_{j \in S}(b_j-b_i)q_{ij} < 0, \text{ for all } i \in S.\]
\item[(ii)]
Suppose   $\sum_{i \in S} u_i \pi_i >0$.
Then there exist $(b_i, i \in S)$ such that 
\[ u_i + \sum_{j \in S}(b_j-b_i)q_{ij} > 0, \text{ for all } i \in S.\]
\end{itemize}
\end{lemma}
\begin{proof}
We prove only part~(i); the proof of~(ii) is similar.
Suppose that $\sum_{i \in S} u_i \pi_i = -\eps$ for some $\eps >0$.
Then taking $\eps_i = \frac{\eps}{|S| \pi_i}$
we get $\sum_{i \in S} (u_i + \eps_i ) \pi_i = 0$.
An application of Lemma~\ref{lem:L4}
with $d_i = u_i + \eps_i$ shows that there exist $b_i$ such that
\[ u_i +\eps_i + \sum_{j \in S}(b_j-b_i)q_{ij} = 0, \text{ for all } i \in S,\]
which gives the result since $\eps_i >0$.
\end{proof}

\section{Proof of the constant drift classification}

In this section, we will give a new proof of Theorem \ref{t:drift-constant} using the method of Lyapunov functions. We will apply Theorem \ref{thm:mcrc} and Theorem \ref{thm:mctc} with the Lyapunov functions stated in \eqref{lya_con1} and \eqref{lya_con2}.

\begin{proof}[Proof of Theorem \ref{t:drift-constant}]

For the recurrence part, we will use the Lyapunov  function $g (x, i)$ defined at \eqref{lya_con1}, with suitably chosen $b_i$. First we see that $g (x, i) \to \infty$ as $x \to \infty$.  Thus Theorem \ref{thm:mcrc} shows that the process is recurrent if 
\begin{equation}
\label{rec01}
\E_{x,i} \left[ g(X_{n+1},\eta_{n+1}) - g(X_n,\eta_n) \right] \le 0
\end{equation}
for all sufficiently large $x$. Now suppose $\sum_{i \in S} d_i \pi_i<0$, then we use Lemma \ref{lem:L2} (i) from our Fredholm alternative corollaries, with $u_i = d_i$ to show that we may choose $b_i$ so that 
\begin{equation*}
d_i + \sum_{j \in S} (b_j - b_i) q_{ij}  < 0.
\end{equation*}
Hence from Lemma \ref{lem:calf0} we know the condition \eqref{rec01} is satisfied for $x$ sufficiently large.

For the transience part, this time we will use the Lyapunov function $h_\nu(x, i)$ defined at \eqref{lya_con2} for a small positive value of $\nu$ close to $0$, and apply Theorem \ref{thm:mctc}. We see that the condition in equation \eqref{con:mctc} is satisfied as $h_\nu(x, i)$ is a decreasing function. Hence the process is transient if
\begin{equation}
\label{rec02}
\E_{x,i} \left[ h_\nu(X_{n+1},\eta_{n+1}) - h_\nu(X_n,\eta_n) \right] \le 0
\end{equation}
for all sufficiently large $x$. Now suppose $\sum_{i \in S} d_i \pi_i >0$, using Lemma \ref{lem:L2} (ii) from our Fredholm alternative corollaries, with $u_i = d_i$ we may choose $b_i$ so that 
\begin{equation*}
d_i + \sum_{j \in S} (b_j - b_i) q_{ij}  > 0.
\end{equation*}
Finally, from Lemma \ref{lem:calf01} we know the condition \eqref{rec02} is satisfied for $x$ sufficiently large as we wanted. This completes the proof of the theorem. 
\end{proof}

\section{Proofs of results for Lamperti drift}

The first goal of the section is to give a new proof of the first three points in Theorem \ref{thm:GW} using the method of Lyapunov functions. To prove the whole classification, we should separate the argument in a few parts. 

First, in this subsection, we will prove the conditions for recurrence and transience, by applying Theorem \ref{thm:mcrc} and Theorem \ref{thm:mctc}. 

In the second and the third subsection, we turn our attention to our second and third objectives, the proof of existence and non-existence of moments. 

Lastly in the fourth subsection, we will show the conditions for positive recurrence and null, which are in fact special cases for the existence and non-existence of moments. Combining with the dichotomy in the first subsection will give us the full classification as stated in Theorem \ref{thm:GW}, with the exception of the boundary cases.

For the critical case for null recurrence in Theorem \ref{thm:GW}, we would need a more delicate treatment with a Lyapunov function which grows slower, like $(\log x)^\theta$, $\theta \in (0,1)$. Some general calculation can be found in the book by Menshikov et. al. \cite{MPW}.

\subsection{Proof of recurrence and transience in the Lamperti drift case}

Here is our formal statement to be proved in this subsection.

\begin{theorem} 
\label{thm:GW1}
Suppose that~\eqref{ass:basic} holds, and that~\eqref{ass:p-moments} holds for some $p>2$. Suppose also that~\eqref{ass:q-lim} and~\eqref{ass:drift-lamperti} hold.
Then the following classification applies.
\begin{itemize}
\item If $\sum_{i \in S}(2c_i-s_i^2)\pi_i<0$, then $(X_n,\eta_n)$ is recurrent.
\item If $\sum_{i \in S}(2c_i-s_i^2)\pi_i>0$, then $(X_n,\eta_n)$ is transient.
\end{itemize}
\end{theorem}

\begin{proof}
Using the Lyapunov function $f_\nu(x,i)$ stated in \eqref{lya_lam}, we would like to apply Theorem \ref{thm:mcrc} to get a condition for recurrence. 

Suppose that $\nu >0$, then by Lemma \ref{cal21}, $f_\nu(x,i) \to \infty$ as $x \to \infty$. So we know the process is recurrent if 
\begin{equation}
\label{rec1}
\E_{x,i} \left[ f_\nu(X_{n+1},\eta_{n+1}) - f_\nu(X_n,\eta_n) \right] \le 0
\end{equation}
for all $x$ sufficiently large. Now suppose that $\sum_{i \in S}(2c_i-s_i^2)\pi_i<0$, then there exists $\nu >0$ such that
\begin{equation*}
\sum_{i \in S} \left[2 c_i + (\nu-1) s_i^2 \right] \pi_i<0.
\end{equation*}
Now we use Lemma \ref{lem:L2} (i) from our Fredholm alternative corollaries, with $u_i = 2 c_i + (\nu-1) s_i^2$ to show that we may choose $b_i$ such that
\begin{equation*}
2 c_i + (\nu-1) s_i^2 + \sum_{j \in S} (b_j - b_i) q_{ij} <0.
\end{equation*}
Hence we get
\begin{equation*}
\frac{\nu}{2} x^{\nu-2} \left( 2 c_i + (\nu-1) s_i^2 + \sum_{j \in S} (b_j - b_i) q_{ij} + o(1) \right) \le 0
\end{equation*}
for all $x$ sufficiently large. Finally, apply Lemma \ref{lem:calf} to get our recurrence condition in equation \eqref{rec1} as desired.

For the transient side, this time we take $\nu <0$ and apply Theorem \ref{thm:mctc}. With Lemma \ref{cal21}, we have $f_\nu(x,i) \to 0$ as $x \to \infty$, hence the condition in equation \eqref{con:mctc} is immediately satisfied. So the process is transient if \eqref{rec1} holds for all $x$ sufficiently large. This time we suppose that $\sum_{i \in S}(2c_i-s_i^2)\pi_i >0$, then there exists $\nu <0$ such that
\begin{equation*}
\sum_{i \in S} \left[2 c_i + (\nu-1) s_i^2 \right] \pi_i >0.
\end{equation*}
Now we use Lemma \ref{lem:L2} (ii) from our Fredholm alternative corollaries, with $u_i = 2 c_i + (\nu-1) s_i^2$ to show that we can choose $b_i$ such that
\begin{equation*}
2 c_i + (\nu-1) s_i^2 + \sum_{j \in S} (b_j - b_i) q_{ij} >0.
\end{equation*}
Hence we get
\begin{equation*}
\frac{\nu}{2} x^{\nu-2} \left( 2 c_i + (\nu-1) s_i^2 + \sum_{j \in S} (b_j - b_i) q_{ij} + o(1) \right) \le 0.
\end{equation*}
for all $x$ sufficiently large. Finally, apply Lemma \ref{lem:calf} to get our transience condition in equation \eqref{rec1} as desired. Hence the proof is completed.
\end{proof}

\subsection{Proof of existence of moments}

To obtain existence of moments of hitting times, we apply the following semimartingale result, which is a reformulation of Theorem~1 from~\cite{AIM}, see also \cite{MPW} Corollary 2.7.3.
\begin{lemma}
\label{thm:moments}
Let $W_n$ be an integrable $\cF_n$-adapted stochastic process, taking values in an unbounded subset of $\RP$, 
with $W_0=w_0$ fixed. Suppose that there exist $\delta>0$, $w>0$, and $\gamma<1$ such that for any $n \ge 0$, 
\begin{equation}
\E[W_{n+1}-W_n \mid \cF_n] \le -\delta W_n^\gamma , \text{ on } \{n < \lambda_w\}, \label{eqn1}
\end{equation}
where $\lambda_w=\min\{n \ge 0 : W_n \le w\}$. Then, for any $s \in [0,\frac{1}{1-\gamma}]$, $\E[\lambda_w^s]< \infty$.
\end{lemma}

Now we can give the proof of Theorem~\ref{thm:L2}.

\begin{proof}[Proof of Theorem~\ref{thm:L2}]
Set $W_n := f_\nu ( X_n, \eta_n)$ for $\nu \in (0,p]$; note $W_n \to \infty$ as $X_n \to \infty$.
We aim to show that~\eqref{eqn1} holds with $\gamma= \frac{\nu-2}{\nu} < 1$.
First note that, for $X_n$ sufficiently large,
\begin{align*}
W_n^{\gamma} & =  \left(X_n^\nu+\frac{\nu}{2} a_{\eta_n} X_n^{\nu-2}\right)^{\frac{\nu-2}{\nu}} \\
& =  X_n^{\nu-2}\left(1+\frac{\nu}{2} a_{\eta_n} X_n^{-2}\right)^{\frac{\nu-2}{\nu}} \\
& = X_n^{\nu-2} + O\left(X_n^{\nu-4}\right) ,
\end{align*}
using the fact that $a_{\eta_n}$ is uniformly bounded. In other words, $X_n^{\nu -2 } = W_n^\gamma + o ( W_n^\gamma )$,
so we have from Lemma~\ref{lem:calf} that
\begin{align}
\label{e:511a} 
 \E   [  W_{n+1} - W_n   \mid X_n , \eta_n ] 
= \frac{\nu}{2} W_n^\gamma \left( 2 c_i + (\nu-1) s_i^2 + \sum_{j \in S} (a_j - a_i) q_{ij} \right)  + o(W_n^\gamma) 
 .\end{align}
Take $\nu =  p \wedge 2\theta$ and set $u_i = 2c_i + (\nu-1) s_i^2$; then, by~\eqref{eqn:L2},
\[ \sum_{i\in  S} u_i \pi_i \leq \sum_{i\in  S} \left[ 2 c_i + (2\theta-1) s_i^2 \right]\pi_i < 0 ,\]
so that by Lemma~\ref{lem:L2}(i) we have that the coefficient of $W_n^\gamma$ on the right-hand side
of~\eqref{e:511a} is strictly negative. Hence there exists $\delta >0$ such that
\[
 \E   [  W_{n+1} - W_n   \mid X_n , \eta_n ] \le - \delta W_n^\gamma  , \text{ on } \{ W_n \geq w \},
\]
for some $w$ big enough. 
Note that $\frac{1}{1-\gamma} = \frac{\nu}{2} = \theta \wedge \frac{p}{2}$;
thus we may apply Lemma~\ref{thm:moments} to conclude that $\E [\lambda_w^s]<\infty$ for all $w$ sufficiently large,
for any $s \in \left[0,\theta \wedge \frac{p}{2}\right]$.

It remains to deduce that $\E[\tau_x^s] < \infty$ for all $x$ sufficiently large. But Lemma~\ref{lem:indicator}
shows that $X_n \leq C W_n^{1/\nu}$ for some $C \in \RP$, so $\{W_n \leq w\}$ implies that $\{X_n~\leq~C w^{1/\nu}~\}$.
It follows that $\tau_{Cw^{1/\nu}} \leq \lambda_w$, completing the proof of the \nobreak{theorem}.
\end{proof}

\subsection{Proof of non-existence of moments}

To obtain non-existence of moments of hitting times, we apply the following semimartingale result, which is a variation on Theorem~2 from~\cite{AIM}, see also \cite{MPW} Theorem 2.7.4.
\begin{lemma}
\label{thm:nomo}
Let $Z_n$ be a $\cF_n$-adapted stochastic process taking values in an \nobreak{unbounded} subset of $\RP$. Suppose that there exist finite positive constants $z$, $B$, and $c$ such that, for any $n \ge 0$,
\begin{align}
\E[Z_{n+1}-Z_n \mid \mathcal{F}_n] & \ge -\frac{c}{Z_n}, \text{ on }\{Z_n \ge z\};
\label{infcond1} \\
\E[(Z_{n+1}-Z_n)^2 \mid \mathcal{F}_n] & \le B, \text{ on } \{Z_n \ge z\}.
\label{infcond2}
\end{align}
Suppose in addition that for some $p_0 > 0$, the process $Z_{n \wedge \lambda_z}^{2p_0}$ is a submartingale,
where $\lambda_z=\min\{n \ge 0 : Z_n \le z\}$. Then for any $p > p_0$, we have $\E[\lambda_z^p \mid Z_0 = z_0 ] = \infty$ for any $z_0 > z$.
\end{lemma}

We will apply this result with $Z_n := W_n^{1/\nu} = ( f_\nu (X_n , \eta_n ) )^{1/\nu}$.
Thus we must establish some estimates on the first and second moments of the increments of $Z_n$;
this is the purpose of the next result.
\begin{lemma}
\label{lem:z1}
Suppose that~\eqref{ass:basic} holds, and that~\eqref{ass:p-moments} holds for some $p>2$. Suppose also that~\eqref{ass:q-lim} and~\eqref{ass:drift-lamperti} hold.
Then for any $\nu \in (0, p] $, we have
\begin{align*}
\E_{x,i}[Z_{n+1}-Z_n ] 
&= \frac{c_i}{x} + \frac{1}{2x} \sum_{j \in S} (b_{j}-b_i)q_{ij}  + o \left(x^{-1} \right); \text{ and } \\
\E_{x,i}[(Z_{n+1}-Z_n)^2 ] &\le B,
\end{align*}
where $B$ is a constant.
\end{lemma}
\begin{proof}
Again we define the event   $E_n := \{ | \Delta_n | \leq X_n^\zeta \}$ for $\zeta \in (0,1)$; then
\begin{align}
 \E_{x,i}[Z_{n+1}-Z_n ]  = \E_{x,i} \left[ \left(Z_{n+1}-Z_n \right) \2{ E_n}  \right] + \E_{x,i} \left[ \left( Z_{n+1}-Z_n \right) \2{ E^\rc_n}   \right]. \label{cal14}
\end{align}
For the first term on the right-hand side of~\eqref{cal14}, we first notice that for $x$ large enough  
\begin{align*}
f_\nu^{1/\nu}(x,i) = x \left(1 + \frac{\nu}{2}b_i x^{-2} \right)^{1/\nu} = x + \frac{b_i}{2x} + O(x^{-3}).
\end{align*}
Also, given $(X_n,\eta_n)=(x,i)$, on $E_n$  we have $|X_{n+1}-X_n| \le x^\zeta$ so that
$Z_{n+1} \2{E_n} = X_{n+1} + \frac{b_{\eta_{n+1}}}{2x} +o(x^{-1})$. As a result we get
\begin{align}
& \quad \E_{x,i} \left[ \left(Z_{n+1}-Z_n \right) \2{ E_n}  \right] \nonumber \\ &= \E_{x,i} \left[(X_{n+1}-X_n) \2{ E_n}  \right] + \frac{1}{2x} \E_{x,i} \left[(b_{\eta_{n+1}}-b_{\eta_n})\2{ E_n}  \right] +o(x^{-1}) \nonumber \\
&= \frac{c_i}{x} + \frac{1}{2x} \sum_{j \in S} (b_{j}-b_i)q_{ij}  + o \left(x^{-1} \right) \label{cal80}
\end{align}
where in the last equality we used Lemma~\ref{cal4} and the $r=0$ case of Lemma~\ref{cal11} for the first and second term respectively. On the other hand, to bound  
$\E_{x,i} \left[ \left(Z_{n+1}-Z_n \right) \2{ E^\rc_n} \right]$, we can just mimic  the proof of Lemma~\ref{cal17}, inserting additional
powers of $1/\nu$, to obtain
\[ 
\E_{x,i} \left[ \left(Z_{n+1}-Z_n \right) \2{ E^\rc_n} \right] = o(x^{-1}) ,
\]
which combined with~\eqref{cal80} gives the first statement in the lemma. For the second moment,
given $(X_n, \eta_n) = (x,i)$ we have
\begin{align*}
(Z_{n+1}-Z_n)^2 \2{ E_n} &\le (X_{n+1}-X_n)^2 \2{ E_n} + \frac{|X_{n+1}-X_n|}{ x}  |b_{\eta_{n+1}}-b_{\eta_n}| \2{ E_n} + O(x^{-1}) \\
&\le (X_{n+1}-X_n)^2 + o(1).
\end{align*}
 Taking expectations, we obtain 
\begin{align*}
\E_{x,i}\left[ (Z_{n+1}-Z_n)^2 \2{ E_n} \right] \le C,
\end{align*}
for some $C \in \RP$. On the other hand, we follow the proof of Lemma~\ref{cal17}, inserting powers of $2/\nu$ and $1/\nu$ respectively, to get  
\begin{align*}
\E_{x,i}\left[ (Z_{n+1}-Z_n)^2 \2{ E^\rc_n} \right] &= \E_{x,i}[(Z_{n+1}^2-Z_n^2)\2{ E^\rc_n} ] -2\E_{x,i}[Z_n(Z_{n+1}-Z_n)\2{ E^\rc_n} ] \\
&= o(1).  
\end{align*}
Combining these estimates the second statement in the lemma follows.
\end{proof}

Now we can complete the proof of Theorem~\ref{thm:L3}.

\begin{proof}[Proof of Theorem~\ref{thm:L3}.]
Take $\nu =   2 \theta$. We will apply Lemma~\ref{thm:nomo} with $Z_n=(f_\nu(X_n, \eta_n))^{1/\nu}$ and $p_0 = \frac{\nu}{2}$; we must verify~\eqref{infcond1} and~\eqref{infcond2}, and that $Z_{n \wedge \lambda_z}^{\nu}$ is a submartingale. For the latter condition, it suffices to show that
\[
\E[ f_\nu (X_{n+1}, \eta_{n+1}) - f_\nu (X_n, \eta_n) \mid  X_n, \eta_n ] \ge 0, \text{ on } \{Z_n>z\}, 
\]
which follows from the $\nu = 2 \theta$ case of Lemma~\ref{lem:calf} with hypothesis~\eqref{eqn:L3}.

Of the remaining two conditions,~\eqref{infcond2} follows immediately from  the second statement in Lemma~\ref{lem:z1}, provided $\nu \leq p$, i.e., $\theta \leq \frac{p}{2}$. From the first statement in Lemma~\ref{lem:z1}, we have that for all $x$ sufficiently large
\begin{align}
\label{infcond1a}
 \E_{x,i} [Z_{n+1}-Z_n  ]  
= \frac{1}{x} \left( c_i + \frac{1}{2} \sum_{j \in S} (b_{j}-b_i)q_{ij} + o(1) \right)   
 \ge - \frac{C_1}{x}   ,
\end{align}
for all $i$ and all $x$ sufficiently large, where $C_1 \in \RP$ depends on the $c_i$ and $b_i$. Now Lemma~\ref{cal21} implies that $Z_n \leq C_2 X_n$ for some $C_2 >0$, so we deduce condition~\eqref{infcond1} from~\eqref{infcond1a}.
\end{proof}

\subsection{Complete classification}

To complete the classification in Theorem \ref{thm:GW}, we need to classify the different conditions for positive recurrent and null. Hence we should prove the following theorem, which is a simple consequence of the moment existence and non-existence results in the previous subsections.

\begin{theorem} 
\label{thm:GW2}
Suppose that~\eqref{ass:basic} holds, and that~\eqref{ass:p-moments} holds for some $p>2$. Suppose also that~\eqref{ass:q-lim} and~\eqref{ass:drift-lamperti} hold.
Then the following classification applies.
\begin{itemize}
\item If $\sum_{i \in S}(2c_i+s_i^2)\pi_i<0$, then $(X_n,\eta_n)$ is positive recurrent.
\item If $\sum_{i \in S}(2c_i+s_i^2)\pi_i>0$, then $(X_n,\eta_n)$ is null.
\end{itemize}
\end{theorem}

\begin{proof}
Taking $\theta=1$ in Theorem \ref{thm:L2}, we recover the condition for the first moment to exist, applying a technical Lemma 2.6.1 in \cite{MPW} will give us the positive-recurrent part. Take $\theta=1$ in Theorem \ref{thm:L3}, we get the condition for the first moment to not exist, which means the process is null. 
\end{proof}

If we group the results from Theorem \ref{thm:GW1} and Theorem \ref{thm:GW2}, we proved Theorem \ref{thm:GW} using the Lyapunov function method, completing the classification of positive recurrence and transience for the Lamperti drift case.

\section{Proofs of results for generalized Lamperti drift}

For this section, we turn our attention to the last and most subtle case with the generalized Lamperti drift, and give a complete classification for all situations for the half strip problem. Our main goal is to prove Theorem \ref{thm:L1}.

\subsection{Transformation to Lamperti drift case}

As the structure of $(X_n, \eta_n)$ is quite complicated, in the first step of the proof, we want to transform $X_n$ to $\tX_n$ so that we can have a simpler form for the drifts and conditions. An enlightening transformation would be $(\tX_n, \eta_n ) = (X_n +a_{\eta_n} , \eta_n)$, where $a_i$ are the solution of the system of equations $d_i+\sum_{j \in S}(a_j-a_i)q_{ij}=0$ for all $i \in S$. The intuition behind this is that the transformation pulls or pushes each line separately in a way such that the effects of $d_i$ are eliminated, turning all $d_i=0$ after the transformation, i.e., the constant components of the drifts are eliminated; then we can apply the results in Section~\ref{s:ld}, once we have at hand increment moment estimates for the transformed process $(\tX_n, \eta_n)$. A few items of technicality that need to be handled are listed as follows.

\begin{enumerate}
\item Existence and uniqueness up to translation of $a_i$,
\item Preservation of the classification under the transformation,
\item Changed moments. 
\end{enumerate}

For the first item, for $d_i$, $i \in S$ with $\sum_{i \in S} \pi_i d_i = 0$ as specified under assumption~\eqref{ass:drift-gl}, choose $a_i$, $i \in S$ as guaranteed by Lemma~\ref{lem:L4}, so that $d_i+\sum_{j \in S}(a_j-a_i)q_{ij}=0$; since we are free to translate the $a_i$ by any constant, we may (and do) suppose here that $a_i \geq 0$ for all $i \in S$.

Let $\Sigma' = \{ (x+a_i, i) : (x,i) \in \Sigma \}$ denote the state space of the transformed process; then $\Sigma'$ is a locally finite subset of $\RP \times S$ with unbounded lines $\Lambda'_k = \{ x \in \RP : (x,k) \in \Sigma' \}$. The map $(x,i) \mapsto (x+a_i , i)$ is a bijection, so the Markov chain $(\tX_n, \eta_n )$ has precisely the same abstract structure as the Markov chain $(X_n, \eta_n)$, in particular, the transformed process is (positive) recurrent if and only if the original process is (positive) recurrent, and so on, see Theorem~\ref{thm:transformation} for a formal statement. Thus to obtain results for the process $(X_n, \eta_n)$ it is sufficient to prove results for the transformed process $(\tX_n, \eta_n )$.

\subsection{Preservation of the recurrence classification after transformation}

To see the preservation of the recurrence classification after the transformation, we want to establish the following theorem. 

Suppose we have a Markov chain $Z_n$ on a countable space $\Sigma$. Define $\lambda_x=\min\{n \ge 0 : Z_n = x\}$ and $\lambda_A=\min_{x \in A}\lambda_x$. Then we call the Markov chain $Z_n$ $s$-recurrent if, for every $x \in \Sigma$, and any finite, non-empty $A \subset \Sigma$, $\E_x[\lambda_A^s]<\infty$. The following theorem will give all the facts we need for both the recurrence classification, and the moment existence and non-existence.

\begin{theorem}
\label{thm:transformation}
For any irreducible Markov chain $Z_n$ on a countable space $\Sigma$, and any bijective function $f:\Sigma \to \Sigma'$, denote $Z_n'=f(Z_n)$, then we have all of the following.

\begin{enumerate}
\item $Z_n'$ is an irreducible Markov chain on $\Sigma'$,
\item $Z_n$ is recurrent if and only if $Z_n'$ is recurrent,
\item $Z_n$ is positive recurrent if and only if $Z_n'$ is positive recurrent,
\item $Z_n$ is $s$-recurrent if and only if $Z_n'$ is $s$-recurrent.
\end{enumerate}
\end{theorem}

\begin{proof}
As $Z_n$ is a Markov chain, then $\Pr(Z_{n+1}=z \mid Z_1=z_1, Z_2=z_2, \ldots, Z_n=z_n) = \Pr(Z_{n+1}=z\mid Z_n=z_n)$. By applying the transformation $f$ on every point on the space, we have $\Pr(Z'_{n+1}=f(z) \mid Z'_1=f(z_1), Z'_2=f(z_2), \ldots, Z'_n=f(z_n)) = \Pr(Z'_{n+1}=f(z) \mid Z'_n=f(z_n))$. So for any $z',z'_1,z'_2, \ldots, z'_n \in \Sigma'$, we can find $z=f^{-1}(z'), z_1= f^{-1}(z'_1), z_2= f^{-1}(z'_2), \ldots, z_n= f^{-1}(z'_n)$ so that the Markov property preserves. Also, for any state $z'_1, z'_2 \in \Sigma'$, we can find $z_1= f^{-1}(z'_1)$ and $z_2= f^{-1}(z'_2)$. If $Z_n$ is irreducible, then $z_1$ and $z_2$ communicates, and so does $z'_1$ and $z'_2$. So $Z'_n$ is also irreducible. The other direction follows as $f^{-1}$ is also a bijection.

Next, for any $y \in \Sigma'$, there exists $x \in \Sigma$ such that $x=f^{-1}(y)$. Then if $Z_n$ is recurrent, we have $\Pr[Z_n=x \text{ } i.o.]=1$. This means $\Pr[Z'_n=y \text{ } i.o.]=1$ by applying the transformation. So we get $Z'_n$ is recurrent. The other directing follows similarly. Hence we have proved the second statement.

The third claim is in fact a special case of the fourth claim with $s=1$, so we will just prove the fourth claim instead.

Suppose $Z_n$ is $s$-recurrent. Let $\lambda'_B = \min\{n \ge 0 : Z'_n \in B\}$. Then
\begin{align*}
\lambda'_B &= \min\{n \ge 0 : f(Z_n) \in B\} \\
&= \min\{n \ge 0 : Z_n \in f^{-1}(B)\} \\
&= \lambda_{f^{-1}(B)}
\end{align*}
so $\E_y(\lambda'_B)^s = \E_{f^{-1}(y)} \lambda_{f^{-1}(B)}^s < \infty$, i.e. $Z'_n$ is $s$-recurrent. The argument clearly goes both ways.

Hence we have proved all claims in the Theorem.
\end{proof}

Now with the first three claims of Theorem \ref{thm:transformation}, we know the structure of the Markov chain and the classification of positive recurrent, null recurrent and transient is preserved under the transformation.

\subsection{Calculation of the transformed moments}

For the increment moments of the transformed process, we use the notation
\begin{align*}
\widetilde{\mu}_i(y) & := \E[\widetilde{X}_{n+1}-\widetilde{X}_n \mid \widetilde{X}_n=y,\eta_n=i], \\
\widetilde{\sigma}_i^2(y) & := \E[(\widetilde{X}_{n+1}-\widetilde{X}_n)^2 \mid \widetilde{X}_n=y,\eta_n=i].
\end{align*}
\begin{lemma}
\label{lem:transform}
Suppose that~\eqref{ass:basic} holds, and that~\eqref{ass:p-moments} holds for some $p>2$. Suppose also that~\eqref{ass:q-lim-gl} and~\eqref{ass:drift-gl} hold.
Define $a_i$, $i \in S$ to be a solution to~\eqref{e:d-system} with $a_i \geq 0$ for all $i$,
whose existence is guaranteed by Lemma~\ref{lem:L4}.
Either (i) set $\delta_4 = 0$; or (ii) suppose that~\eqref{ass:q-lim-gl+} and~\eqref{ass:drift-gl+} hold and set $\delta_4 = \delta_2 \wedge \delta_3 \in (0,1)$. 
For $i \in S$, define 
\begin{equation}
\label{e:new_constants}
  c_i := e_i+\sum_{j \in S}a_j\gamma_{ij}, ~~\text{and}~~ s_i^2 := t_i^2 +2\sum_{j \in S}a_jd_{ij} + \sum_{j \in S}(a_j^2-a_i^2)q_{ij}. \end{equation}
Then we have that, as $x \to \infty$,
\begin{align*}
\widetilde{\mu}_i(x)  = \frac{c_i}{x} +o(x^{-1-\delta_4}), ~~\text{and}~~ \widetilde{\sigma}_i^2(x)  = s_i^2 + o(x^{-\delta_4}) .
\end{align*}
\end{lemma}
\begin{proof}
For concreteness, we give the proof when~\eqref{ass:q-lim-gl+} and~\eqref{ass:drift-gl+} hold; in the other
case the argument is the same but with $\delta_4$ set to $0$ throughout.
For the first moment, we have
\begin{align*}
\widetilde{\mu}_i(x)
 &= \E_{x-a_i,i}[X_{n+1}-X_n]+\sum_{j \in S}\E_{x-a_i,i}\left[\left(a_{\eta_{n+1}}-a_{\eta_n}\right) \1 {a_{\eta_{n+1}}=j }\right] \\
&= \mu_i(x-a_i)+\sum_{j \in S}(a_j-a_i)q_{ij}(x-a_i).
\end{align*}
Now using hypothesis~(a) in~\eqref{ass:drift-gl+} and~\eqref{ass:q-lim-gl+} we obtain
\begin{align*}
\widetilde{\mu}_i(x) &= d_i + \frac{e_i}{x-a_i}+\sum_{j \in S}(a_j-a_i)q_{ij}+\sum_{j \in S}(a_j-a_i)\frac{\gamma_{ij}}{x-a_i} +o((x-a_i)^{-1-\delta_4}) \\
&= d_i +\sum_{j \in S}(a_j-a_i)q_{ij} + \frac{e_i}{x} + \frac{1}{x} \sum_{j \in S} a_j\gamma_{ij} +o(x^{-1-\delta_4}),
\end{align*}
since $\sum_{j \in S}\gamma_{ij}=0$. By hypothesis~(d)  in~\eqref{ass:drift-gl} and choice of the $a_i$
(cf Lemma~\ref{lem:L4}), the constant term here vanishes, so we obtain the
expression for $\widetilde{\mu}_i(x)$ in the lemma.

For the second moment, we have
\begin{align*}
\widetilde{\sigma}_i^2(x) 
&= \E_{x-a_i,i} \left[(X_{n+1}-X_n)^2 \right]+2\E_{x-a_i,i}\left[ (a_{\eta_{n+1}}-a_{\eta_n})(X_{n+1}-X_n)\right] \\
& \quad +\E_{x-a_i,i}\left[a_{\eta_{n+1}}^2-a_{\eta_n}^2 \right] \\
&= s_i^2 + 2\sum_{j \in S}(a_j-a_i)\mu_{ij}(x-a_i)+\sum_{j \in S}(a_j-a_i)^2 q_{ij}(x-a_i) . \end{align*}
Then using hypothesis~(c) in~\eqref{ass:drift-gl+} and~\eqref{ass:q-lim-gl+} we obtain
\begin{align*}  
\widetilde{\sigma}_i^2(x) & = s_i^2 + 2\sum_{j \in S}(a_j-a_i)d_{ij} + \sum_{j \in S}(a_j-a_i)^2 q_{ij} +o(x^{-\delta_4})  
\\
&= s_i^2 + 2\sum_{j \in S}a_j d_{ij} + \sum_{j \in S}(a_j^2-a_i^2) q_{ij} -2a_i \biggl( d_i + \sum_{j \in S}(a_j-a_i) q_{ij} \biggr)   +o(x^{-\delta_4}),
\end{align*}
which gives the result after once again using the fact that $d_i+\sum_{j \in S}(a_j-a_i) q_{ij}=0$.
\end{proof}

\subsection{Proof of recurrence classification}

Armed with our transformation of the process, we can now use the results in Section~\ref{s:ld} to complete the proofs of the
theorems in Section~\ref{s:gld}.

\begin{proof}[Proof of Theorem~\ref{thm:L1}]
Lemma~\ref{lem:transform} shows that if $(X_n, \eta_n)$ satisfies the conditions of Theorem~\ref{thm:L1},
then $(\tX_n, \eta_n)$ satisfies the conditions of Theorem~\ref{thm:GW}
with $c_i$ and $s_i^2$ as given by~\eqref{e:new_constants}.
Theorem~\ref{thm:GW} shows that the process is transient if
\begin{align*}
0 < \sum_{i \in S}[2c_i-s_i^2]\pi_i =& \sum_{i \in S}\left[2e_i + 2\sum_{j \in S}a_j\gamma_{ij}-\left(t_i^2 +2\sum_{j \in S}a_jd_{ij} + \sum_{j \in S}(a_j^2-a_i^2)q_{ij}\right)\right]\pi_i \\
=& \sum_{i \in S}\left(2e_i -t_i^2 + 2\sum_{j \in S}a_j(\gamma_{ij}-d_{ij})\right)\pi_i -\sum_{i \in S}\sum_{j \in S}(a_j^2-a_i^2)q_{ij}\pi_i ,
\end{align*}
using the expressions at~\eqref{e:new_constants}. Note that the final term here vanishes, because
\begin{align*}
\sum_{i \in S}\sum_{j \in S}(a_j^2-a_i^2)q_{ij}\pi_i 
&= \sum_{j \in S}a_j^2\sum_{i \in S}q_{ij}\pi_i - \sum_{i \in S}a_i^2\pi_i\sum_{j \in S}q_{ij} \\
&= \sum_{j \in S}a_j^2\pi_j - \sum_{i \in S}a_i^2\pi_i =0,
\end{align*}
using the fact that $\bpi$ is the stationary distribution for $(q_{ij})$. This gives the condition for transience stated in Theorem~\ref{thm:L1}.

Similarly,  the condition for positive recurrence is
\begin{align*}
0 > \sum_{i \in S}[2c_i + s_i^2]\pi_i = & \sum_{i \in S}\left[2e_i + 2\sum_{j \in S}a_j\gamma_{ij}+\left(t_i^2 +2\sum_{j \in S}a_jd_{ij} + \sum_{j \in S}(a_j^2-a_i^2)q_{ij}\right)\right]\pi_i \\
=& \sum_{i \in S}\left(2e_i +t_i^2 + 2\sum_{j \in S}a_j(\gamma_{ij}+d_{ij})\right)\pi_i +\sum_{i \in S}\sum_{j \in S}(a_j^2-a_i^2)q_{ij}\pi_i \\
=& \sum_{i \in S}\left(2e_i +t_i^2 + 2\sum_{j \in S}a_j(\gamma_{ij}+d_{ij})\right)\pi_i ,
\end{align*}
which gives the condition for positive recurrence in the theorem.

\subsection{Proofs of existence and non-existence of moments}

The case for null recurrence and at the critical point follows accordingly by the same calculation.
\end{proof}

\begin{proof}[Proof of Theorem~\ref{thm:L4}]
The proof is analogous to the proof of Theorem~\ref{thm:L1}, this time applying Theorem~\ref{thm:L2} to the transformed process.
\end{proof}

\begin{proof}[Proof of Theorem~\ref{thm:L5}]
This time we apply Theorem~\ref{thm:L3} to the transformed process.
\end{proof}

\chapter{Examples, applications and simulations}
\label{ch:exphs}

To finish this part of the thesis, we present an application of our results to a simple model of a \emph{correlated random walk} and some more complicated numerical examples to see the delicacy of the phase transition.

\section{Correlated random walk}
\emph{Correlated random walk} is a type of random walk that remembers a fixed number of previous steps of its past trajectory (except the first few steps, which may affected by less steps due to lack of history). There is a long list of literature in which this model is studied with various names by different researchers: as `persistent random walks' by F\"urth \cite{RF}, `correlated random walks' by Gillis \cite{JG}, `random walks with restricted reversals' by Domb and Fisher \cite{DF}, and, recently, `Newtonian random walks' by Lenci \cite{ML}.  The correlated random walk also leads to the telegrapher's equation in the scaling limit under suitable rescaling, see Goldstein \cite{SG} and Kac \cite{MK}. For some recent work on correlated random walk and related models, see \cite{AM, CR, HS, ST}. Some applications or motivation for studying these models can be found in \cite{RF2} on physical Brownian motion  and \cite{HED} on models for molecular configurations. Some background material can be found in \cite{BDH}.

Now we formally state the model. Suppose that a particle performs a random walk on $\ZP$ with a short-term memory: the distribution of $X_{n+1}$ depends not only on the current position $X_n$, but also on the `direction of travel' $X_n - X_{n-1}$. Formally, $(X_n , X_n - X_{n-1} )$ is a Markov chain on $\ZP \times S$ with $S = \{ -1 , +1 \}$, with
\[ \Pr [ (X_{n+1},\eta_{n+1} ) = (x+j, j) \mid (X_n , \eta_n ) = (x, i ) ] = q_{ij} (x) , ~\text{for} ~ i,j\in S. \]
Then for $i \in S$,
\[ \mu_i (x) = \E [ X_{n+1} - X_n \mid (X_n, \eta_n ) = (x, i) ] = q_{i,+1} (x) - q_{i, -1} (x) .\]
The simplest model has $q_{ii} (x) = q > 1/2$ for $x \geq 1$, so the walker has a 
tendency to continue in its direction of travel. 

More generally,  suppose that for $q \in (0,1)$ and constants $c_{-1}, c_{+1} \in \R$ and $\delta>0$,
\begin{equation}
\label{e:crw}
 q_{ij} (x) = \begin{cases} 
q  + \frac{i c_i }{2x} + O (x^{-1-\delta} ) & \text{if } j =i; \\
1 - q  - \frac{i c_i }{2x} + O (x^{-1-\delta} ) & \text{if } j \neq i.
\end{cases}
\end{equation}
Here is the recurrence/transience classification for this model, which includes as the special case $q=1/2$ the recurrence
classification in Corollary~3.1 of~\cite{AW}.

\begin{theorem}
\label{thm:Example}
Consider the correlated random walk specified by~\eqref{e:crw}. Let $c = (c_{+1} + c_{-1} ) /2$.
If $c < -q$, then the walk is positive recurrent. If $c > q$, then the walk is transient. If $|c| \le q$, then the walk is null recurrent.
\end{theorem}

We can also achieve results for moment existence and non-existence for the correlated random walk model.

\begin{theorem}
\label{thm:Example2}
Consider the correlated random walk specified by~\eqref{e:crw}. Let $c = (c_{+1} + c_{-1} ) /2$. Then $\E [\tau_x^\theta] < \infty$ if $\theta < \frac{1}{2} -\frac{c}{2q}$ while $\E [\tau_x^\theta] = \infty$ if $\theta > \frac{1}{2} -\frac{c}{2q}$.
\end{theorem}

\section{Proofs of theorems on correlated random walk}
In this section we complete the proof of Theorem~\ref{thm:Example} and Theorem~\ref{thm:Example2} on correlated random walk given in the last section.

\begin{proof}[Proof of Theorem~\ref{thm:Example}]
Note first that $q_{ii} = q$ and $q_{ij} = 1-q$ for $j \neq i$; hence $\pi = (\frac12,\frac12)$.
By direct calculation, we get 
$\mu_{i}(x)=i (2q-1) + \frac{c_i}{x} + O (x^{-1-\delta})$. This gives $d_{i}= i (2q-1)$ and $e_{i}=c_i$. Now we want to solve the system of equations~\eqref{e:d-system} for $a_{i}$,
which amounts to 
\[
1-2q+(a_1-a_{-1})(1-q)=0.
\]
Since the solution $(a_i)$ is unique up to translation, without loss of generality we can choose $a_{-1}=0$ and
then we get $a_{+1} =\frac{2q-1}{1-q}$. Next we observe that $d_{ii}=q$,
while if $i \neq j$ we have  $d_{ij}=1-q$; also, 
 $\gamma_{ij}= \frac{j c_i}{2}$ and $t_i^2=1$. 
Now we calculate 
\begin{align*}
 \sum_{i \in S}\biggl( 2e_i+2\sum_{j \in S}a_j \gamma_{ij} \biggr) \pi_i & = e_{+1} + e_{-1} +a_{+1} \gamma_{+1,+1} + a_{+1} \gamma_{-1,+1}   = \frac{c_{+1} + c_{-1}}{2(1-q)} ; \text{ and} \\
\sum_{i \in S} \biggl( t_i^2+2\sum_{j \in S}a_j d_{ij} \biggr) \pi_i & = 1 + \left( a_{+1} d_{+1,+1} + a_{+1} d_{-1,+1} \right) =\frac{q}{1-q}. 
\end{align*}
Then applying Theorem~\ref{thm:L1}  we obtain the desired result.
\end{proof}
\begin{proof}[Proof of Theorem~\ref{thm:Example2}]
With the calculation in the last proof, apply Theorem~\ref{thm:L4} and Theorem~\ref{thm:L5} gives the desirable results.
\end{proof}
\section{Numerical examples and simulations}

Now we will give some simulations of the correlated random walk.
For the first example, we consider a one-step correlated random walk, with the transition probabilities as in the following Figure \ref{fig_sim}, where the `+1' line represents the case where the last step was going to the right and the line `-1' means the last step was going to the left. In this sense we can model the different behaviour of the transition probabilities (i.e. the correlation).

\begin{figure}[H]
\color{blue}
\begin{center}
\includegraphics[width=13cm, angle=0, clip = true]{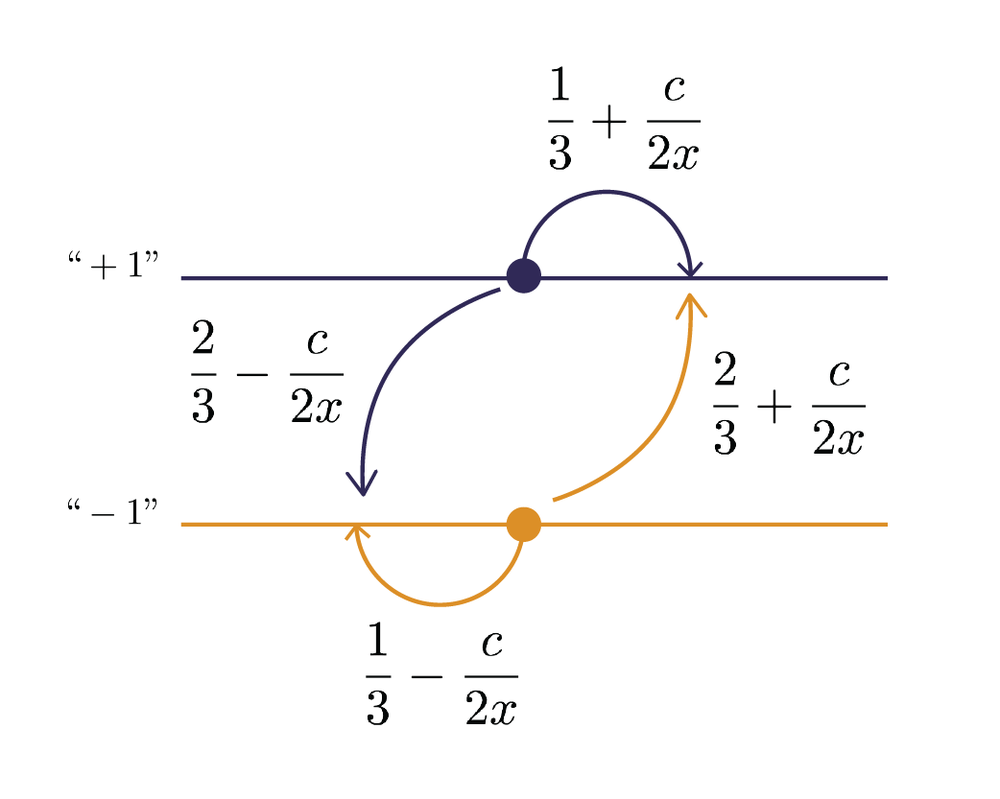}
\caption{An example of a one-step correlated random walk.}
\label{fig_sim}
\end{center}
\end{figure}

We now conduct two simulations of $10^3$ steps of the walk, with different values of $c=1$ (Figure \ref{fig_simr} left) and $c=-1$ (Figure \ref{fig_simr} right). The horizontal axis represents the number of steps $n$ and the vertical axis represents the value of $X_n$. The colour of the line segment represents which line the process is from, with the same colouring scheme shown in Figure \ref{fig_sim}. 

\newpage

\begin{figure}[H]
\color{blue}
\begin{center}
\includegraphics[width=15cm, angle=0, clip = true]{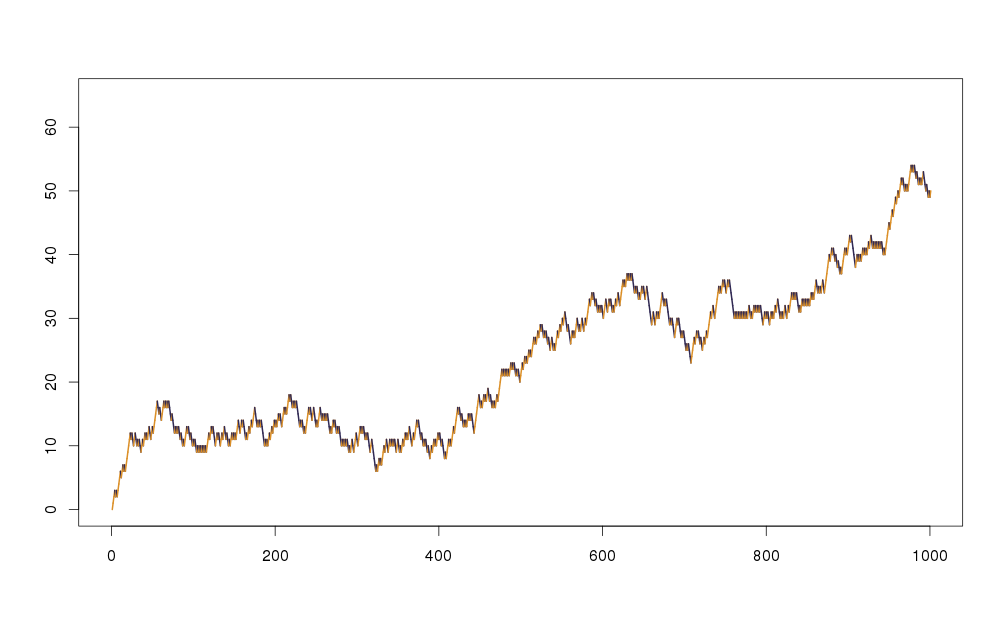}
\qquad
\includegraphics[width=15cm, angle=0, clip = true]{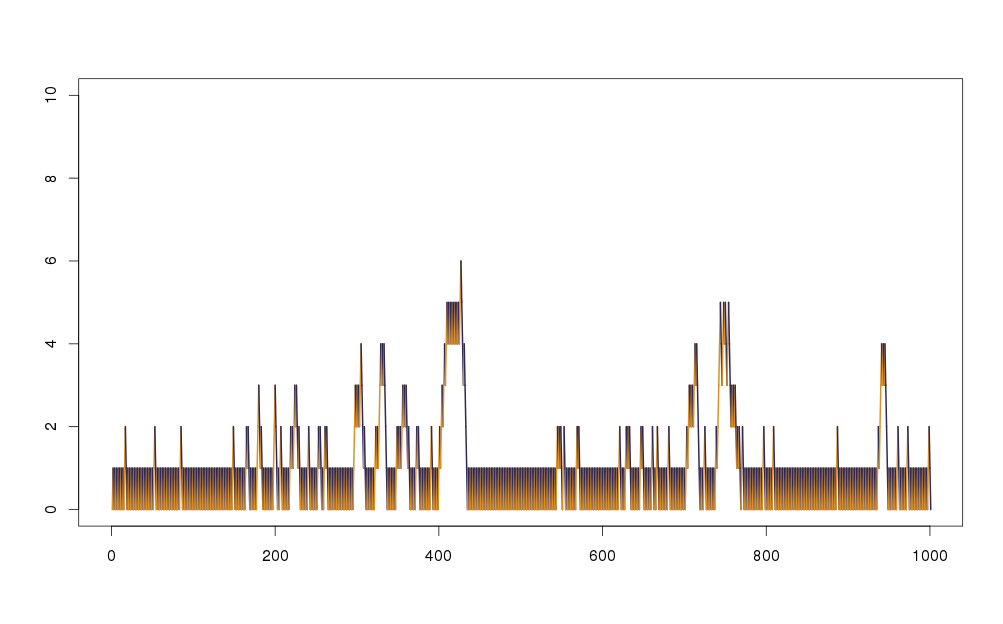}

\caption{ Two simulations of $10^3$ steps of one-step correlated random walks, as an application of the half strip model. Top: $c=1$, bottom: $c=-1$.}
\label{fig_simr}
\end{center}
\end{figure}

\newpage

You can see the two simulations give very different behaviour in this walk. In fact, the left picture in Figure \ref{fig_simr} shows the transient case and the right one shows the positive-recurrent case. The phase transitions of $c$ are actually as the following, transient for $c>\frac{1}{3}$, positive recurrent for $c<-\frac{1}{3}$ and null recurrent for $|c| \le \frac{1}{3}$ from the application of Theorem~\ref{thm:Example}. Moreover, if we apply Theorem~\ref{thm:Example2}, we know $\E [\tau_x^\theta] < \infty$ if $\theta < \frac{1}{2} -\frac{3c}{2}$ while $\E [\tau_x^\theta] = \infty$ if $\theta > \frac{1}{2} -\frac{3c}{2}$.

Here is another example of two-steps correlated random walk, as the transition probabilities as shown in the following Figure \ref{fig_sim2}, where this time we will have four lines, spanning all the combinations of the direction of the last two steps.

\begin{figure}[H]
\color{blue}
\begin{center}
\includegraphics[width=13cm, angle=0, clip = true]{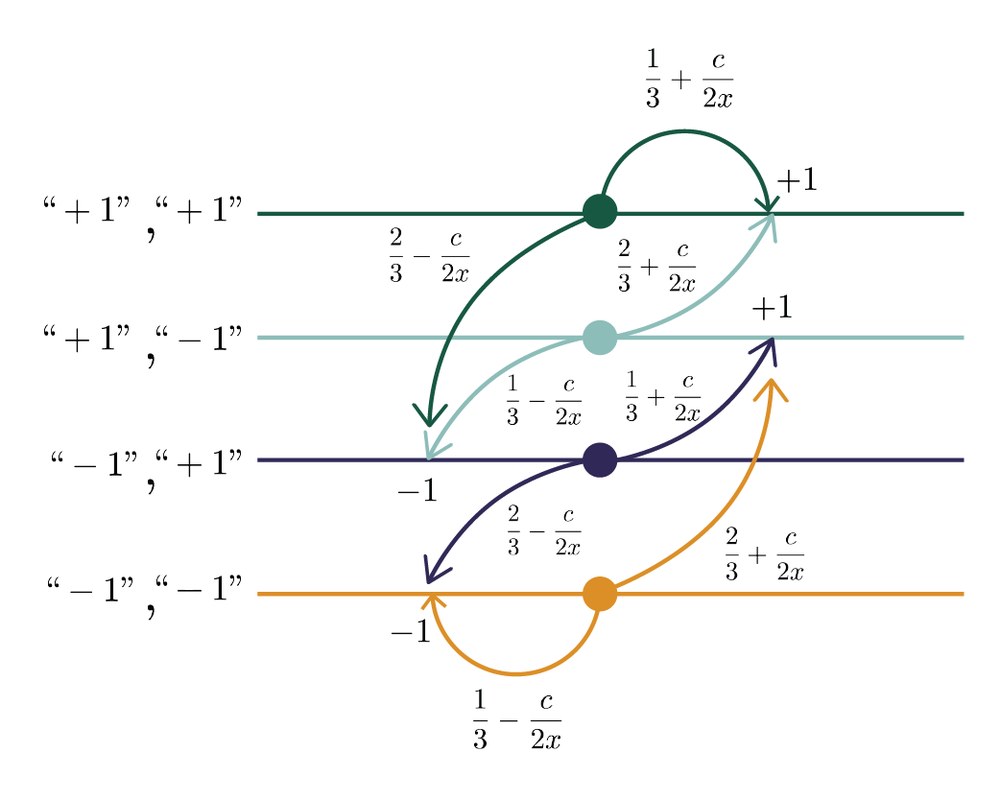}
\caption{An example of a two-steps correlated random walk.}
\label{fig_sim2}
\end{center}
\end{figure}

Again, for different value of $c$, we will get different results of classification.

\newpage

\begin{figure}[H]
\color{blue}
\begin{center}
\includegraphics[width=15cm, angle=0, clip = true]{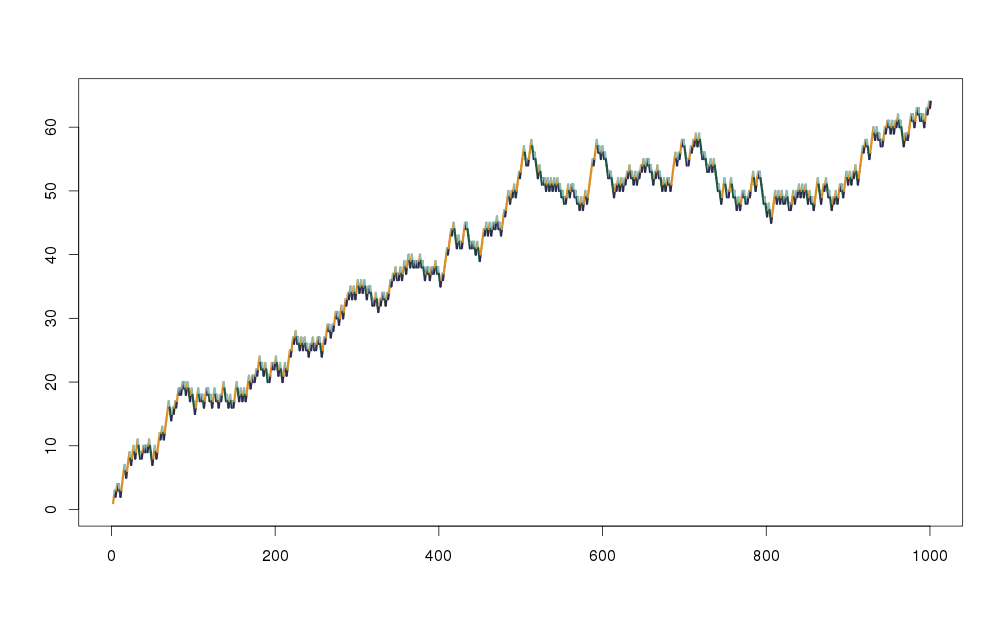}
\qquad
\includegraphics[width=15cm, angle=0, clip = true]{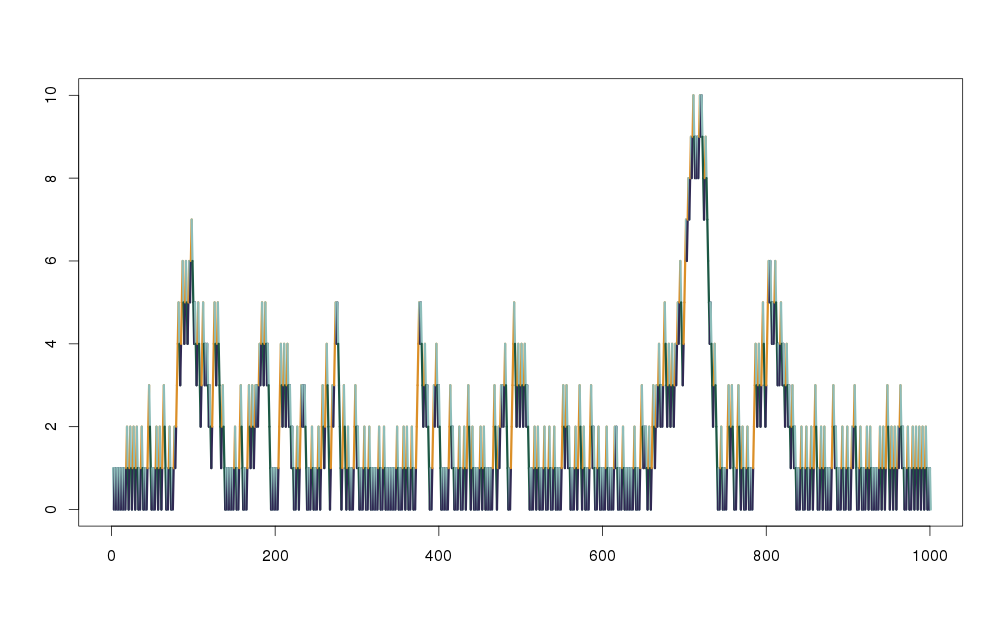}

\caption{ Two simulations of $10^3$ steps of two-steps correlated random walks, as an application of the half strip model. Top: $c=1$, bottom: $c=-1$.}
\label{fig_sim2r}
\end{center}
\end{figure}

\newpage

The random walk is transient in the left picture in Figure \ref{fig_sim2r}, with the value $c=1$ while the walk is positive recurrent for the right picture in Figure \ref{fig_sim2r}, with the value $c=-1$. The phase transitions of $c$ in this case are transient for $c>\frac{1}{3}$, positive recurrent for $c<-\frac{1}{3}$ and null recurrent for $|c| \le \frac{1}{3}$, from the application of  Theorem~\ref{thm:L1}. Moreover, if we apply Theorem~\ref{thm:L4} and Theorem~\ref{thm:L5}, we know $\E [\tau_x^\theta] < \infty$ if $\theta < \frac{1}{2} -\frac{3c}{2}$ while $\E [\tau_x^\theta] = \infty$ if $\theta > \frac{1}{2} -\frac{3c}{2}$.

An important observation is that we have the same phase boundary as in the last example with only one step correlated random walk. This is not a mere coincidence. In fact, a careful calculation shows the same result for any $n$-step correlated random walk for any positive integer $n$, irrelevant to how we assign the favourableness to stick with or change direction with a certain pattern of the previous steps, for a symmetric and balance design. It would be quite clumsy to state the formal statement here and we shall leave the reader to discover the scintillating calculation.

\chapter*{Part II: \\ Centre of Mass of Random Walks}
\addcontentsline{toc}{chapter}{Part II: Centre of Mass of Random Walks}

\chapter{Notation, preliminaries and prerequisites}

\section{Literature review}

In physics, the center of mass of a distribution of mass in space is defined as the unique point that the weighted relative position of the distributed mass sums to zero and so the distribution of mass is balanced around the center of mass. We can get the coordinates of the center of mass by calculating the average of the weighted position coordinates of the distribution of mass. It is a very important concept and has a lot of useful applications in physics.

Back to random walk, properties of random walk in $\Z^d$ are undoubtedly a popular subject to study. A vast amount of study has been devoted to the investigation of the recurrence classification. This includes the famous P\'olya's and Chung-Fuchs results in Chapter \ref{introduction}. 

Now we want to go one step further and combine these two concepts, to consider the centre of mass (or time average, centre of gravity) of a random walk. Let $d \ge 1$. Suppose that $X, X_1, X_2, \ldots$ is a sequence of i.i.d.~random variables on $\R^d$.
We consider the random walk $(S_n, n \in \ZP)$ in $\R^d$ defined by $S_0 := \0$ and $S_n :=\sum_{i=1}^{n}{X_i}$
($n \geq 1$). 

Our object of interest is the \emph{centre of mass process} $(G_n , n \in \ZP)$ corresponding to the random walk, defined by $G_0 := \0$ and $G_n := \frac{1}{n}\sum_{i=1}^{n}{S_i}$ ($n \ge 1$). The question of the asymptotic behaviour of $G_n$ was raised by P.~Erd\H os (see \cite{KG}).

Lets look at some simulations of how it looks.

\vspace*{\fill}
\begin{figure}[H]
\color{blue}
\begin{center}
\includegraphics[width=13cm, angle=0, clip = true]{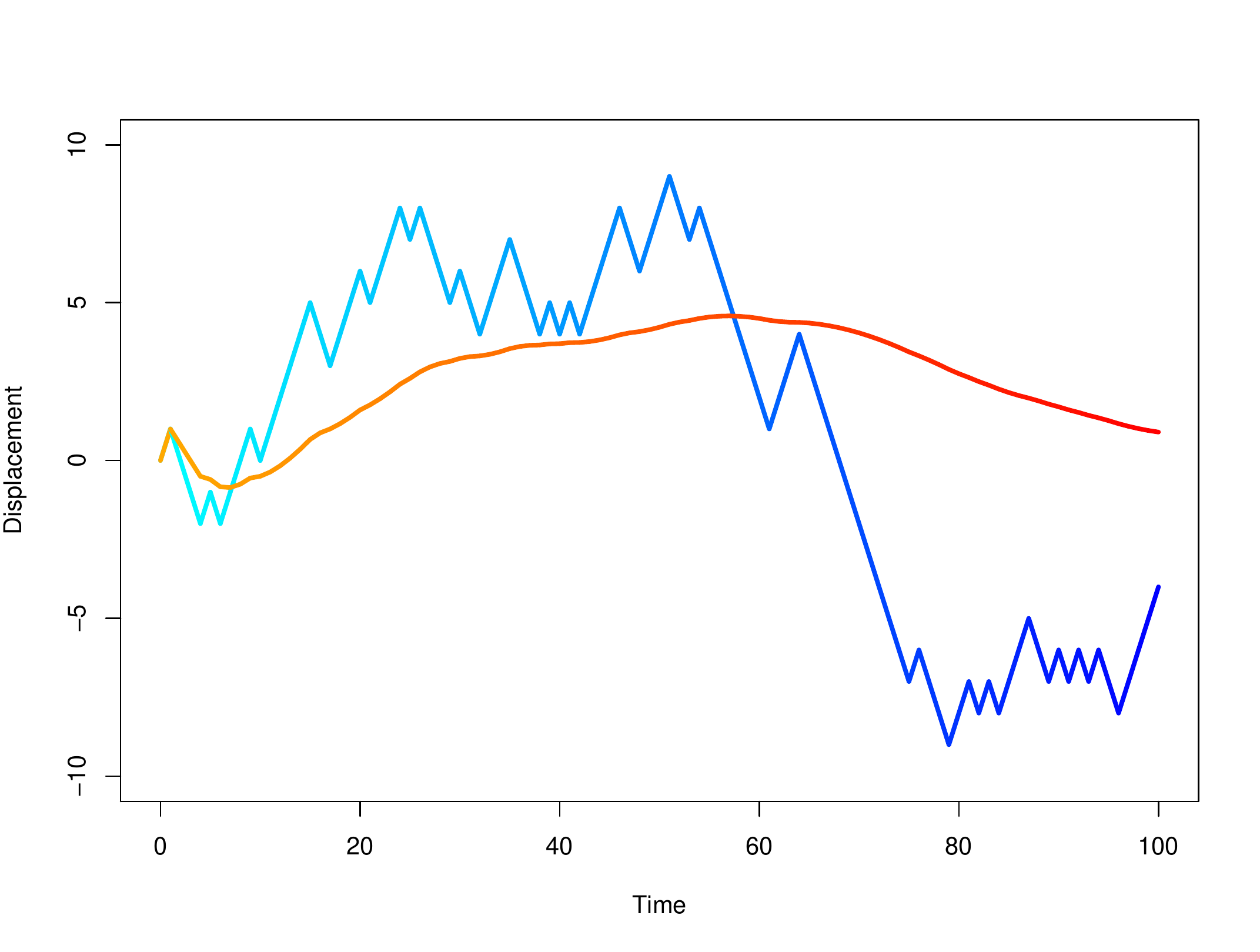}
\caption{An example of a one dimensional random walk (Light blue to blue) with the corresponding centre of mass process (Orange to red).}
\label{fig_SG1}
\end{center}
\end{figure}
\vfill

\newpage

\vspace*{\fill}
\begin{figure}[H]
\color{blue}
\vspace{-20mm}
\begin{center}
\includegraphics[width=12cm, angle=0, clip = true]{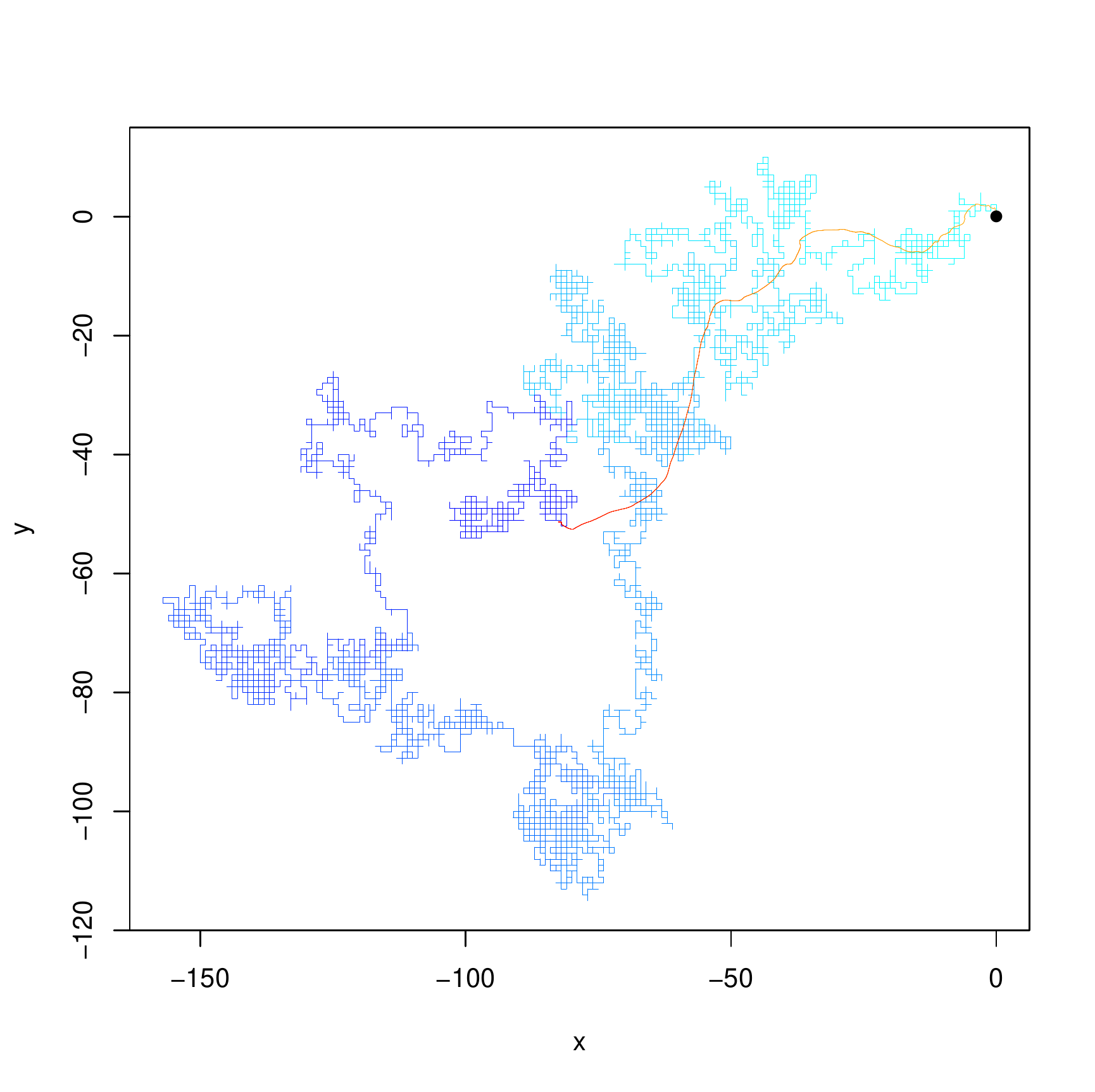}
\vspace{-10mm}
\includegraphics[width=12cm, angle=0, clip = true]{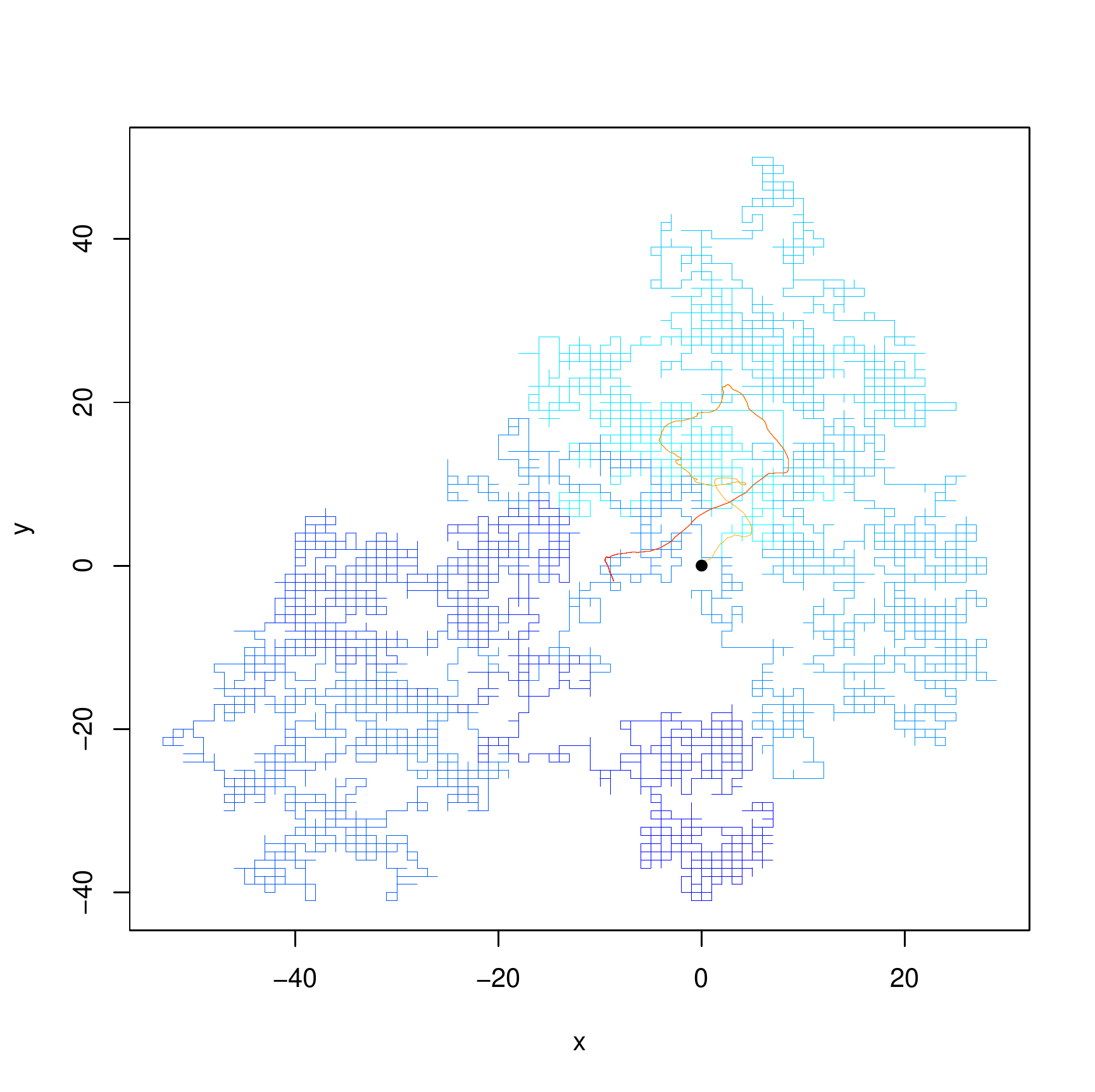}
\caption{Two examples of a two dimensional random walk (Light blue to blue) with the corresponding centre of mass process (Orange to red).}
\label{fig_SG2}
\end{center}
\end{figure}
\vfill

\newpage

\vspace*{\fill}
\begin{figure}[H]
\color{blue}
\begin{center}
\includegraphics[width=7.8cm, angle=0, clip = true]{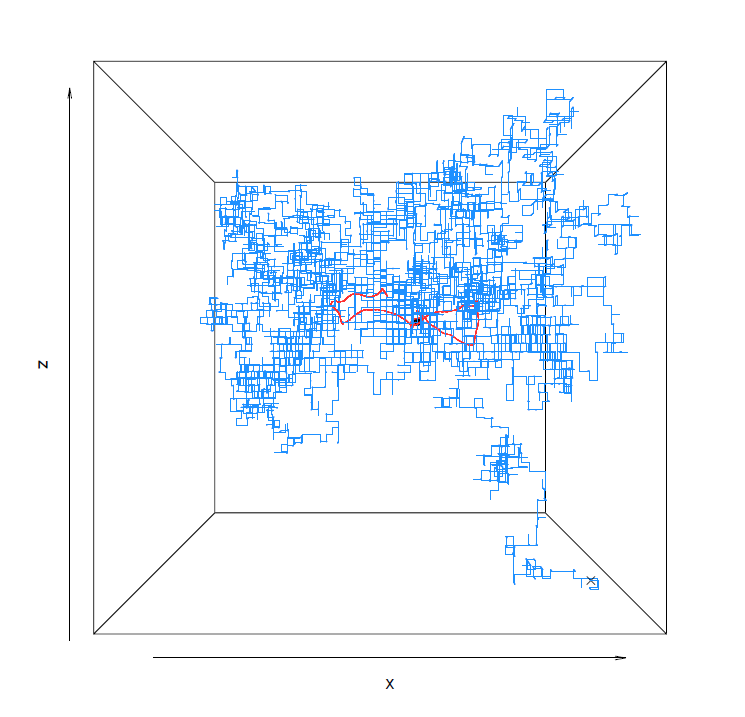}\hspace{-9mm}
\includegraphics[width=7.8cm, angle=0, clip = true]{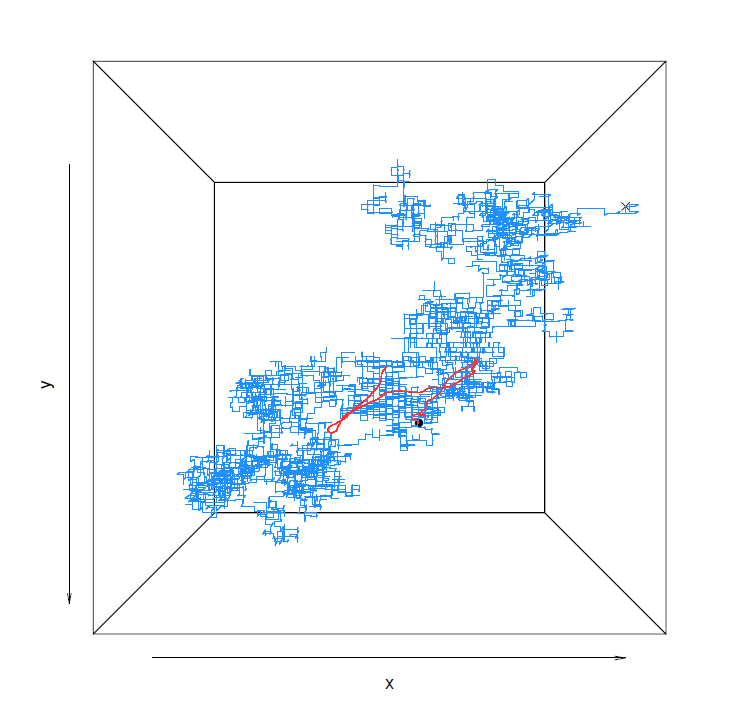}
\hspace{-9mm}
\includegraphics[width=7.8cm, angle=0, clip = true]{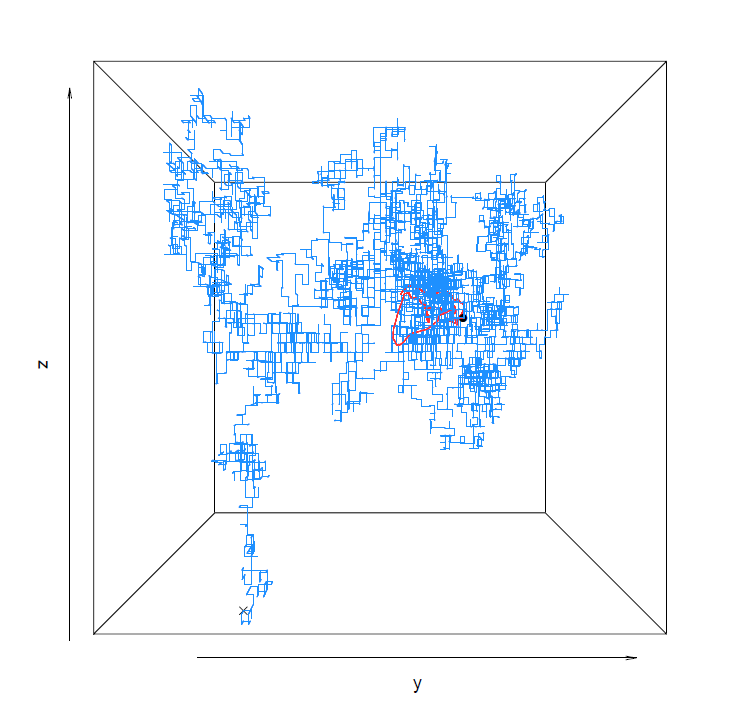}
\hspace{-9mm}
\includegraphics[width=7.8cm, angle=0, clip = true]{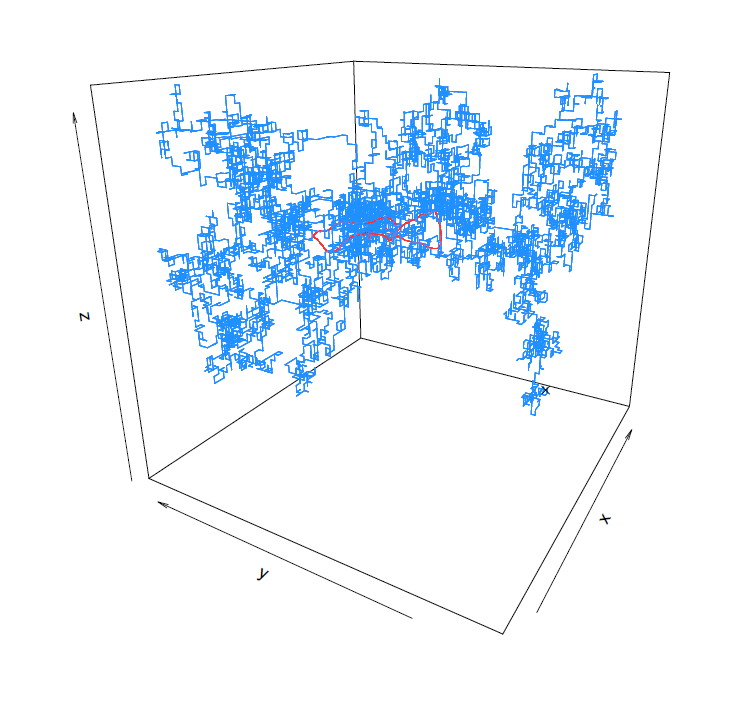}
\caption{An example of a three dimensional random walk (Blue) with the corresponding centre of mass process (Red) at front (top left), top (top right), left (bottom left) and side (bottom right) angle.}
\label{fig_SG3}
\end{center}
\end{figure}
\vfill

\newpage

Compared to Part I of this thesis, this is a very different model. We used a non homogeneous walk in a very special state space in the former to fit in some applications with the specific structure. However, this is not the only way to put our theory of random walk in action. Instead, we will try to use a common state space on a $d$-dimensional space, but as general as possible, together with the centre of mass, an averaging process derived from a random walk with as little structure as possible. As a result, we are able to get very general results, which can apply to a broad range of problems. 

Random walks can be used to model physical polymer molecules \cite{CMVW, RC} in which case the centre of mass is of obvious physical relevance. The random walk can also be used to model animal behaviour, and the motion of both macroscopic and microscopic organisms \cite{PH, PB3}. In this context the centre of mass is a natural summary statistic of an animal's roaming behaviour.

Despite the fact that it is a interesting and useful model, there is an extremely limited literature which gives any recurrence property of the centre of mass of random walk. The only result related to the properties of this process that we found is due to Karl Grill \cite{KG}, which gives the fact that for simple symmetric random walk, its center of mass is recurrent for $d=1$ only and is transient for $d \ge 2$. This answer was originally conjectured by Paul Erd\H os, and no further generalization or conjecture is found in the literature. 

This centre of mass process has a rich meaning in terms of application. For example in \cite{CMVW}, the asymptotic behaviour of a $d$-dimensional self-interacting random walk, which is repelled or attracted by the centre of mass of its previous trajectory, was studied. The walk's trajectory models a random polymer chain in either poor or good solvent. 

In this part of the thesis we will deeply investigate this mysterious centre of mass process and prove some of its properties in a general setting. The first goal in this part is to get the asymptotic behaviour of the process. We give the \emph{strong law of large numbers} and the \emph{central limit theorem} for the centre of mass, under minimal assumptions and we will further extend these results in the next part of the thesis. The first main result is a local central limit theorem. Unlike the usual central limit theorem, it is a much more precise theorem that requires a lot of effort to prove as it is not a direct consequence of the central limit theorem. One crucial difference here is that the usual central limit theorem does not distinguish between continuous and discrete random variables, while the local limit theorem does. Especially in the discrete case, the local limit theorem tells us more on the structure of the process. Roughly speaking, the central limit theorem tell us the process should behave like the normal distribution as a whole, but the local limit theorem tell us that the probability distribution of the process at each point behave like the probability density function of the normal distribution.

The second important goal in this part is to obtain the recurrence classification of the center of mass process, associated with any homogeneous random walk in $d$-dimensions. If we just focus on the recurrence classification, then the case where the random walk has a non-zero drift will not be very interesting, as the strong law of large numbers that we prove tells us the process must be transient in any dimension. In the case of zero drift, it turns out to be a very difficult problem without any additional assumptions. We acquire the classification result for any lattice random walk on $\R^d$ with finite second moments of the drift, which is a minor but important assumption to control the behaviour of the jumps. We show that under these assumptions the centre of mass is recurrent in one dimension but transient in two or more dimensions. We will also provide more general results in one dimension.

It is notable that the behaviour of the centre of mass of the random walk, especially in $2$-dimensions, is counter-intuitive at first sight. If we just observe the central limit theorem of both process, they behave very similarly. In fact, comparing Proposition 7.1.2 with Theorem 1.6.2, we see $S_n$ converge to $\mu$ and $G_n$ converge to $\mu /2$. The variance in the normal distribution for $G_n$ is $1/3$ of the one for $S_n$. Also, after some calculation, we will see that both process have the same magnitude of probabilities to come back to a fixed region. Comparing Theorem 7.2.1 with Theorem 6.5.1 we see that for a fixed ball $\mathcal{B}$, both $\Pr (S_n \in \mathcal{B})$ and $\Pr (G_n \in \mathcal{B})$ are of order $n^{-d/2}$. Hence one might think the recurrence classification of the center of mass process should follow from the random walk. Shockingly, this is not the case because we did not consider the speed of the processes, which will be the key for our proofs of theorems in the next chapter. This once again suggests that the recurrence classification is easy to define and to understand, but to really classify it properly around the critical region is a very subtle problem. 

The outline of this part of the thesis is as follows. This chapter will provide the background material we need for the proofs of our main theorems. In the next chapter, we will give a local central limit theorem for the center of mass process, and give the recurrence classification of it in different dimensions. Chapter 8 will provide all the proofs and technical details for these theorems, and we will close up with some examples in Chapter 9.

We view vectors in $\R^d$ as column vectors throughout; $\0$ denotes the zero vector.

\section{Lattice distributions and characteristic functions}

The first main goal for this part of the thesis is a \emph{local} central limit theorem for the centre of mass process. To achieve this we have to know some facts about lattice distributions and characteristic functions. 

In our centre of mass model, we assume that $X$ has a non-degenerate $d$-dimensional lattice distribution. Thus (see \cite[Ch.~5]{BR}) there is a unique minimal subgroup $L := H \Z^d$ of $\R^d$, where $H$ is a $d$ by $d$ matrix, such that $\Pr ( X \in \bb + L ) =1$ for some $\bb \in \R^d$, with the property that if $\Pr ( X \in \bx + L' ) = 1$ for some closed subgroup $L'$ of $\R^d$ and $\bx \in \R^d$, then $L \subseteq L'$, and with $h := | \det H | \in (0, \infty)$. In other words, we make the following assumption.

\newpage
\begin{description}
\item
[\namedlabel{ass:basicd2}{L}]
Suppose that the minimal subgroup of $\R^d$ associated with $X$ is $L := H \Z^d$ with $h := | \det H | > 0$.
\end{description}
Equivalent conditions to~\eqref{ass:basicd2} can be formulated in terms of the characteristic function of $X$ or in terms of the maximality of $h$: see Lemma~\ref{lem:equivalent} below. Note that there may be many matrices $H$ for which $H \Z^d$ is equal to (unique) $L$, but for all of these $| \det H|$ is the same.
Also note that symmetric simple random walk (SSRW) does not satisfy~\eqref{ass:basicd2} with the obvious choice $H = I$ (the identity), but does
satisfy~\eqref{ass:basicd2} if $H$ has the maximal choice $h=2$: see Chapter 9 for more details.

We collect some facts about lattice distributions: 
for reference see \cite[Ch.~5]{BR} and \cite[\S 7]{PRW}. 
Let
\[ \cH := \{  H   : \Pr ( X \in \bb + H \Z^d ) =1 \text{ for some } \bb \in \R^d \} .\]
If $X$ has a lattice distribution, then $\cH$ is nonempty, and if $X$
is non-degenerate then any $H \in \cH$ has $| \det H | >0$. (Here and elsewhere, `non-degenerate' means not supported on any $(d-1)$-dimensional hyperplane.)
Let $K := \{ | \det H | : H \in \cH \}$. The next result gives an upper bound on $h \in K$;
note that this bound
is sharp in both of the examples in Chapter 9.

\begin{lemma}
\label{lem:h-bounded}
Suppose that $X$ has a non-degenerate lattice distribution. Then $K \subseteq (0,\infty)$ is bounded, and $\inf K = 0$.
\end{lemma}
\begin{proof}
Since $X$ has a non-degenerate lattice distribution, we have that (i) $\cH$ is non-empty and $|\det H| > 0$ for all $H \in \cH$;
and (ii) there exists $\cX := \{ \bx_0, \bx_1, \ldots, \bx_d \}$ such that
$\bx_0, \ldots, \bx_d$ are affinely
independent, and $\Pr ( X = \bx_i ) >0$ for each $i$. 
Statement (i) shows that $K \subseteq (0,\infty)$ is nonempty,
and statement (ii) shows that $K$ is bounded. 
Indeed, for any $H \in \cH$ we have that there exists $\bb$ such that $\cX \subset  \bb + H \Z^d$, i.e.,
$H^{-1} ( \cX - \bb ) \subset  \Z^d$.
For $i \in \{1,\ldots, d\}$ let $\lambda_i = \bx_i - \bx_0$.
Then the linearly independent vectors $\lambda_1, \ldots, \lambda_d$ define a parallelepiped $P$
with volume $| \det \Lambda | \in (0,\infty)$, where $\Lambda$ denotes the $d \times d$ matrix whose columns are $\lambda_1, \ldots, \lambda_d$.
Since $H^{-1} ( \cX - \bb )$ are points of $\Z^d$, we have that all the vertices of the parallelepiped $P':= H^{-1} (\bx_0 + P - \bb )$
are points of $\Z^d$. Now $P'$ has volume $h^{-1} | \det \Lambda | > 0$, but, as a parallelepiped of positive volume
whose vertices are in $\Z^d$, must have volume at least $1$. Thus $h^{-1} | \det \Lambda | \geq 1$, i.e., $h \leq | \det \Lambda | < \infty$.
Also, we see that if $H \in \cH$, then $H/2 \in \cH$ as well, so if $h \in K$ then $h/2^d \in K$ too. 
\end{proof}

The next lemma will be on the \emph{characteristic function} of $X$, which is defined for $\bt \in \R^d$ as $\varphi (\bt) := \E \re^{i \bt^\tra X}$. We also define $U := \{ \bt \in \R^d : | \varphi (\bt ) | = 1 \}$.
Given an invertible $d$ by $d$ matrix $H$, set $S_H := 2 \pi ( H^\tra)^{-1} \Z^d$. The next result shows that if $H \in \cH$, then $S_H \subseteq U$.

\begin{lemma}
\label{lem:property}
Suppose that $H \in \cH$. Then 
$|\varphi (\bu )| = 1$ for all $\bu \in S_H$.
\end{lemma}
\begin{proof}
First observe that the norm of the characteristic function is invariant under translation by any vector of the form of $2\pi (H^\tra)^{-1} \bk$ with $\bk \in \Z^d$. 
To see this, note that for any $\bk \in \Z^d$, 
\[
\left|\varphi(\bt+2\pi (H^\tra)^{-1}\bk)\right|=\left| \E \left[ \re^{i\bt^\tra X}  \cdot \re^{ 2 \pi i \bk^\tra H^{-1} X} \right] \right|.
\]
Since $H \in \cH$, we may write
 $X =\bb + H W$, where $\bb \in \R^d$ is constant and $W \in \Z^d$. Hence 
\[
\left|\varphi(\bt+2\pi (H^\tra)^{-1}\bk)\right|
=\left|\re^{ 2 \pi i \bk^\tra H^{-1} \bb}\right| \cdot \left| \E \left[ \re^{i\bt^\tra X} \cdot \re^{2 \pi i \bk^\tra W } \right] \right|  ,\]
because $\bk^\tra H^{-1} \bb$ is a non-random scalar.
Then, since $|\exp\{ 2 \pi i \bk^\tra H^{-1} \bb \}|=1$
and $\bk^\tra W \in \Z$, so that 
$\exp\{2 \pi i \bk^\tra W \}=1$, it follows that for any $\bk \in \Z^d$,
\begin{equation}
\label{eq:periodic}
\left|\varphi(\bt+2\pi (H^\tra)^{-1}\bk)\right| = \left|\varphi(\bt)\right|.
\end{equation}
In particular, the case $\bt = \0$ of~\eqref{eq:periodic} 
shows that $|\varphi( \bu ) | = 1$ if $\bu \in S_H$.
\end{proof}

If $\Pr ( X \in \bb + H \Z^d) = 1$ and $\Pr (X = \bx ) >0$, then
$\bx - \bb \in H \Z^d$ so that $\bx + H \Z^d = \bb + H\Z^d$,
and so if $H \in \cH$ then $\Pr (X \in \bx + H \Z^d ) = 1$
for any $\bx$ with $\Pr (X = \bx ) >0$.

Lemma~21.4 of~\cite{BR} shows that there is a unique minimal subgroup $L$ of $\R^d$ such that
$\Pr ( X \in \bx + L ) =1$ for any $\bx$ with $\Pr ( X = \bx ) >0$ and
if $H \in \cH$ then
$L \subseteq H \Z^d$. Moreover, the discrete subgroup $L$
is generated by $\{ \xi : \Pr ( X = \bx + \xi ) > 0 \}$ for any given
$\bx$ with $\Pr ( X = \bx ) > 0$. We have $L= H_0 \Z^d$ for some (not necessarily unique) $H_0 \in \cH$;
let $\cH_0 := \{ H \in \cH : L = H \Z^d \}$.

The next result gives equivalent formulations of the fundamental assumption~\eqref{ass:basicd2}.
For $\rho >0$, define $S_H(\rho) := \cup_{\by \in S_H} B(\by;\rho)$, 
where $B(\by; \rho)$ is the open Euclidean ball of radius $\rho$ centred at $\by \in \R^d$.

\begin{lemma}
\label{lem:equivalent}
Suppose that $X$ is non-degenerate and $H \in \cH$.
The following are equivalent.
\begin{itemize}
\item[(i)] $H \in \cH_0$.
\item[(ii)] $| \det H |$ is the maximal element of $K$.
\item[(iii)] $S_H = U$.
\end{itemize}
Moreover, if any one of these conditions holds then, for any $\rho >0$, there exists a positive constant $c_{\rho}$ such that 
\[
\left| \varphi(\bu) \right| \le \re^{-c_{\rho}}, \text{ for any } \bu \notin S_H(\rho).
\]
\end{lemma}
\begin{proof}
Suppose that $H_0 \in \cH_0$ and $H \in \cH$.
Let $h_0 = | \det H_0|$ and $h = | \det H|$.
Then, by minimality, $H_0 \Z^d \subseteq H \Z^d$, i.e., $H^{-1} H_0 \Z^d \subseteq \Z^d$.
Thus $H^{-1} H_0 [0,1]^d$ is a parallelepiped whose vertices are all in $\Z^d$,
and necessarily this parallelepiped has volume at least 1. Hence $h_0 / h \geq 1$, i.e., $h \leq h_0$. 
Thus if $H \in \cH_0$ then $| \det H |$ is maximal. 
On the other hand, suppose $H \in \cH \setminus \cH_0$ and $H_0 \in \cH_0$.
Then $H_0 \Z^d \subset H \Z^d$ are not equal, so there is some $\bx \in H \Z^d$ with $\bx \notin H_0 \Z^d$.
Thus if $\by = H^{-1} \bx \in \Z^d$, we have that $H^{-1} H_0 \Z^d \subset \Z^d$ with $\by \notin H^{-1} H_0 \Z^d$.
For $\bz \in \Z^d$ we have $\by = H^{-1} H_0 ( \bz + \alpha )$ where $\alpha \in [0,1]^d$ is not a vertex;
but then $\by - H^{-1} H_0 \bz \in \Z^d$ as well. Thus $\beta = H^{-1} H_0 \alpha$ is a point of $\Z^d$
contained in the parallelepiped $P = H^{-1} H_0 [0,1]^d$, and moreover all the vertices of $P$ are in $\Z^d$,
and $\beta$ is not a vertex. Hence
the parallelepiped $P$ has volume strictly greater than $1$ (see \cite[p.~69]{PRW}), and so $h_0 / h > 1$.
Thus if $H \notin \cH_0$ then $| \det H |$ is not maximal. 
 Thus (i) and (ii) are equivalent.

We show that (i) implies (iii).
For $H \in \cH$ set
\begin{align*} R_H & := \{ \bt \in \R^d : \bx^\tra \bt \in 2 \pi \Z \text{ for all } \bx \in H \Z^d \}\\
& = \{ \bt \in \R^d : \bz^\tra H^\tra \bt \in 2 \pi \Z \text{ for all } \bz \in \Z^d \} . \end{align*}
It follows that
\[ R_H =  2 \pi ( H^\tra )^{-1} \{ \by \in \R^d : \bz^\tra \by \in \Z \text{ for all } \bz \in \Z^d \} =  
 2 \pi ( H^\tra )^{-1}  \Z^d = S_H .\]
So $R_H = S_H$ for any $H \in \cH$ with $| \det H | >0$. Moreover,
Lemma~21.6 of~\cite{BR} shows that $R_H = U$ if $H \Z^d$ is minimal.
Thus (i) implies (iii).

Next we show that (iii) implies (ii).
Let $h_\star := \sup K$, which, by Lemma~\ref{lem:h-bounded} is finite and positive.
Suppose that $H \in \cH$ with $|\det H| = h \in (0,h_\star)$.
Then for any $\eps >0$ sufficiently small, we can find $H_1 \in \cH$ with $| \det H_1 | = h_1 \in ( h , h_\star]$ such that $h_1 >  ( 1+2 \eps) h$
and $h_1 > (1-\eps) h_\star$. Let $S = 2\pi (H^\tra)^{-1} \Z^d$ and $S_1 = 2\pi (H_1^\tra)^{-1} \Z^d$.

Consider $\bx$ with $\Pr (X = \bx ) >0$. Then there exist $\bb, \bb_1 \in \R^d$ (not depending on $\bx$) and $\bz,\bz_1 \in \Z^d$ (depending on $\bx$)
such that 
\[ \bx = \bb + H \bz  = \bb_1 + H_1 \bz_1 ,\]
and hence 
\begin{equation}
\label{calH1}
\bz = H^{-1} (\bb_1 - \bb) + H^{-1} H_1 \bz_1. 
\end{equation}
Take $\bs = 2\pi(H_1^\tra)^{-1} \bz_1 \in S_1$.
Assume, for the purpose of deriving a contradiction,
 that $S_1 \subseteq S$. Then $\bs \in S$, i.e., there exists $\bz_2 \in \Z^d$ such that
\[ \bs = 2\pi(H_1^\tra)^{-1} \bz_1 = 2\pi(H^\tra)^{-1} \bz_2.\]
Together with~\eqref{calH1}, this implies that 
\[ \bz = H^{-1} H_1 H_1^{\tra} (H^\tra)^{-1} \bz_2 + H^{-1}(\bb_1 - \bb) .\]
It follows that 
\begin{align*}
\bx  = \bb + H \bz  = \bb_1 + H_1 H_1^{\tra} (H^\tra)^{-1} \bz_2.
\end{align*}
Now if we take $\bb_2= \bb_1$ and $H_2 = H_1 H_1^{\tra} (H^\tra)^{-1}$, 
we have shown that every $\bx$ for which $\Pr ( X = \bx ) >0$ has $\bx \in \bb_2 + H_2 \Z^d$, i.e.,
$H_2 \in \cH$. But
\begin{align*}
\left|\det H_2 \right| & = \left|\det H_1 \right| \left|\det H_1^{\tra} \right| \left|\det (H^\tra)^{-1} \right| = \frac{h_1^2}{h} \\
& > (1+2\eps) (1-\eps) h_\star > h_\star ,\end{align*}
for $\eps$ sufficiently small,
which contradicts the definition of $h_\star$. Thus there exists some $\bx \in S_1$ with $\bx \notin S$.

From Lemma~\ref{lem:property}, we have $S_1 \subseteq U$; hence there is some $\bx \in U$ with $\bx \notin S$. In other words,
we have shown that if $h \in (0,h_\star)$ then $S \neq U$. Thus if we assume that $S = U$, the only possibility is $h = h_\star \in \cH$.
Thus (iii) implies (ii).

To prove the final statement in the lemma, we may suppose that (iii) holds. Then $| \varphi ( \bu ) | < 1$ if $\bu \notin S_H$.
To finish the proof of the lemma, it suffices to show that $\sup_{\bu \notin S_H(\rho)} |\varphi(\bu)| <1$.
But, by the periodicity property~\eqref{eq:periodic}, we have
$\sup_{\bu \notin S_H(\rho)} | \varphi (\bu) | = \sup_{\bu \in T_H(\rho)} | \varphi (\bu) |$
where $T_H(\rho) := 2 \pi ( H^\tra )^{-1} [ -\frac{1}{2}, \frac{1}{2} ]^d \setminus B ( \0 ; \rho )$.
Suppose that $\sup_{\bu \in T_H(\rho)} | \varphi (\bu) | = 1$; then by the continuity of $|\varphi(\bu)|$, the supremum is attained at a point $\bu$ in the compact set $T_H(\rho)$, 
contradicting the fact that $|\varphi (\bu)| < 1$ for all $\bu \notin S_H$. Hence $\sup_{\bu \in T_H(\rho)} | \varphi (\bu) | < 1$, and the proof is completed.
\end{proof}

Further information about lattice distributions can also be found in \cite{BR} and \cite{PRW}.

\section{Hewitt-Savage zero one law}

The \emph{Hewitt-Savage zero-one law} is important for us to establish some recurrence results in one dimensional center of mass process.

It is a theorem similar to \emph{Kolmogorov's zero-one law} and the Borel-Cantelli lemma, that specifies that a particular type of event will happen almost surely or not happen almost surely.

Formally, we will follow the formulation and definitions from Chow and Teicher \cite[p.232]{CT}.

Let $(S, \mathcal{S})$ be a measurable space. Define $S^\infty = S \times S \times S \times \ldots$ and let $\mathcal{S}^\infty$ be the Borel subsets of $S^\infty$. Let $X=(X_1,X_2, \ldots)$ be a sequence of random variables on $(\Omega, \mathcal{F}, \Pr)$ taking values in $S$, so $X \in S^\infty$. A mapping $\pi=(\pi_1, \pi_2, \ldots)$ from $\N$ to $\N$ is called a finite permutation if $\pi$ is one-to-one and $\pi_n=n$ for all $n$ sufficiently large. Let $\Pi$ be the set of all finite permutations. For $\pi \in \Pi$ define $\pi X = (X_{\pi_1}, X_{\pi_2}, \ldots)$. Then
\begin{equation*}
\mathcal{E} = \left\{ X^{-1}(\mathcal{B}) : \mathcal{B} \in \mathcal{S}^\infty, \Pr \left( X^{-1}(\mathcal{B}) \bigtriangleup (\pi X)^{-1}(\mathcal{B})\right) = 0  \text{ for all } \pi \in \Pi \right\}
\end{equation*}
is called the exchangeable $\sigma$-algebra of $X$, where $\bigtriangleup$ is the symmetric difference. In words, events in $\mathcal{E}$ are those events which are invariant under all finite permutations of the $X_i$. Now we state the Hewitt-Savage zero-one law as in \cite{CT}. For the proof see \cite[p.238]{CT}.

\begin{theorem}[Hewitt-Savage zero-one law]
Let $X_1, X_2, \ldots$ be  i.i.d. random variables on $(\Omega, \mathcal{F}, \Pr)$ taking values in a measurable space $(S, \mathcal{S})$, and let $\mathcal{E}$ denote the exchangeable $\sigma$-algebra. Then $\mathcal{E}$ is trivial, i.e., if $A \in \mathcal{E}$ then $\Pr(A) \in \{0, 1\}$. 
\end{theorem}

In the special case of spatially homogeneous random walk, one can apply the Hewitt-Savage zero-one law to prove the dichotomy of the recurrence classification of the walk as the the event $R = \{ S_n = 0 \text{ infinitely often} \}$ is invariant under finite permutations of the increments, i.e., exchangeable. Notice that the Kolmogorov's zero-one law cannot directly apply here as $R$ is not a tail event.

\section{Local limit theorem for random walks}

It is always good to start with a similar or classical result in random walk theory and try to compare them with our centre of mass process. Sometimes, it is possible to extract useful bits in the classical proofs of similar theorems. Towards our goal to the local limit theorem of the centre of mass process, we should first establish a local limit theorem for our random walk. Here is the formal setup.

Throughout we will use the notation
\[ \bmu := \E X , ~~~ M := \E [ ( X - \bmu) (X - \bmu)^\tra ] \]
whenever the expectations exist; when defined, $M$ is a symmetric $d$ by $d$ matrix.

To go further we typically assume the following.
\begin{description}
\item
[\namedlabel{ass:basicd}{M}]
Suppose that $\E [ \| X \|^2 ] < \infty$ and $M$ is positive-definite. 
\end{description}

Notice that $\Pr(X \in \bb +H \Z^d)=1$ implies $\Pr(S_n \in n\bb + H \Z^d)=1$. For $\bx \in \R^d$, define
$q_n ( \bx )  := \Pr ( n^{-1/2} S_n= \bx )$,  and
\begin{align}
\label{eq:ndefrw}
m (\bx) & : = \frac{\exp \{- \tfrac{1}{2} \bx^\tra M^{-1}\bx  \} }{(2\pi)^{d/2} \sqrt{\det M}} 
.
\end{align}
Also define
\[ L_n := \left\{ n^{-1/2} \left( n\bb + H \Z^d \right) \right\}. \]
Here is our local limit theorem for the random walk.
\begin{theorem}
\label{thm:LCLTRW}
Suppose that~\eqref{ass:basicd2} hold and suppose $\E [ \| X \|^2 ] < \infty$ and $M$ is positive-definite.  Then we have 
\begin{equation}
\lim_{n \to \infty} \sup_{\bx \in L_n} \left| \frac{n^{d/ 2}}{h} q_n(\bx)-m \left(\bx - n^{1/2} \bmu \right) \right| = 0. 
\label{eq:lcltrw}
\end{equation}
\end{theorem}

Note that $\Pr (S_n \in \mathcal{B}) = O(n^{-d/2})$ for a fixed ball $\mathcal{B}$ since $n^{-1/2} \mathcal{B}$ contains $O(1)$ lattice points.

This rest of this section is devoted to the proof of Theorem~\ref{thm:LCLTRW}. 

\begin{proof}[Proof of Theorem~\ref{thm:LCLTRW}]

With standard Fourier analysis, we know Theorem~\ref{thm:LCLTRW} in the case where $\bb = \0$ and $H = I$ (the identity), see e.g. \cite[\S 2.2--\S 2.3]{LL} for details. The first proof of the one-dimensional result can be trace back to \cite{BVG}. 

Now it remains to show that it suffices to establish Theorem~\ref{thm:LCLTRW} in this special case. To see this, suppose that $X \in \bb + H \mathbb{Z}^d$ and set $\tilde X = H^{-1} ( X - \bb )$. Then $\tilde X \in \Z^d$. By linearity of expectation, we have 
\[ \tilde \bmu := \E \tilde X = H^{-1} ( \mu - \bb), \text{ and }
\tilde M := \E [ ( \tilde X - \tilde \bmu) ( \tilde X - \tilde \bmu)^\tra ] = H^{-1} M (H^{-1} )^\tra. \]
Note that $(H^{-1})^\tra$ is nonsingular, so $(H^{-1})^\tra \bx \neq \0$ for all $\bx \neq \0$. Hence for $\bx \neq \0$, $\bx^\tra \tilde M \bx = \by^\tra M \by$ where $\by = (H^{-1})^\tra \bx \neq \0$, so that since $M$ is positive definite we have $\bx^\tra \tilde M \bx > 0$; hence $\tilde M$ is also positive definite. We also have $\tilde S_n := \sum_{i=1}^n \tilde X_i = H^{-1} (S_n - n\bb)$. The assumption that $H \Z^d$ is minimal for $X$ implies that $\Z^d$ is minimal for $\tilde X$. Thus the process defined by $\tilde X$ satisfies the hypotheses of Theorem~\ref{thm:LCLTRW} in the case
where $\bb = \0$ and $H = I$, with mean $\tilde \bmu$ and covariance $\tilde M$,
and that result yields
\begin{equation}
\label{eq:reduction1rw}
 \lim_{n \to \infty} \sup_{\bx \in n^{-1/2} \Z^d} \left| n^{d/2} \Pr ( n^{-1/2} \tilde S_n = \bx ) - \tilde m \left( \bx - n^{1/2} \tilde \bmu \right) \right| = 0,\end{equation}
where
\[ \tilde m ( \bz ) := \frac{( \det \tilde M  )^{-1/2}}{(2 \pi )^{d/2}}  \exp \left\{ - \frac{1}{2} \bz^\tra \tilde M^{-1} \bz \right\} .\]
But
\[ \Pr ( n^{-1/2} \tilde S_n = \bx ) = \Pr \left( n^{-1/2} S_n = n^{1/2} \bb + H \bx \right) = \Pr (n^{-1/2} S_n = \by ) \]
where $\by = n^{1/2} \bb + H \bx$
so $\by \in n^{-1/2} ( n \bb + H \Z^d )$.
Also,
\begin{align*}   \bx - n^{1/2} \tilde \bmu   & = 
\left( H^{-1} \by - n^{1/2} H^{-1} \bb \right) -n^{1/2} H^{-1} ( \bmu -\bb) \\
& =  H^{-1} \by -  n^{1/2} H^{-1}  \bmu .
\end{align*}
Hence, since $\tilde M^{-1} = H^\tra M^{-1} H$ and $\det \tilde M = h^{-2} \det M$,
\begin{align*}
\tilde m \left( \bx - n^{1/2} \tilde \bmu \right)  & =\frac{( \det \tilde M  )^{-1/2}}{(2 \pi )^{d/2}} 
 \exp \left\{ - \frac{1}{2} \left(  \by - n^{1/2} \bmu \right)^\tra M^{-1}
\left(  \by - n^{1/2} \bmu \right) \right\} \\
& = h m 
\left(  \by - n^{1/2} \bmu \right) .
 \end{align*}
It follows that~\eqref{eq:reduction1rw} is equivalent to
\[ \lim_{n \to \infty} \sup_{\by \in n^{-1/2} ( n \bb + H \Z^d ) } \left| \frac{n^{d/2}}{h} \Pr ( n^{-1/2}  S_n = \by ) -   m \left(  \by - n^{1/2} \bmu \right) \right| = 0 ,\]
which is the general statement of Theorem~\ref{thm:LCLTRW}. 
\end{proof}

\section{Stable distributions and domains of attraction}

In this section, we will recall some basic theory on stable distributions, which will be used in the one dimensional case of the center of mass process. Definitions, proofs of theorems and more background material can be found in \cite{PB} and \cite{WW}. 


Formally, a distribution is defined to be stable if any linear combination of two independent random variables with this distribution has the same distribution, up to a shift and rescaling. A random variable is defined to be stable if its distribution is stable. 

We can characterize this type of function using the characteristic function of the distribution, as stated as the next theorem, a rewritten version of \cite[Def. 1.6]{JN2}.

\begin{theorem}
\label{thm:StableCharFn}
A distribution is stable if and only if the logarithm of its characteristic function is of the form
\begin{align*}
\log (\phi (t) ) = i \gamma t - c |t|^\alpha \left( 1 + i \beta \frac{t}{|t|} \omega(t,\alpha) \right)
\end{align*}
where $\gamma \in \R$, $\alpha \in (0,2]$, $\beta \in [-1,1]$, and
\begin{align*}
\omega(t, \alpha) = 
\begin{cases}
\tan \left( \frac{\pi}{2} \alpha \right) & \mathrm{\ if\ } \alpha \neq 1 \\
\frac{2}{\pi} \log (|t|) & \mathrm{\ if\ } \alpha = 1
\end{cases}
\end{align*}
\end{theorem}

In the case of $\alpha < 2$, the formula can express in the following way. 
\begin{align}
\log ( \phi (t)) & = i \gamma t + c_1 \int_{-\infty}^0 \left( e^{i t u} - 1 - \frac{itu}{1 + u^2} \right) \frac{du}{|u|^{1+ \alpha}} \nonumber \\
& \quad + c_2 \int_0^{\infty} \left (  e^{i t u} - 1 - \frac{itu}{1 + u^2} \right) \frac{du }{u^{1 + \alpha}} \label{eqn:StableCharFn}
\end{align}
where $c_1$ and $c_2$ are some positive constants and, again, $\gamma \in \R$.

The case that $\alpha = 2$ is the classical Gaussian distribution, which we can write down the close form of the distribution. We can also find closed forms for the density functions in the cases $\alpha = 1$ (the Cauchy distribution) and $\alpha = 0.5$ (the L\'{e}vy distribution). However, there are no closed form in general, meaning that working with the characteristic functions is necessary.  For the purpose of this thesis, we will only consider the symmetric case, i.e. $\gamma=\beta=0$, when $\phi (t) = e^{-c|t|^{\alpha}}$, $\alpha \in (0,2)$.

We also recall the definition of the domains of attraction as follows.

\begin{definition}
A random variable $X$ belongs to the domain of attraction of a stable law $G$ iff there exist $a_n>0, b_n \in \R$ such that 
$$ \frac{S_n - b_n}{a_n}  \tod G,$$
where $S_n$ denotes the $n$th partial sum, $S_n = \sum_{k=1}^n X_k$, and $X_k$ are i.i.d copies of $X$. 
We write $X \in \mathcal{D}(G)$, or, in terms of the distribution functions, $F_{X} \in \mathcal{D}(F_G)$.
\end{definition}

The stable distribution is closely related to domain of attraction in the following way.

\begin{theorem}
A distribution or random variable has a domain of attraction if and only if it is stable.
\end{theorem}

We will omit the proof here. For a deeper analysis of domains of attraction, and a characterisation of the distributions within them, see, for example, Chapter XVII of \cite{WF}.

All stable distributions are, by definition, in their own domains of attraction. Intuitively, by considering the distribution of the later terms of the sequence, we expect the limiting distribution to exhibit the same ``self-similar'' behaviour we required of stable distributions above.

\begin{theorem}
If $\xi$ is in the domain of attraction of a stable distribution with index $\alpha \in (0,2)$, then 
\begin{align*}
\mathbb{E} [|\xi|^r] < \infty &\mathrm{\ if\ } r < \alpha \\
= \infty & \mathrm{\ if\ } r > \alpha .
\end{align*}
For $r = \alpha$ either case is possible.
\end{theorem}

This implies that for $\alpha <1$, $\xi$ has no mean, and that for $\alpha < 2$, $\xi$ has no variance. The Normal ($\alpha = 2$) distribution is the only stable distribution with finite variance.

Now we should focus on a slightly narrower interest; the domain of \emph{normal} attraction, in which we specify the form of the constants $a_n$.

\begin{definition}
A random variable $X$ belongs to the domain of normal attraction of a stable law $G$ iff there exist $a >0$ and $b_n \in \R$ such that
\begin{align}
\label{eqn:domnorattr}
\frac{S_n - b_n}{a n^{\frac{1}{\alpha}}} \tod G,
\end{align}
where $S_n$ denotes the nth partial sum, $S_n = \sum_{k=1}^n X_k$.
\end{definition}
\noindent We have set $a_n = a n^{\frac{1}{\alpha}}$. The Convergence to Types theorem in \cite{PB} allows us to assume without loss of generality that $a = 1$.

\chapter{Main results}

\section[Law of large numbers and central limit theorem for centre of mass]{Law of large numbers and central limit theorem for centre of mass}
\sectionmark{LLN and CLT for centre of mass}

In this first section we would like to establish some standard properties and theorems in random walk theory for our centre of mass process.

Recall that we will use the notation
\[ \bmu := \E X , ~~~ M := \E [ ( X - \bmu) (X - \bmu)^\tra ] \]
whenever the expectations exist; when defined, $M$ is a symmetric $d$ by $d$ matrix.

Suppose $X, X_1, X_2, \ldots$ is a sequence of i.i.d. random variables on $\R^d$. The strong law of large numbers for $S_n$ yields the following strong law for $G_n$.
\begin{proposition}
\label{prop:LLN}
Suppose that $\E \| X \| < \infty$. Then $n^{-1} G_n \to \frac{1}{2} \bmu$, a.s., as $n \to \infty$.
\end{proposition}

Comparing with the strong law of large numbers for $S_n$, i.e., Theorem 1.6.1, we see $G_n$ behaves very similarly to $S_n$, up to a constant factor of $\frac{1}{2}$.

\begin{proof}[Proof of Proposition~\ref{prop:LLN}.]
By the strong law for $S_n$, we have that
for any $\eps>0$ there  exists $N_\eps$ with $\Pr (N_\eps < \infty) =1$ such that $\| S_n - n\bmu \| \le n \eps$ for all $n \ge N_\eps$. Then, by the triangle inequality,
\begin{align*}
\left\| G_n - (n+1) (\bmu/2)  \right\| & = \frac{1}{n} \left\|\sum_{i=1}^n ( S_i - i \bmu ) \right\| \\
&\le \frac{1}{n}  \sum_{i=1}^{N_\eps} \| S_i - i \bmu \|  + \frac{1}{n}   \sum_{i=N_\eps}^n \| S_i - i \bmu \|  \\
& \leq \frac{1}{n}  \sum_{i=1}^{N_\eps} \| S_i - i \bmu \|  + \frac{1}{n} \sum_{i=1}^n i \eps .
\end{align*}
It follows that 
\[ \limsup_{n \to \infty} n^{-1} \left\| G_n - (n+1) (\bmu/2)  \right\| \leq \eps /2 ,\]
and since $\eps >0$ was arbitrary we get the result.
\end{proof}

To go further we typically assume the condition~\eqref{ass:basicd} holds. Note that
\begin{equation}
\label{eq:weighted-sum}
G_n = \sum_{i=1}^n \left( \frac{n-i+1}{n} \right) X_i .
\end{equation}
The representation~\eqref{eq:weighted-sum} leads via the
Lindeberg--Feller theorem for triangular arrays 
to the following central limit theorem;
we write `$\tod$' for convergence in distribution, and $\cN_d ( \mathbf{m} , \Sigma )$
for a $d$-dimensional normal random variable with mean $\mathbf{m}$ and covariance $\Sigma$. 
\begin{proposition}
\label{prop:CLT}
If~\eqref{ass:basicd} holds, then, as $n \to \infty$,
\[ n^{-1/2} \left( G_n - \frac{n}{2} \bmu \right) \tod \cN_d ( \0 , M/ 3 ) .\]
\end{proposition}

Comparing with the central limit theorem for $S_n$, i.e., Theorem 1.6.2, we see $G_n$ behaves very similarly to $S_n$ except the variance in the normal distribution is one third of the corresponding variance for $S_n$.

\begin{proof}[Proof of Proposition~\ref{prop:CLT}.]
For any unit vector $\be \in \Sp^{d-1}$, 
$\be \cdot G_n$ is the centre-of-mass associated with the one-dimensional random walk
with increments $\be \cdot X_i$; thus, by the Cramer--Wold device (see e.g~\cite[Theorem~3.9.5]{RD}), it suffices to 
establish the central limit theorem for $d=1$.

So take $d=1$ and write $\bmu = \mu$, $M = \sigma^2 \in (0,\infty)$.
It follows from~\eqref{eq:weighted-sum} that for fixed $n$,
 $G_n$ has the same distribution as
\[ G_n' := \sum_{i=1}^n \left( \frac{i}{n} \right) X_i .\]
It thus suffices to show that $n^{-1/2} ( G_n' - \frac{n}{2} \mu )$ converges in distribution
to $\cN (0, \sigma^2/3)$. We show that this follows from~\cite[Corollary~8.4.1]{AB}.
Define $T_{n,i} := \frac{i}{n^{3/2}} ( X_i - \mu)$.
Then
\[ \sum_{i=1}^n \Var (T_{n,i}) = \sum_{i=1}^n \frac{i^2}{n^3} \sigma^2 \to \frac{\sigma^2}{3}. \]
It remains to verify the Lindeberg condition for triangular arrays: for every $\eps>0$,
\[ 
\lim_{n \to \infty} \sum_{i=1}^n \E \left[ T_{n,i}^2 \1 {|T_{n,i}| > \eps }  \right] = 0 .
\] 
But we have that 
\begin{align*}
\sum_{i=1}^n \E \left[T_{n,i}^2 \1{|T_{n,i}| > \eps}  \right] & \le \sum_{i=1}^n \E \left[T_{n,n}^2 \1{|T_{n,n}| > \eps}  \right] \\
&= \sum_{i=1}^n \frac{1}{n} \E \left[(X_n-\mu)^2 \1{|X_n -\mu| > \eps\sqrt{n}}  \right] \\
&= \E \left[(X -\mu)^2 \1{|X-\mu| > \eps\sqrt{n}}  \right].
\end{align*}
Now   $(X-\mu)^2 \1{|X -\mu| > \eps \sqrt{n}} \to 0$ a.s.~as $n \to \infty$ and  $|(X -\mu)^2 \1{|X -\mu| > \eps \sqrt{n}}| \le (X-\mu)^2$ which has $\E [ (X-\mu)^2 ] < \infty$. Thus the dominated convergence theorem yields $\E[(X-\mu)^2 \1{|X-\mu| > \eps \sqrt{n}}] \to 0$ as $n \to \infty$ and the Lindeberg condition is verified.
\end{proof}

\section{Local limit theorem for centre of mass}

To get the recurrence classification of the centre of mass process, only knowing the law of large numbers and central limit theorem is not enough to control the trajectory of the process. We need a more precise result. Our first main result is a \emph{local} central limit theorem. 


Notice that $\Pr(X \in \bb +H \Z^d)=1$ implies $\Pr(S_n \in n\bb + H \Z^d)=1$ which again implies $\Pr(G_n \in n^{-1}(\frac{1}{2}n(n+1)\bb +H \Z^d)=1$. For $\bx \in \R^d$, define
\[p_n ( \bx )  := \Pr ( n^{-1/2} G_n= \bx ),\] 
and
\begin{align}
\label{eq:ndef}
n (\bx) & : = \frac{\exp \{- \tfrac{3}{2} \bx^\tra M^{-1}\bx  \} }{(2\pi)^{d/2} \sqrt{\det (M/3)}} 
.
\end{align}
Also define
\[ \cL_n := \left\{ n^{-3/2} \left( \tfrac{1}{2}n(n+1)\bb + H \Z^d \right) \right\}. \]
Here is our local limit theorem.
\begin{theorem}
\label{thm:LCLT}
Suppose that~\eqref{ass:basicd2} and~\eqref{ass:basicd} hold. Then we have 
\begin{equation}
\lim_{n \to \infty} \sup_{\bx \in \cL_n} \left| \frac{n^{3d/ 2}}{h} p_n(\bx)-n \left(\bx - \frac{(n+1)}{2n^{1/2}} \bmu \right) \right| = 0. 
\label{eq:lclt}
\end{equation}
\end{theorem}

Comparing to Theorem~\ref{thm:LCLTRW}, the local limit theorem for the random walk, $G_n$ has a different scaling factor and a different shift in the normal distribution. The formal difference is one of the key for the difference in recurrence classification of $S_n$ and $G_n$, we will see more in Section 8.3.

\begin{remarks}
(i)
In the case $d=1$, versions of Theorem~\ref{thm:LCLT} are given in~\cite[Lemma~4.3]{MP} and in~\cite[Lemma~1]{KG}; the latter result deals only with the special case of SSRW and only bounds $p_n (\bx)$ up to constant factors. See Chapter 9 for a demonstration that our assumptions are indeed satisfied by SSRW on $\Z^d$ for appropriate choice of $H$ with $h=2$. The proof in \cite{MP} is only a sketch, and the claim that ``it is enough to apply the usual analytical methods'' \cite[p.~515]{MP} does not quite tell the whole story, even in the one-dimensional case. Both~\cite{MP,KG} also give bivariate local limit theorems for $(S_n, G_n)$ (in the case $d=1$). A related result  is \cite[Theorem~4.2]{DH}.

(ii)
If $Z_n := S_n - G_n$, then note that $Z_{n+1} = \frac{n}{n+1} \sum_{i=1}^n (i/n) X_{i+1}$, which means that $Z_{n+1} \eqd \frac{n}{n+1} G_n$, where  `$\eqd$' stands for equality in distribution. Thus Theorem~\ref{thm:LCLT} also yields a local limit theorem for $Z_n$. However, the \emph{processes} $Z_n$ and $G_n$ may behave
very differently: see \cite[Remark 1.1]{CMVW}.
\end{remarks}

\section[Transience and diffusive rate of escape in two or higher dimensions]{Transience and diffusive rate of escape in two or higher dimensions%
\sectionmark{Transience in $2D+$}}
\sectionmark{Transience in $2D+$}

Now we are ready for the recurrence classification for the centre of mass process and we now turn to the almost-sure asymptotic behaviour of $G_n$. We have the following transience result in dimensions greater than one. In particular, Theorem~\ref{thm:classification2} says that $\lim_{n \to \infty} \|G_n \| = +\infty$, a.s., and gives a diffusive rate of escape; in the case of SSRW the result is due to Grill~\cite[Theorem~1]{KG}. 

\begin{theorem}
\label{thm:classification2}
Suppose that $d \geq 2$ and that~\eqref{ass:basicd2} and~\eqref{ass:basicd} hold, and that $\bmu = \0$. Then 
\[ \lim_{n \to \infty} \frac{ \log \| G_n \|}{\log n} = \frac{1}{2}, \as \]
\end{theorem}

As said at the start of this part, it is interesting to investigate why the walk itself is recurrent while the centre of mass process is transient in two dimensions. The real reason behind this is because the centre of mass travels much slower than the original process, with steps of the order of $O(n^{-1/2})$ comparing to $O(1)$ respectively. This suggests that the former process is too slow to return to a specific region in a short amount of time. This idea will be formalized and in fact greatly contribute to the last part of our proof of the recurrence classification.

Obtaining necessary and sufficient conditions for recurrence and transience of $G_n$ is an open problem. See Section~\ref{s:opccom} for details.

\section{One dimension}
\label{s:com1d}

Next we have a recurrence result in one dimension; in the case
of SSRW the fact that $G_n$ returns i.o.~(infinitely often) to a neighbourhood of the origin is due to Grill~\cite[Theorem~1]{KG}.
\begin{theorem}
\label{thm:classification}
Suppose that $d=1$ and that either of the following two conditions
holds.
\begin{itemize}
\item[(i)] Suppose that $\E | X  | \in (0, \infty)$ and $X \eqd - X$.
\item[(ii)] Suppose that~\eqref{ass:basicd} holds and that $\E X = 0$.
\end{itemize}
Then  $\liminf_{n \to \infty} G_n= -\infty$, $\limsup_{n \to \infty} G_n= +\infty$, and
 $\liminf_{n \to \infty} | G_n -x  | = 0$ for any $x \in \R$.
\end{theorem} 

In contrast to Theorem~\ref{thm:classification}, we will show that in the case where
 $\E |X| = \infty$, $G_n$ may be transient. The condition we assume is as follows.

\begin{description}
\item
[\namedlabel{ass:stable}{S}]
Suppose that $X \eqd -X$ and $X$ is in the domain of normal attraction of a symmetric $\alpha$-stable  distribution with $\alpha \in (0,1)$.
\end{description}

\begin{theorem}
\label{thm:stable-transience}
Suppose that~$d=1$ and the lattice condition~\eqref{ass:basicd2} holds, i.e., $\Pr(X \in b + h \Z)=1$ for $b \in \R$ and maximal $h >0$. Suppose also that~\eqref{ass:stable} holds. 
Then 
$\liminf_{n \to \infty} G_n= -\infty$, $\limsup_{n \to \infty} G_n= +\infty$, and
$\lim_{n \to \infty} | G_n | = \infty$.
\end{theorem}
\begin{remark}
The transience here fails in the natural continuous time version of this model. The analogous continuum model, a symmetric $\alpha$-stable L\'evy process for $\alpha \in (0,1)$, $s_t$,
has centre of mass $g_t = \frac{1}{t} \int_0^t s_u \ud u$, and it is surely true that $g_t$ again changes sign i.o., but in this case continuity of $g_t$
implies that $g_t =0$ i.o.
\end{remark}

Chapter~\ref{ch:egcom} verifies our main assumptions for a couple of simple examples. The proof of Theorem~\ref{thm:LCLT} is given in Section~\ref{s:lcltpf}. The proof of Theorem~\ref{thm:classification} uses Proposition~\ref{prop:CLT}, some observations following from the Hewitt--Savage zero--one law, and the fact that in the case where $\E X = 0$ oscillating behaviour is sufficient for $\liminf_{n \to \infty} | G_n - x| = 0$: see Section~\ref{s:com1d} and Section~\ref{s:com1dpf}. The proof of Theorem~\ref{thm:stable-transience} uses another local limit theorem (Theorem~\ref{thm:SLCLT}) and is also presented in Section~\ref{s:com1dpf}. The proof of Theorem~\ref{thm:classification2} relies on Theorem~\ref{thm:LCLT}: see Section~\ref{s:lcltpf}. Section~\ref{s:cf} collects auxiliary results on characteristic functions that we need for the proofs of our local limit theorems.

\chapter{Proofs and technical details}
\label{ch:pfcom}

In this chapter, we will provide the proofs of all the theorems stated in the last chapter. Our main goal is to prove Theorem~\ref{thm:LCLT}, Theorem~\ref{thm:classification2}, Theorem~\ref{thm:classification} and Theorem~\ref{thm:stable-transience}.

\section{A characteristic function result}
\label{s:cf}

We will use the following characteristic function estimation based on Taylor expansion a few times throughout our proof.

\begin{lemma}
\label{lem:estchf}
Suppose that $\E [ \| X \|^2 ] < \infty$. For any $\bt \in \R^d$,
\begin{equation}
\varphi(\bt)=1+i\bt^\tra \E X -\frac{1}{2}\bt^\tra \E[X X^\tra]\bt + \|\bt\|^2 W(\bt), 
\label{eq:estchf}
\end{equation}
where for any $\eps >0$, there exists $\delta >0$ such that $|W(\bt)| \le \eps$ for all $\bt$ with $\|\bt\| \le \delta$.
\end{lemma}
\begin{proof}
Applying \cite[Lemma 3.3.7]{RD} with $x = \bt^\tra X$, we get that if $\E [ \| X \|^n ] < \infty$, then
\begin{align*}
\left| \E \re^{i\bt^\tra X} - \sum_{m=0}^n \E \frac{(i\bt^\tra X)^m}{m!} \right| &\le \E \left| \re^{i\bt^\tra X} - \sum_{m=0}^n \frac{(i\bt^\tra X)^m}{m!} \right| \\
&\le \E \min \left( \frac{| \bt^\tra X |^{n+1}}{(n+1)!} , \frac{2 | \bt^\tra X |^{n} }{n!} \right).
\end{align*}
Taking $n=2$ and rearranging, we get equation~\eqref{eq:estchf}, and 
$|W(\bt)| \le \E Z(t)$,
where $Z(\bt) = \min \{\|\bt\| \|X \|^3, \|X \|^2 \}$. Now $|Z(\bt)| \le \|X \|^2$ and $\E[ \|X \|^2 ] < \infty$. Also we have $|Z(\bt)| \le \|\bt\| \|X \|^3 \to 0$ a.s.\ as $\|\bt\| \to 0$. So the dominated convergence theorem implies that $\E Z(\bt) \to 0$ as $\|\bt \| \to 0$. 
\end{proof}

\section{Proof of the local limit theorem}
\label{s:lcltpf}

This section is devoted to the proof of Theorem~\ref{thm:LCLT}. 
The outline of the proof mirrors the standard Fourier-analytic 
proof of the local central limit theorem for
the random walk: compare e.g.~\cite{BVG}, \cite[Ch.~9]{GK}, \cite[\S 3.5]{RD}, or \cite[Ch.~4]{IL}
for the one-dimensional case, and~\cite[\S 2.2--\S2.3]{LL} for the case of walks on $\Z^d$. The details
of the proof require some extra effort, however.

First we show that it suffices to establish Theorem~\ref{thm:LCLT} in the case
where $\bb = \0$ and $H = I$ (the identity). To see this,
suppose that $X \in \bb + H \mathbb{Z}^d$ and set $\tilde X = H^{-1} ( X - \bb )$.
Then $\tilde X \in \Z^d$. By linearity of expectation, we have 
\[ \tilde \bmu := \E \tilde X = H^{-1} ( \mu - \bb), \text{ and }
\tilde M := \E [ ( \tilde X - \tilde \bmu) ( \tilde X - \tilde \bmu)^\tra ] = H^{-1} M (H^{-1} )^\tra. \]
Note that $(H^{-1})^\tra$ is nonsingular, so $(H^{-1})^\tra \bx \neq \0$ for all $\bx \neq \0$.
Hence for $\bx \neq \0$,
$\bx^\tra \tilde M \bx = \by^\tra M \by$ where $\by = (H^{-1})^\tra \bx \neq \0$, so that
since $M$ is positive definite we have $\bx^\tra \tilde M \bx > 0$; hence $\tilde M$ is also positive definite.
Also, $\tilde S_n := \sum_{i=1}^n \tilde X_i = H^{-1} (S_n - n\bb)$ and
$\tilde G_n := n^{-1} \sum_{i=1}^n \tilde S_i = H^{-1} ( G_n - \frac{n+1}{2} \bb )$.
The assumption
that $H \Z^d$ is minimal for $X$ implies that $\Z^d$ is minimal for $\tilde X$.
Thus the process defined by $\tilde X$ satisfies the hypotheses of Theorem~\ref{thm:LCLT} in the case
where $\bb = \0$ and $H = I$, with mean $\tilde \bmu$ and covariance $\tilde M$,
and that result yields
\begin{equation}
\label{eq:reduction1}
 \lim_{n \to \infty} \sup_{\bx \in n^{-3/2} \Z^d} \left| n^{3d/2} \Pr ( n^{-1/2} \tilde G_n = \bx ) - \tilde n \left( \bx - \frac{(n+1)}{2 n^{1/2}} \tilde \bmu \right) \right| = 0,\end{equation}
where
\[ \tilde n ( \bz ) := \frac{( \det \tilde M / 3 )^{-1/2}}{(2 \pi )^{d/2}}  \exp \left\{ - \frac{3}{2} \bz^\tra \tilde M^{-1} \bz \right\} .\]
But
\[ \Pr ( n^{-1/2} \tilde G_n = \bx ) = \Pr \left( n^{-1/2} G_n = \frac{(n+1)}{2 n^{1/2}} \bb + H \bx \right) = \Pr (n^{-1/2} G_n = \by ) \]
where $\by = \frac{(n+1)}{2 n^{1/2}} \bb + H \bx$
so $\by \in n^{-3/2} ( \frac{1}{2} n(n+1) \bb + H \Z^d )$.
Also,
\begin{align*}   \bx - \frac{(n+1)}{2 n^{1/2}} \tilde \bmu   & = 
\left( H^{-1} \by - \frac{(n+1)}{2n^{1/2}} H^{-1} \bb \right) -\frac{(n+1)}{2n^{1/2}} H^{-1} ( \bmu -\bb) \\
& =  H^{-1} \by -  \frac{(n+1)}{2n^{1/2}} H^{-1}  \bmu .
\end{align*}
Hence, since $\tilde M^{-1} = H^\tra M^{-1} H$ and $\det \tilde M = h^{-2} \det M$,
\begin{align*}
& \quad \tilde n \left( \bx - \frac{(n+1)}{2 n^{1/2}} \tilde \bmu \right) \\
& =\frac{( \det \tilde M / 3 )^{-1/2}}{(2 \pi )^{d/2}} 
 \exp \left\{ - \frac{3}{2} \left(  \by - \frac{(n+1)}{2n^{1/2}} \bmu \right)^\tra M^{-1}
\left(  \by - \frac{(n+1)}{2 n^{1/2}} \bmu \right) \right\} \\
& = h n 
\left(  \by - \frac{(n+1)}{2 n^{1/2}} \bmu \right) .
 \end{align*}
It follows that~\eqref{eq:reduction1} is equivalent to
\[ \lim_{n \to \infty} \sup_{\by \in n^{-3/2} ( \frac{1}{2} n(n+1) \bb + H \Z^d ) } \left| \frac{n^{3d/2}}{h} \Pr ( n^{-1/2}  G_n = \by ) -   n \left(  \by - \frac{(n+1)}{2 n^{1/2}} \bmu \right) \right| = 0 ,\]
which is the general statement of Theorem~\ref{thm:LCLT}. Thus for the remainder of this section we  suppose that $\bb = \0$ and $H = I$; hence $\cL_n = n^{-3/2} \Z^d$.

\subsection{Integral estimates}
After the reduction, the next step of the proof will focus on the estimation of the characteristic function of the centre of mass process. This can done by delicate Fourier-type analysis as follows.

Let $Y_n := \sum_{i=1}^{n}{S_i}$ and thus $G_n = Y_n/n$. 
Recall that $\varphi$ denotes the characteristic function (ch.f.) of $X$,
and let $\Phi_n$ be the ch.f.~of  $n^{-3/2}Y_n$, i.e., for $\bt \in \R^d$,
\[ \varphi (\bt) := \E \re^{i\bt^\tra X} \quad \text{ and } \quad \Phi_n(\bt) := \E \re^{in^{-3/2}\bt^\tra Y_n}. \]
Denoting the smallest eigenvalue of $M$ by $\lambda_{\rm min} (M)$ we have that
 \begin{equation}
\label{eq:positive-definite}
\inf_{\bt \neq \0} \hat \bt^\tra M \hat \bt = \lambda_{\rm min} (M) >0, 
\end{equation}
under assumption~\eqref{ass:basicd}, where $\hat \bt := \bt / \| \bt \|$ for $\bt \neq \0$.
Define
\begin{equation}
\label{f-def}
 f_n ( \bt ) := \exp \left\{ \frac{i (n+1) \bt^\tra \bmu}{2 n^{1/2} } - \frac{ \bt^\tra M \bt}{6} \right\} .
\end{equation}
Our starting point for the proof of the local limit theorem is the following.

\begin{lemma}
\label{lem:lclt-first}
Suppose that~\eqref{ass:basicd} holds and that $\Pr ( X \in \Z^d ) = 1$.
 Then
\begin{align*}
\sup_{\bx \in \cL_n} &\left|  n^{3d/2}  p_n(\bx) -n\left(\bx - \frac{(n+1)}{2n^{1/2}} \bmu \right) \right| \\
& \quad \quad \le \int_{R(n)}
 \left| D_n(\bt) \right| \ud \bt  +  
\int_{R^{\rc}(n)} \exp \left\{ -\frac{ \lambda_{\rm min} (M)}{6}  \| \bt \|^2 \right\} \ud \bt ,
\end{align*}
where $R(n) := [-\pi n^{3/2} , \pi n^{3/2} ]^{d}$,
$R^\rc (n) := \R^d \setminus R(n)$,
 and $D_n(\bt):= \Phi_n(\bt) - f_n (\bt)$. 
\end{lemma}
\begin{proof}
For a lattice random variable $W \in \Z^d$, by the 
inversion formula for the characteristic function (see e.g.~\cite[Corollary~2.2.3, p.~29]{LL})
we have that 
\begin{equation}
\Pr(W = \by ) = \frac{1}{(2\pi )^{d}} \int_{[-\pi,\pi]^d} \re^{-i\bu^\tra \by} \E \bigl[ \re^{i\bu^\tra W} \bigr] \ud \bu, 
\label{eq:formula}
\end{equation}
for $\by \in \Z^d$. Now we have for $\bx \in \cL_n$,
$
p_n(\bx) = \Pr ( Y_n= n^{3/2} \bx )$,
so applying~\eqref{eq:formula} with $W=Y_n \in \Z^d$, we get for $\bx \in \cL_n$ that
\begin{align*}
p_n(\bx )  = 
\frac{1}{(2\pi )^{d}} 
\int_{[-\pi,\pi]^d} \re^{ -i n^{3/2} \bu^\tra \bx } \E\bigl[ \re^{i\bu^\tra Y_n }\bigr] \ud \bu.
\end{align*}
Using the substitution $\bu =n^{-3/2} \bt$, we obtain
\begin{equation}
 n^{3d/2}  p_n(\bx ) = \frac{1}{(2\pi )^{d}} \int_{[-\pi n^{3/2},\pi n^{3/2}]^d} \re^{-i\bt^\tra \bx} \Phi_n(\bt)\ud \bt. 
\label{eq:cal11} 
\end{equation}
On the other hand, since the probability density 
$n(\bx - \frac{(n+1)}{2n^{1/2}} \bmu)$, with $n(\blob)$ as defined at~\eqref{eq:ndef}, corresponds to the ch.f.~$f_n (\bt)$
as defined at~\eqref{f-def}, the inversion formula for densities yields
\begin{equation}
n \left( \bx - \frac{(n+1)}{2n^{1/2}} \bmu \right)=\frac{1}{(2\pi)^d}\int_{\R^d} \re^{-i\bt^\tra \bx} f_n (\bt) \ud \bt, 
\label{eq:cal13}
\end{equation}
for $\bx \in \R^d$. Now we subtract~\eqref{eq:cal13} from~\eqref{eq:cal11} to get
\pagebreak
\begin{align*}
& \quad n^{3d/2}  p_n(\bx)-n \left( \bx - \frac{(n+1)}{2n^{1/2}} \bmu \right) \\
&= \frac{1}{(2\pi)^d} \int_{R(n)} \re^{-i\bt^\tra \bx} D_n (\bt) \ud \bt - \frac{1}{(2\pi)^d} \int_{R^\rc(n)} \re^{-i\bt^\tra \bx} f_n (\bt) \ud \bt .
\end{align*} 
Thus, by the triangle inequality with the estimates $\pi>1$ and $|\re^{-i \bt^\tra \bx}| \le 1$, we obtain 
\[ \sup_{\bx \in \cL_n} \left|  n^{3d/2}  p_n(\bx)-n \left( \bx - \frac{(n+1)}{2n^{1/2}} \bmu \right) \right|  
\le \int_{R(n)} 
 \left| D_n(\bt) \right| \ud \bt  +  
\int_{R^{\rc}(n)} \exp \left\{ -\frac{\bt^\tra M \bt}{6} \right\} \ud \bt ,
 \]
which with~\eqref{eq:positive-definite} yields the statement in the lemma.
\end{proof}

To prove Theorem~\ref{thm:LCLT} we must show that the right-hand side of the inequality in Lemma~\ref{lem:lclt-first} approaches $0$ when $n \to \infty$.
To do so, we bound $D_n (\bt)$ in different regions for $\bt$.
Observing that $Y_n= \sum_{i=1}^n S_i=\sum_{j=1}^n (n-j+1)X_j$, we see
\[
\Phi_n(\bt) = \E\left[\exp\left\{ i n^{-3/2} \bt^\tra Y_n\right\} \right] = \E\left[\exp\left\{ in^{-3/2}\sum_{j=1}^{n}(n-j+1)\bt^\tra X_j\right\} \right].
\]
For fixed $n$, $\sum_{j=1}^{n}(n-j+1) \bt^\tra X_j \eqd \sum_{j=1}^{n}j \bt^\tra X_j$, so that 
\[
\Phi_n(\bt) = \E\left[\exp\left\{ in^{-3/2} \sum_{j=1}^{n}j \bt^\tra X_j\right\} \right] = \prod_{j=1}^{n}\E\left[\exp\left\{ i n^{-3/2}j \bt^\tra X_j \right\} \right] .
\]
Hence we conclude that for $\bt \in \R^d$,
\begin{equation}
\label{eq:Phi-phi}
\Phi_n( \bt) 
= \prod_{j=1}^{n} 
\varphi(n^{-3/2}j \bt ).
\end{equation}
To study $\Phi_n$ we 
require certain characteristic function estimates, presented in Section~\ref{s:cf}.

We partition $R(n)$ into four regions defined as follows:
\begin{align*}
R_1 & :=  [-A,A]^d \\
R_2(n) & :=  [-\delta \sqrt{n}, \delta \sqrt{n}]^d \setminus R_1 \\
R_3(n) & :=  [-\pi \sqrt{n} , \pi \sqrt{n} ]^d \setminus (R_1 \cup R_2 (n) ) \\
R_4(n) & := R (n) \setminus (R_1 \cup R_2 (n) \cup R_3 (n))
\end{align*}
where constants $A \in (0,\infty)$ and $\delta \in (0, \pi)$ will be chosen later. We also denote the corresponding integrals $I_k(n) :=\int_{R_k} \left| D_n(t) \right| \ud t$, $k=1,2,3,4$. 

\begin{lemma}
\label{lem:parts1}
For $\delta >0$ sufficiently small, the following statements are true.
\begin{itemize}
\item[(i)] For any $A \in \RP$, $\lim_{n \to \infty} |I_1(n)|=0$,
\item[(ii)] $\lim_{A \to \infty} \sup_n |I_2(n)|=0$, 
\item[(iii)] $\lim_{n \to \infty} |I_3(n)|=0$,
\item[(iv)] $\lim_{n \to \infty} |I_4(n)|=0$.
\end{itemize}
\end{lemma}
We will combine all the estimates at the end of the argument. 
\begin{proof}[Proof of Lemma~\ref{lem:parts1}]
First we aim to show that
\begin{align}
\lim_{n \to \infty} \sup_{\bt \in R_1} |D_n(\bt)| =0 .
\label{eq:sup1}
\end{align} 
Since $\E X =\bmu$ and $\E[(X -\bmu) (X-\bmu)^\tra]=M$, we have
$\E [ X X^\tra ] = M + \bmu \bmu^\tra$,
so that
Lemma~\ref{lem:estchf} implies, uniformly over 
$\bt \in R_1$, as $n \to \infty$, 
\begin{align*}
\prod_{j=1}^{n} \varphi(n^{-3/2}j\bt) 
&= \exp \left\{ \sum_{j=1}^{n} \log \left[ 1 + A( n, j, \bt ) + o(n^{-1}) \right] \right\},
\end{align*}
where 
\begin{equation}
\label{A-def}
 A (n, j, \bt) :=   i n^{-3/2} j   \bt^\tra \bmu - \frac{1}{2} n^{-3} j^2 \bt^\tra (M + \bmu \bmu^\tra) \bt  .\end{equation}
Taylor's theorem for a complex variable shows that for a constant $C < \infty$,
\begin{equation}
\label{eq:complex-taylor}
\left| \log ( 1+ z ) - \left( z - \frac{z^2}{2} \right) \right| \leq C | z |^3 ,
\end{equation}
for $z$ in an open disc containing $0$.
Note from~\eqref{A-def} that
\begin{equation}
\label{A-sq}
 A (n, j, \bt)^2 = - n^{-3} j^2 \bt^\tra \bmu \bmu^\tra \bt + \Delta_0 (n,j,\bt) , \end{equation}
where $\max_{1 \leq j \leq n} \sup_{\bt \in R_1} | \Delta_0 (n ,j, \bt) | = O ( n^{-3/2} )$.
Then, by~\eqref{eq:Phi-phi}, \eqref{eq:complex-taylor}, \eqref{A-sq},
and the fact that
$\max_{1 \leq j \leq n} \sup_{\bt \in R_1} | A (n ,j, \bt) | = O ( n^{-1/2} )$, 
 it follows that
\[
\Phi_n (\bt) = \exp \left\{ \sum_{j=1}^{n} \left(i n^{-3/2} j   \bt^\tra \bmu -\frac{1}{2}n^{-3}j^2 \bt^\tra M \bt \right) +  \Delta_0 ( n ,\bt)   \right\}, 
\]
where $\sup_{\bt \in R_1} | \Delta_0 ( n ,\bt) | \to 0$.
 Elementary algebra gives $\sum_{j=1}^{n} j  = \frac{1}{2}n(n+1)$ and $\sum_{j=1}^{n} j^2 = \frac{1}{6}n(n+1)(2n+1)$,
 so we obtain the estimate 
\begin{align}
\nonumber
\Phi_n(\bt )= \exp \left\{ \frac{i (n+1) \bt^\tra \bmu}{2n^{1/2}} -\frac{\bt^\tra M \bt}{6} + \Delta_1 ( n ,\bt) \right\},
\end{align} 
where $\sup_{\bt \in R_1} | \Delta_1 ( n ,\bt) | \to 0$ as $n \to \infty$.
Hence, by~\eqref{f-def},
\[ | D_n (\bt ) | = | \Phi_n (\bt) - f_n (\bt) | \leq \left| 1 - \exp \{ \Delta_1 (n, \bt) \} \right| ,\]
which establishes~\eqref{eq:sup1} and proves part (i) of the lemma. 
For part (ii), suppose that $\bt \in [0, \delta n^{1/2}]^d$. Fix $\eps>0$.
Then for $1 \leq j \leq n$, we have $\| n^{-3/2} j \bt \| \leq \delta d^{1/2}$. 
Thus, from  Lemma~\ref{lem:estchf},
\begin{align*}
\varphi ( n^{-3/2} j \bt ) & =  1 + A (n,j ,\bt) + \Delta_1 (n,j,\bt) ,
\end{align*}
where $A (n, j ,\bt)$ is as defined at~\eqref{A-def}, and
$|\Delta_1 ( n, j, \bt ) | \leq \eps n^{-1} \| \bt \|^2$
for all $\bt \in  [0, \delta n^{1/2}]^d$ and $\delta$ sufficiently small.
Also note that $| A (n , j, \bt ) | \leq C n^{-1/2} \| \bt \|$, so that
\begin{equation}
\label{A-cubed}
 | A (n,j,\bt) |^3 \leq C n^{-3/2} \| \bt \|^3 \leq C' \delta n^{-1} \| \bt \|^2 \leq \eps n^{-1} \| \bt \|^2 ,\end{equation}
for $\delta$ sufficiently small; here $C$ and $C'$ are constants that do not depend on $\delta$.
Thus we may apply~\eqref{eq:complex-taylor} to obtain
\begin{align*}
\prod_{j=1}^n \varphi ( n^{-3/2} j \bt ) & = \exp \left\{ \sum_{j=1}^n \log \left[ 1 + A (n,j ,\bt) + \Delta_1 (n,j,\bt) \right] \right\} \\
& = \exp \left\{ \sum_{j=1}^n \left( A(n,j,\bt) - \frac{1}{2} A(n,j,\bt)^2 \right) + \Delta_1 (n, \bt) \right\} ,\end{align*}
where $| \Delta_1 (n, \bt ) | \leq \eps \| \bt \|^2 $ for $\delta$ sufficiently small.
Here~\eqref{A-sq} holds, where now, for all $\bt \in  [0, \delta n^{1/2}]^d$, 
similarly to~\eqref{A-cubed}, $|\Delta_0 (n, j, \bt ) | \leq \eps n^{-1} \| \bt \|^2$
for $\delta$ sufficiently small.
So, for $\delta$ sufficiently small, for $\bt \in R_2 (n)$,
\[ \prod_{j=1}^n \varphi ( n^{-3/2} j \bt )  = \exp \left\{  \frac{ i (n+1) \bt^\tra \mu}{2n^{1/2}}
- \frac{\bt^\tra M \bt}{6} + \Delta_2 ( n , \bt ) \right\} ,\]
where $| \Delta_2 (n , \bt ) | \leq \eps \| \bt \|^2$ for all $n$ sufficiently large.
Suppose $\eps \in (0, \lambda_{\rm min} (M)/ 12 )$, so that, by \eqref{eq:positive-definite},
$\bt^\tra M \bt \geq 12 \eps \| \bt \|^2$. Then
\begin{align*}
| \Phi_n (\bt ) |   = \left| \exp \left\{ - \frac{ \bt ^\tra M \bt}{6} + \Delta_2 ( n , \bt ) \right\} \right|  
  \leq  \exp \{ -   \eps \| \bt \|^2 \} .\end{align*}
So we have 
\begin{align*} 
I_2 (n) & \leq \int_{R_2 (n)} | \Phi_n ( \bt ) | \ud \bt + \int_{R_2 (n) } | f_n (\bt ) | \ud t \\
& \leq 2\int_{\R^d \setminus R_1} \exp \left\{ - \eps \| \bt \|^2 \right\} \ud \bt,
\end{align*}
for $\delta$ sufficiently small and $n$ sufficiently large. This yields part (ii) of the lemma. 

Now we proceed to estimate $I_3(n)$. 
First note that, by~\eqref{eq:Phi-phi}, 
\begin{equation}
|\Phi_n(\bt)| = \prod_{j=1}^{n}|\varphi(n^{-3/2}j\bt)| \le \prod_{j=\lceil n/2 \rceil}^{n}|\varphi(n^{-3/2}j\bt)|. 
\label{eq:cal20}
\end{equation}
For any $\bt \in R_3(n)$, we have
 $n^{-3/2}j\bt \in  [-\pi j / n,  \pi j / n]^d \setminus   [-\delta j / n, \delta j / n]^d$. 
In particular
\[ \bigcup_{j=\lceil n/2 \rceil}^n \{ n^{-3/2}j\bt \} \subset  [-\pi , \pi  ]^d \setminus [-\delta  / 2, \delta / 2]^d . \]
Thus we may apply the final statement in
Lemma~\ref{lem:equivalent} for some $\rho$ sufficiently small 
to obtain
\[ \sup_{\bt \in R_3(n)} \sup_{\lceil n/2 \rceil \leq j \leq n} | \varphi ( n^{-3/2} j \bt ) | \leq \re^{-c_\rho} ,\]
for some $c_\rho >0$.
Hence from~\eqref{eq:cal20} we have  
\[ \sup_{\bt \in R_3 (n)} |\Phi_n(\bt)| \le \re^{-n c_\rho / 2} .\]
It follows that
\begin{align*}
|I_3(n)| 
& \le \int_{[ - \pi \sqrt{n}, \pi \sqrt{n} ]^d} \re^{-n c_\rho / 2}
+ \int_{\R^d \setminus [- \delta \sqrt{n}, \delta \sqrt{n}  ]^d} \exp\left\{-\frac{\bt^\tra M \bt}{6}\right\}  \ud \bt \\ 
& \leq (2\pi)^d n^{d/2} \re^{-n c_\rho / 2}
+ \int_{\R^d \setminus [-\delta \sqrt{n}, \delta \sqrt{n} ]^d} \exp\left\{-\frac{\lambda_{\rm min} (M)}{6} \| \bt \|^2 \right\}  \ud \bt,
\end{align*}
using~\eqref{eq:positive-definite}.
This gives part (iii) of the lemma. 

It remains to estimate $I_4(n)$. Fix $\bt \in R_4(n)$, and consider sets
\begin{align*}
\Lambda_n(\bt) = \left\{ n^{-3/2} j \bt : j \in \{1,2, \ldots, n\} \right\} , \text{ and } 
L_n(\bt) &= \left\{ n^{-3/2} u \bt : 1 \le u \le n \right\} .
\end{align*}
Recall that $S_H := 2\pi \Z^d$ in the case $H = I$, and, for $\rho >0$, define $S_H(\rho) := \cup_{\by \in S} B(\by;\rho)$, 
where $B(\by; \rho)$ is the open Euclidean ball of radius $\rho$ centred at $\by \in \R^d$.
Define $N_n(\bt) := | \Lambda_n(\bt) \setminus S_H(\rho)|$.
  Lemma~\ref{lem:equivalent} and~\eqref{eq:Phi-phi} show that
\begin{equation}
\label{eq:counting_bound}
|\Phi_n(\bt)| = \prod_{j=1}^{n}|\varphi(n^{-3/2}j\bt)| \le \exp \{ - c_\rho N_n(\bt)\},
\end{equation}
for some positive constant $c_\rho$. 
We aim to show that $N_n (\bt)$ is bounded below by a constant times $n$.
To do this we use a counting argument related  to one used in~\cite[Lemma~4.4]{DH}.

Let $K_n(\bt)$ be the number of $\bx \in S_H$ such that $B(\bx ; \rho) \cap L_n(\bt) \neq \emptyset$. 
Set
$\nu := n^{-3/2}\|\bt\|$. As $\bt \in [-\pi n^{3/2}, \pi n^{3/2}]^d \setminus [-\pi n^{1/2}, \pi n^{1/2}]^d$, we have
\begin{equation}
\label{nu-bounds}
\frac{\pi}{n} \le \nu \le \pi \sqrt{d}.
\end{equation}
Take $\rho = \pi / 8$. We claim that between any two balls of $S_H(\rho)$ that intersect $L_n(\bt)$
there is at least one point of $\Lambda_n (\bt)$. 
Write $\by_j = n^{-3/2} j \bt$ for $j \in \{1,\ldots, n\}$. Suppose $i_1, i_2 \in \{1,\ldots, n\}$
with $i_1 < i_2$ and $\bx_1, \bx_2 \in S_H$ with $\bx_1 \neq \bx_2$ are such that $\by_{i_1} \in B (\bx_1 ; \rho)$
and $\by_{i_2} \in B (\bx_2, \rho)$. To prove the claim we need to show that there exists $j$ with $i_1 < j < i_2$
such that $\by_j \notin S_H(\rho)$. First note that since $n^{-3/2} \bt \in [ -\pi, \pi]^d$ and 
$\by_{i_1} \in B (\bx_1 ; \rho)$, the point $\by_{i_1 +1}$ must lie in the box $Q(\bx_1) = \bx_1 + [-9\pi/8,9\pi/8]^d$.
As $9\pi/8 < 15\pi/8 = 2\pi - \rho$, the box $Q(\bx_1)$ does not intersect any balls in $S_H(\rho)$ other than $B (\bx_1 ; \rho)$.
There are two cases. Either (i) $\by_{i_1 +1} \notin B (\bx_1 ; \rho)$, or (ii)
$\by_{i_1 +1} \in B (\bx_1 ; \rho)$. In case (i) the claim is proved. In case (ii), we have $\nu \leq 2 \rho$, and since
$B (\bx_1 ; 3 \rho)$ is contained in $Q(\bx_1)$, there is some $j$ with $i_1 + 1 < j < i_2$ such that $\by_j \notin S_H(\rho)$,
proving the claim.
Hence
\begin{equation}
\label{N-lower-bound2}
 N_n (\bt) \geq   K_n (\bt) - 1  
 .\end{equation}
The total length of $L_n(\bt)$ is less than $\nu n$,
and each segment of $L_n (\bt)$ between neighbouring balls that intersect $L_n (\bt)$
has length at least $2\pi - 2\rho$, so 
$(K_n (\bt) -1 ) (2\pi - 2\rho) \leq \nu n$, or, equivalently,
\begin{equation}
\label{K-upper-bound}
K_n(\bt) \le \frac{4\nu n}{7\pi} +1.
\end{equation}
Moreover, each ball of $S_H(\rho)$ that intersects $L_n (\bt)$ contains
at most $2 \rho/ \nu +1$ points of $\Lambda_n (\bt)$, so that 
the number of points in $\Lambda_n(\bt) \cap S_H(\rho)$ satisfies
\begin{equation}
\label{N-lower-bound3}
n-N_n(\bt) \leq K_n(\bt) \left(\frac{\pi}{4\nu}+1\right).
\end{equation}
Let $\eps >0$ be a constant. We consider the following two cases. \\
\emph{Case 1:} $K_n(\bt) \le \eps n \nu$.
In this case we have from~\eqref{N-lower-bound3} and~\eqref{nu-bounds} that
\begin{equation*}
N_n(\bt) \ge n- \frac{\pi \eps}{4}n - \eps n\nu \ge n- \frac{\pi}{4}\eps n - \eps n \pi \sqrt{d}  \ge \eps n,
\end{equation*}
for $\eps$ small enough. \\
\emph{Case 2:} $K_n(\mathbf{t}) > \eps n \nu$.
If $\nu \ge \frac{1}{2}$, then we have from~\eqref{N-lower-bound2} that,
\begin{equation*}
N_n(\mathbf{t}) \ge K_n(\mathbf{t})-1 \ge (\eps/3) n ,
\end{equation*}
for $n$ sufficiently large.
On the other hand, if $\nu < \frac{1}{2}$, then~\eqref{N-lower-bound3} and~\eqref{K-upper-bound} show 
that
\begin{align*}
N_n(\bt) & \ge  n- \left(\frac{4\nu n}{7 \pi } + 1\right)\left(\frac{\pi}{4\nu} +1\right) \\
& = \frac{6n}{7} - \frac{\pi}{4 \nu} - \frac{4 \nu n}{7 \pi} -1 \\
& \geq \frac{6n}{7} - \frac{n}{4} - \frac{2 n}{7 \pi} -1 ,
\end{align*}
by~\eqref{nu-bounds}.
Thus we have shown that, in any case,
$N_n(\bt) \ge \eps n$
for some constant $\eps >0$ and all $n$ sufficiently large. Thus from~\eqref{eq:counting_bound} we conclude that 
\begin{align*}
| I_4 (n) | & \leq \int_{R_4(n)} | \Phi_n (\bt) | \ud \bt + \int_{R_4(n)} | f_n (\bt) | \ud \bt \\
& \le ( 2\pi n^{3/2} )^d  \exp\left\{-\eps c_\rho n \right\} + \int_{\R^d \setminus [ - \pi n^{1/2},
\pi n^{1/2} ]^d } \exp \left\{ -\frac{\lambda_{\rm min} (M)}{6} \| \bt \|^2   \right\} \ud \bt .
\end{align*}
Hence we have proved the last statement in Lemma~\ref{lem:parts1}.
\end{proof}

Now we can gather all our estimates and complete the proof of Theorem~\ref{thm:LCLT}.

\begin{proof}[Proof of Theorem~\ref{thm:LCLT}]
We have from Lemma~\ref{lem:lclt-first} that
\begin{align}
\sup_{\bx \in \cL_n} \left|\frac{n^{3d/2}}{h} p_n(\bx) -n(\bx) \right| 
&\le \sum_{k=1}^{4} I_k(n) + \int_{R^\rc (n)} \exp \left\{ -\frac{\lambda_{\rm min} (M)}{6} \| \bt \|^2 \right\} \ud \bt.
\end{align}
Clearly the integral term tends to $0$ as $n \to \infty$, while Lemma~\ref{lem:parts1} shows that
$|I_3 (n) + I_4 (n)| \to 0$. Lemma~\ref{lem:parts1} also shows that for any $\eps >0$, we can choose $A$ large enough so that 
$|I_2(n)| \le \eps$ for all $n$, and hence $\limsup_{n \to \infty}|I_1(n)+I_2(n)|\le \eps$. Hence $| I_1 (n) + I_2 (n) | \to 0$ as well. This completes the proof of the theorem. 
\end{proof}

\section{Proofs for two or higher dimensions}

This section is devoted to the proof of Theorem~\ref{thm:classification2}. 
The idea is to use the local limit theorem to control (via Borel--Cantelli) the visits of $G_n$
to a growing ball, along a subsequence of times suitably chosen so that
the slow movement of the centre of mass controls the trajectory between the times
of the subsequence as well. Here is our estimate on the deviations.

\begin{lemma}
\label{lem:diff}
Suppose that~\eqref{ass:basicd} holds and that $\bmu = \0$. 
Let $a_n=\lceil n^\beta \rceil$  for some $\beta >1$. Then, for any $\varepsilon>0$, a.s. for all but finitely many $n$,  
\[
\max_{a_n \le m \le a_{n+1}} \|G_m-G_{a_n} \| \le n^{\frac{\beta}{2}-1+\varepsilon}.
\]
\end{lemma}
\begin{proof}
We use the crude bound that for any $\eps>0$,
 $\| S_n \| \leq n^{(1/2)+\eps}$ all but f.o., a.s. It follows from the triangle inequality that
\begin{equation}
\label{upperbound}
\|G_n\| \le \frac{1}{n}  \sum_{i=1}^n \| S_i  \| \le \max_{1 \le i \le n} \|S_i\| \le n^{(1/2)+\eps},
\end{equation}
all but f.o., a.s.
Next, by the triangle inequality again, for any $\eps>0$, a.s., all but f.o.,  
\begin{equation}
\|G_{n+1}-G_n\| = \left\| \frac{S_{n+1}-G_n}{n+1} \right\| \le \frac{\|S_{n+1}\|}{n+1} + \frac{\|G_n\|}{n+1} \le n^{-(1/2)+\eps} .
\label{eq:cal7}
\end{equation}
It follows that for any $\eps>0$, a.s., all but f.o., 
\begin{align*}
\max_{a_n \le m \le a_{n+1}} \|G_m-G_{a_n} \| &\le \left( a_{n+1} - a_n \right) \max_{a_n \le m \le a_{n+1}-1} \|G_{m+1}-G_m\| , \end{align*}
where $a_{n+1} -a_n   \leq (n+1)^\beta -n^\beta +1 = O (n^{\beta -1} )$, and, a.s., all but f.o., by~\eqref{eq:cal7},
\begin{align*}  
\max_{a_n \le m \le a_{n+1}-1} \|G_{m+1}-G_m\| & \leq a_n^{-(1/2)+\eps} = O ( n^{-(\beta/2) + \beta \eps} ) .\end{align*}
 Since $\eps>0$ was arbitrary, the result follows.
\end{proof}

Now we are ready to prove Theorem~\ref{thm:classification2}.

\begin{proof}[Proof of Theorem~\ref{thm:classification2}.]
First, given the upper bound in equation~\eqref{upperbound}, we only need to show that for any $\eps>0$, a.s., for all but finitely many $n$, 
\begin{equation}
\label{lowerbound}
\|G_n\| \ge n^{(1/2)-\eps}.
\end{equation}
Let $B(r)$ denote the closed Euclidean ball, centred at the origin, of radius $r>0$.
We show that for any $\gamma \in (0, 1/2)$, $G_n$ will return to the ball $B(n^\gamma)$ only f.o. 
To do this, we show that along a suitable subsequence $a_n = \lceil n^\beta \rceil$, $\beta >1$,
$G_{a_n}$ returns to the ball
$B(2a_n^\gamma)$ only f.o., and Lemma~\ref{lem:diff} controls the trajectory between the instants of the subsequence.

First, we claim that
\begin{equation}
\Pr(G_n \in B(2n^\gamma) ) \le Cn^{d\left(\gamma- \frac{1}{2}\right)}, 
\label{eq:correctorder}
\end{equation}
for sufficiently large $n$ and some constant $C$.
Then 
\[
\sum_{n=1}^{\infty} \Pr(G_{a_n} \in B(2a_n^\gamma) ) \le C \sum_{n=1}^{\infty} n^{\beta d\left(\gamma- \frac{1}{2}\right)}. 
\]
Assuming that 
\begin{equation}
\beta > \frac{2}{d(1-2\gamma)} \label{betacon1}
\end{equation}
this sum converges, so the Borel--Cantelli lemma shows that $G_{a_n} \notin B(2a_n^\gamma)$ for all but finitely many $n$, a.s.
It then follows from Lemma~\ref{lem:diff} that between any $a_n$ and $a_{n+1}$ with $n$ sufficiently large,
the trajectory deviates by at most $n^{(\beta/2)-1+\eps}$. In particular, the trajectory between times $a_n$ and $a_{n+1}$ will not visit $B(a_n^\gamma)$ if we ensure that
$n^{(\beta/2)-1+\eps} < a_n^\gamma$. (See Figure~\ref{fig1}.) The latter condition can be achieved (for sufficiently small choice of $\eps$) if $(\beta/2) -1 < \beta \gamma$,
i.e., $\beta <  (\frac{1}{2} -\gamma )^{-1}$. Combined with~\eqref{betacon1} we see that we must choose $\beta >1$ such that
\[
\frac{2}{d(1-2\gamma)} < \beta < \frac{2}{(1-2\gamma)},
\]
which is possible for any $\gamma \in (0,1/2)$, provided $d \geq 2$.

Consider $n$ such that $a_m \leq n < a_{m+1}$; then we have shown that a.s., for all but finitely many $n$,
\[ \| G_n \| \geq a_m^\gamma  \geq m^{\beta \gamma} 
\geq \left( \frac{m^{\beta \gamma} }{2 (m+1)^{\beta \gamma} } \right) a^\gamma_{m+1}   .\]
In particular, for all $n$ sufficiently large, $\| G_n \| \geq (1/4) n^\gamma$, 
which establishes~\eqref{lowerbound}.

\begin{figure}[!h]
\center
\includegraphics[width=60mm]{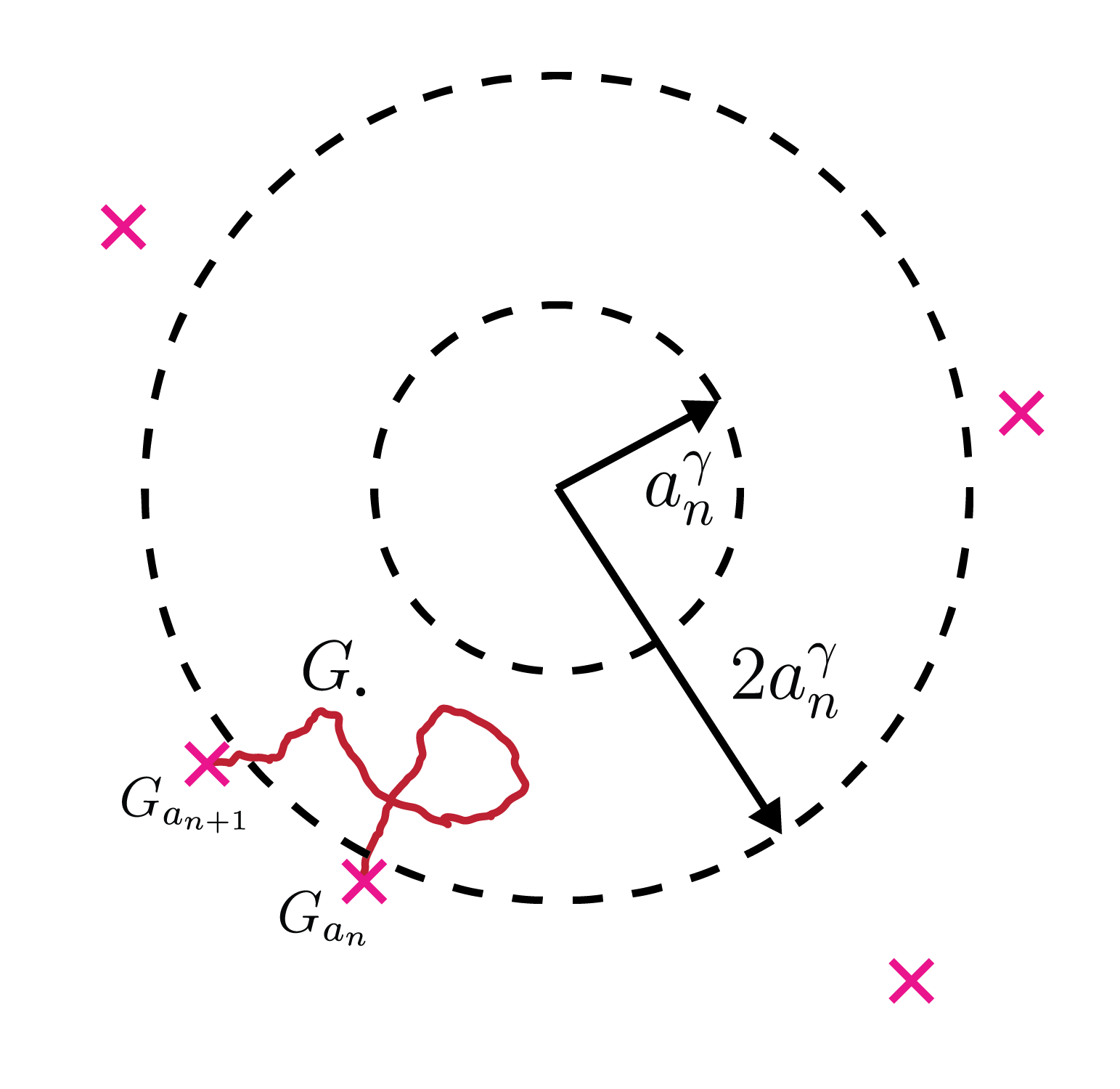}
\caption{Controlling $G_n$ along a subsequence.}
\label{fig1}
\end{figure}

It remains to prove the claim~\eqref{eq:correctorder}; here we use our local limit theorem.
First note that
\[
\Pr(G_{n} \in B(2n^\gamma) )=\Pr(n^{-1/2}G_{n} \in n^{-1/2} B(2n^\gamma) ).
\]
The ball $n^{-1/2} B(2n^\gamma)$ has radius $O ( n^{\gamma-\frac{1}{2}} )$, and the lattice spacing of $\mathcal{L}_n$ is
of order $n^{-3/2}$, so 
$n^{-1/2} B(2n^\gamma)$  contains $O(n^{d(\gamma+1)})$ lattice points. 
From Theorem~\ref{thm:LCLT}, we also know that  for all $x \in \mathcal{L}_n$, 
$
\Pr(n^{-1/2}G_n = x) = O (n^{-3d/2} )
$.
Summing up over all $x \in n^{-1/2} B(2n^\gamma)$ we get
\[
\Pr(n^{-1/2} G_{n} \in n^{-1/2}B(2n^\gamma)) = O \left(n^{-3d/2} \times n^{d(\gamma+1)} \right)= O\left(n^{d\left(\gamma- \frac{1}{2}\right)}\right),
\]
establishing~\eqref{eq:correctorder}. This completes the proof.
\end{proof}

\section{Proofs for one dimension}
\label{s:com1dpf}

\subsection{Recurrence}

We will start with a couple of general observations that help the proof of Theorem ~\ref{thm:classification} in Section~\ref{s:com1dpf}. Recall the definition of an exchangeable event from Section 6.3.


For our one dimensional centre of mass process, we have the following result.

\begin{lemma}
\label{lem:exchangeable}
Let $d=1$.
For any $x \in \R$, the event $\left\{ \limsup_{n \to \infty} G_n \ge x \right\}$ is exchangeable.
\end{lemma}
\begin{proof}
For any $x \in \R$, we notice that for any fixed positive integer $k$,
\begin{align*}
&\quad \left\{ \limsup_{n \to \infty} G_n \ge x \right\} \\
&= \left\{ \limsup_{n \to \infty} \left[ \frac{1}{n}(S_1 +S_2 + \cdots + S_k) + \frac{1}{n}(S_{k+1} +S_{k+2} + \cdots + S_n) \right] \ge x \right\} \\
&= \left\{ \limsup_{n \to \infty}  \frac{1}{n}(S_{k+1} +S_{k+2} + \cdots + S_n) \ge x \right\},
\end{align*}
which is invariant under permutations of $X_1, X_2, \ldots, X_k$.
\end{proof}

The next result shows that $G_n$ can either be trivial, transient, or oscillating.

\begin{lemma}
\label{lem:4case}
Let $d=1$. One and only one of the following will occur with probability $1$.
\begin{itemize}
\item[(i)] $G_n=0$ for all $n$.
\item[(ii)] $G_n \to \infty$.
\item[(iii)] $G_n \to -\infty$.
\item[(iv)] $-\infty = \liminf_{n\to\infty} G_n < \limsup_{n\to\infty} G_n = \infty$.
\end{itemize} 
\end{lemma}
\begin{proof}
We adapt the proof of Theorem~4.1.2 in \cite{RD}.
Lemma~\ref{lem:exchangeable} and the Hewitt--Savage zero--one law (Theorem 6.3.1) imply $\limsup_{n\to\infty} G_n = \ell$, a.s.,
for some  $\ell \in [-\infty, \infty]$.
Let $G'_n := \frac{n+1}{n}(G_{n+1}-X_1) = \sum_{i=1}^{n}\frac{n-i+1}{n}X_{i+1}$. Recalling~\eqref{eq:weighted-sum}, we see the sequence $(G'_n)$ has the same distribution as $(G_n)$. 
So taking $n \to \infty$ in $\frac{n}{n+1} G'_n = G_{n+1}-X_1$ we obtain
$\ell=\ell-X_1$, a.s., implying $X_1=0$ a.s.~if $\ell$ is finite, which is case (i). 
Otherwise, $\ell = -\infty$ or $+\infty$. A similar argument applies to $\liminf_{n\to\infty} G_n$. 
The 3 possible combinations ($\limsup_{n\to \infty} G_n = -\infty$ and $\liminf_{n\to\infty} G_n = \infty$ being impossible) give (ii), (iii), and (iv).
\end{proof}

Clearly cases (ii)~and (iii) of Lemma~\ref{lem:4case} are transient; case~(iv), when the walk oscillates, is the most interesting case. The next result shows that oscillating behaviour is enough to ensure recurrence provided that $\E X =0$.

\begin{lemma}
\label{lem:change_sign}
Suppose that $d=1$ and $\E  X =0$. Suppose that 
$\limsup_{n \to \infty} G_{n}= + \infty$ and $\liminf_{n \to \infty} G_{n}= - \infty$.
Then, for any $x \in \R$, $\liminf_{n \to \infty}|G_n - x|=0$, a.s.
\end{lemma}
\begin{proof}
Fix $\eps >0$. 
Since $S_n /n \to 0$ a.s.~and, by Proposition~\ref{prop:LLN}, $G_n /n \to 0$ a.s., we have
\[
G_{n+1}-G_{n} = \frac{S_{n+1}-G_n}{n+1} \to 0 , \as 
\]
Hence $| G_{n+1} - G_n | < \eps$ all but f.o.~(finitely often).
For any $x \in \R$, $\limsup_{n \to \infty} G_{n}= + \infty$ and $\liminf_{n \to \infty} G_{n}= - \infty$
implies that there are infinitely many $n$ for which $G_n - x$ and $G_{n+1} -x$ have opposite signs.
Hence $| G_n - x | < \eps$ infinitely often.
\end{proof}

The next result shows that $G_n$ does oscillate when~\eqref{ass:basicd} holds.  

\begin{lemma}
\label{lem:change_sign2}
Suppose that $d=1$, that~\eqref{ass:basicd} holds, and that $\E X = 0$.
 Then $\limsup_{n \to \infty} G_{n}= + \infty$ and $\liminf_{n \to \infty} G_{n}= - \infty$.
\end{lemma}
\begin{proof}
For any $x \in \R$, we have that
\begin{align*}
\Pr \left(\limsup_{n \to \infty} G_n \ge x \right) &\ge \Pr \left( G_n \ge x \text{ i.o.} \right) \\
&= \Pr \left( \bigcap_{m=1}^\infty \bigcup_{n \ge m} \{G_n \ge x\} \right) \\
&= \lim_{m \to \infty} \Pr \left( \bigcup_{n \ge m} \{G_n \ge x\} \right) \\
&\ge  \lim_{m \to \infty} \Pr \left( G_m \ge x \right) \\
&= \frac{1}{2},
\end{align*}
by the central limit theorem, Proposition~\ref{prop:CLT}. 
With Lemma~\ref{lem:exchangeable} and the Hewitt--Savage zero--one law (Theorem 6.3.1), it follows that $\limsup_{n \to \infty} G_n \ge x$, a.s., and since $x \in \R$ was arbitrary, we get  $\limsup_{n \to \infty} G_n= +\infty$. A similar argument gives $\liminf_{n \to \infty} G_n= -\infty$.  
\end{proof}
 
Now we are ready to prove our main recurrence theorem in one dimension.

\begin{proof}[Proof of Theorem~\ref{thm:classification}.]
Under the conditions in part~(i) of the theorem, the process $(G_n)$ has the same distribution as the process $(-G_n)$,
and so we must be in either case (i) or (iv) of 
 Lemma~\ref{lem:4case}. The trivial case (i) is ruled out since $\E | X | >0$.  Thus case (iv) applies, and
 $G_n$ changes sign i.o., so by Lemma~\ref{lem:change_sign} we obtain the desired conclusion.

Under the conditions in part~(ii), Lemma~\ref{lem:change_sign2} applies, so (iv) applies again, and the same argument
gives the result.
\end{proof}

\section{Case of stable distribution}
\subsection{Local limit theorem for stable distribution}

For the remainder of this section we work towards a proof of Theorem~\ref{thm:stable-transience}. The proof rests on the following local limit theorem. We use the notation
\[ \cL_n := \left\{ n^{-1-1/\alpha} \left( \tfrac{1}{2} n (n+1) b + h \Z \right) \right\} ,\]
and $p_n (x) := \Pr ( G_n = n^{1/\alpha} x )$.

\begin{theorem}
\label{thm:SLCLT}
Suppose that~$d=1$ and~\eqref{ass:basicd2} holds, i.e., $\Pr(X \in b + h \Z)=1$ for $b \in \R$ and $h >0$ maximal. 
Suppose also that~\eqref{ass:stable} holds.  Then  
\begin{equation}
\lim_{n \to \infty} \sup_{\bx \in \cL_n} \left| \frac{n^{1+1/\alpha}}{h} p_n(x)-(\alpha+1)^{1/\alpha} g\left((\alpha+1)^{1/\alpha}x\right) \right| = 0, 
\label{eq:slclt}
\end{equation}
where $g(x)$ is the density of the stable distribution in~\eqref{ass:stable}.
\end{theorem}
\begin{proof}  
The proof is similar to that of~Theorem~\ref{thm:LCLT}, and can also be compared 
to the proof of the local limit theorem for sums of i.i.d.~random variables
in the domain of attraction of a stable law: see \cite[\S 4.2]{IL}.

Assumption~\eqref{ass:stable} implies that $n^{-1/\alpha} S_n$ converges
in distribution to a (constant multiple of) a random variable with characteristic function
$\nu (t) = \re^{-c |t|^\alpha}$, where $c >0$ and $\alpha \in (0,1)$; see Theorems~2.2.2 and~2.6.7 of~\cite{IL}. It also follows, by an examination of the statements of Theorems~2.6.1 and~2.6.7 of~\cite{IL} and the proof of Theorem~2.6.5 of~\cite{IL}, that for $t$ in a neighbourhood of $0$,
\begin{equation}
\label{domain}
 \log \varphi (t) = - c | t |^\alpha \left( 1 + \eps (t) \right) ,\end{equation}
where $| \eps (t) | \to 0$ as $t \to 0$.
 
Define $Y_n = \sum_{i=1}^n S_i$ and let
\[ \Phi_n(t) := \E \re^{in^{-1-1/\alpha}t Y_n}. \]
Using the $d=1$ case of the inversion formula~\eqref{eq:formula} with $W=\left( Y_n -\frac{n(n+1)}{2}b \right)/h \in \Z$, we get  
\begin{align*}
p_n(x )  = 
\frac{1}{2\pi } 
\int_{-\pi}^\pi \re^{ -\frac{iu}{h} \left( n^{1+1/\alpha} x -\frac{n(n+1)}{2}b \right)  } \E\left[ \re^{\frac{iu}{h} \left( Y_n -\frac{n(n+1)}{2}b \right)  }\right] \ud u, \text{ for } x \in \cL_n.
\end{align*}
Using the substitution $t = u n^{1+1/\alpha}/h$, we obtain
\begin{equation}
\frac{n^{1+1/\alpha}}{h} p_n(x ) = \frac{1}{2\pi} \int_{-\pi n^{1+1/\alpha}/h}^{\pi n^{1+1/\alpha}/h} \re^{-itx} \Phi_n(t)\ud t. 
\label{cal51}
\end{equation}
On the other hand, 
from the inversion formula for densities we have that
\begin{equation*}
g (x )= \frac{1}{2\pi} \int_{-\infty}^{\infty} \re^{-itx}\nu (t )\ud t,
\end{equation*}
where $g$ is the density corresponding to $\nu$. It follows that
\begin{align*}
(\alpha+1)^{1/\alpha} g\left((\alpha+1)^{1/\alpha}x\right) &=\frac{1}{2\pi} \int_{-\infty}^{\infty} (\alpha+1)^{1/\alpha} \re^{-it(\alpha+1)^{1/\alpha} x}\nu\left(t\right)\ud t \\
&=\frac{1}{2\pi} \int_{-\infty}^{\infty} \re^{-isx}\nu\left(\frac{s}{(\alpha+1)^{1/\alpha}}\right)\ud s,
\end{align*}
using the substitution $s = (\alpha+1)^{1/\alpha} t$.
Since $\nu(t)=\re^{-c|t|^\alpha}$, we get
\begin{equation}
(\alpha+1)^{1/\alpha} g\left((\alpha+1)^{1/\alpha}x\right)=\frac{1}{2\pi} \int_{-\infty}^{\infty} \re^{-itx - \frac{c|t|^\alpha}{\alpha+1}}\ud t.
\label{cal52}
\end{equation}
Subtracting equation~\eqref{cal52} from equation~\eqref{cal51}  we obtain
\begin{equation*}
\left| \frac{n^{1+1/\alpha}}{h} p_n(x ) - (\alpha+1)^{1/\alpha} g\left((\alpha+1)^{1/\alpha}x\right) \right| \le \sum_{k=1}^4 J_k (n) + J_5,
\end{equation*}
where 
\begin{align*}
J_1(n) &:= \int_{-A}^A \left| \Phi_n(t) -  \re^{- \frac{c|t|^\alpha}{\alpha+1}} \right| \ud t \\
J_2(n) &:= \int_{A \le |t| \le \delta n^{1/\alpha}} \left| \Phi_n(t) \right| \ud t \\
J_3(n) &:= \int_{\delta n^{1/\alpha} \le |t| \le \pi n^{1/\alpha}/h} \left| \Phi_n(t) \right| \ud t \\
J_4(n) &:= \int_{\pi n^{1/\alpha}/h \le |t| \le \pi n^{1+1/\alpha}/h} \left| \Phi_n(t) \right| \ud t \\
J_5 &:= \int_{|t| > A} \left| \re^{- \frac{c|t|^\alpha}{\alpha+1}} \right| \ud t \\
\end{align*}
for some constants $A$ and $\delta$ to be determined later. 
The statement of the theorem will follow once we show that 
\[ \lim_{A \to \infty} \lim_{n \to \infty} \left( \sum_{k=1}^4 J_k (n) + J_5 \right) = 0 .\]
Thus it remains to establish this fact.

Since $Y_n$ has the same distribution as $\sum_{j=1}^n jX_j$, we get
\begin{align}
\log \Phi_n(t) & = \log \prod_{j=1}^{n} \varphi\left( \frac{jt}{n^{1+1/\alpha}} \right) 
= \sum_{j=1}^n \log \varphi\left( \frac{jt}{n^{1+1/\alpha}} \right) \nonumber\\
& = - \frac{c | t|^\alpha}{n^{\alpha+1}} \sum_{j=1}^n j^\alpha  \left( 1 + \eps\left(\frac{jt}{n^{1+1/\alpha}} \right) \right),
\label{cal54}
\end{align}
using~\eqref{domain}. Since $|\eps(t)|\to0$ as $t \to 0$, we have
\begin{equation}
\label{epsilon-bound}
\lim_{n \to \infty} \sup_{t \in [-A, A]} \max_{j \in \{1,2, \cdots , n\}} \eps\left(\frac{jt}{n^{1+1/\alpha}} \right) =0.
\end{equation}
A simple consequence of the fact that $\sum_{k=0}^{n-1} k^\alpha \le \int_0^n u^\alpha \ud u  \le \sum_{k=1}^{n} k^\alpha$ for $\alpha > 0$ is
\begin{equation}
\label{sum-alpha}
\sum_{j=1}^n j^\alpha = \frac{n^{\alpha+1}}{\alpha+1} + O (n^\alpha ) .
\end{equation}
It follows from~\eqref{cal54}, \eqref{epsilon-bound}
and~\eqref{sum-alpha}, that uniformly over $t \in [-A,A]$, as $n \to \infty$,
\begin{equation*}
\log \Phi_n(t) = - \frac{c |t |^\alpha}{1+ \alpha}(1+ o(1)). 
\end{equation*}
It follows that $\lim_{n \to \infty} J_1 (n)=0$ for any $A \in \RP$.

For $J_2(n)$,  we see that 
\begin{equation*}
\lim_{\delta \to 0} \sup_n \sup_{t \in [-\delta n^{1/\alpha}, \delta n^{1/\alpha}]} \max_{j \in \{1,2, \cdots , n\}} \eps\left(\frac{jt}{n^{1+1/\alpha}} \right) =0.
\end{equation*}
So by~\eqref{sum-alpha} we may choose $\delta$ small enough so that for $t \in [-\delta n^{1/\alpha}, \delta n^{1/\alpha}]$,
\begin{equation*}
\log \Phi_n(t) \le - \frac{c |t|^\alpha}{n^{\alpha+1}}  \left( 1-\frac{1}{4} \right) \left( \frac{1}{\alpha + 1}n^{\alpha+1} + O  (n^\alpha ) \right). 
\end{equation*}
Hence for sufficiently large $n$, for all $t \in [-\delta n^{1/\alpha}, \delta n^{1/\alpha}]$,
\begin{equation*}
\log \Phi_n(t) \le - \frac{1}{2} \frac{c |t|^\alpha}{\alpha+1} .
\end{equation*}
It follows that 
\begin{equation*}
\sup_n J_2(n) \le  \int_{ |t| \ge A} \re^{-\frac{1}{2} \frac{c |t|^\alpha}{\alpha+1}} \ud t ,
\end{equation*}
which tends to $0$ as $A \to \infty$. 

Next we consider $J_3(n)$. First observe that
\begin{equation*}
\left| \Phi_n(t) \right| = \prod_{j=1}^{n} \left| \varphi\left( \frac{jt}{n^{1+1/\alpha}} \right) \right| \le \prod_{j=\lceil n/2 \rceil }^{n} \left| \varphi\left( \frac{jt}{n^{1+1/\alpha}} \right) \right|.
\end{equation*}
Now for any $\delta n^{1/\alpha} \le |t| \le \pi n^{1/\alpha}/h $ and any $\lceil n/2 \rceil \le j \le n$, we have 
\begin{equation*}
\frac{\delta}{2} \le \left| \frac{jt}{n^{1+1/\alpha}} \right| \le  \frac{\pi}{h}.
\end{equation*}
We can take  $\rho$ sufficiently small so that 
\begin{equation*}
\rho < \frac{\delta}{2} \le \left| \frac{jt}{n^{1+1/\alpha}} \right| \le  \frac{\pi}{h} < \frac{2\pi}{h}- \rho.
\end{equation*}
So an application of the $d=1$ case of Lemma~\ref{lem:property} gives, for all $n$,
\begin{equation*}
\sup_{\delta n^{1/\alpha} \le |t| \le \pi n^{1/\alpha}/h} \sup_{\lceil n/2 \rceil \le j \le n} \left| \varphi\left( \frac{jt}{n^{1+1/\alpha}} \right) \right| \le \re^{-c_\rho},
\end{equation*}
for some $c_\rho >0$. Hence we have  
\begin{equation*}
\sup_{\delta n^{1/\alpha} \le |t| \le \pi n^{1/\alpha}/h} |\Phi_n(t)| \le \re^{-n c_\rho / 2},
\end{equation*}
and hence
\begin{equation*}
J_3(n) = \int_{\delta n^{1/\alpha} \le |t| \le \pi n^{1/\alpha}/h} \left| \Phi_n(t) \right| \ud t \le
\frac{\pi}{h} n^{1/\alpha} \re^{-n c_\rho / 2} \to 0,
\end{equation*}
as $n \to \infty$.

For $J_4(n)$, we follow essentially the same counting argument as that used for $I_4(n)$ in Section~\ref{s:lcltpf}.
Let $t' = t/h$. 
Define
\begin{align*}
\Lambda' (t') := \left\{ n^{-1-1/\alpha}jt' : j \in \{1,2, \ldots, n\} \right\} \text{ and } L'_n(t') := \left\{ n^{-1-1/\alpha}ut' : 1 \le u \le n \right\}
\end{align*}
Let $\nu := n^{-1-1/\alpha} | t '|$ denote the spacing of the points of $\Lambda' (t')$.
Since $\pi n^{1/\alpha} \le |t'| \le \pi n^{1+1/\alpha}$, we have
\[ \frac{\pi}{n} \le \nu \le \pi ,\]
which is just the $d=1$ case of~\eqref{nu-bounds}. Since the counting argument is based
on the fact that there are $n$ points with spacing satisfying~\eqref{nu-bounds}, the
rest of the argument goes through unchanged and we get
\[
J_4(n) = \int_{\pi n^{1/\alpha}/h \le |t| \le \pi n^{1+1/\alpha}/h} \left| \Phi_n(t) \right| \ud t \le \frac{\pi}{h} n^{1+1/\alpha}  \exp\left\{- \eps c_\rho n \right\} \to 0,
\]
as $n \to 0$.

Finally, it is clear that $\lim_{A \to \infty} \sup_n J_5 = 0$.
\end{proof}

\subsection{Transience with stable distribution in one dimension}

Using the local limit theorem we just proved, we are ready for the proof of the transient case in one dimension.

\begin{proof}[Proof of Theorem~\ref{thm:stable-transience}]
Fix $x \in (0,\infty)$ and consider the interval $I = (-x,x)$.
Then $\Pr ( G_n \in I) = \Pr (n^{-1/\alpha} G_n \in n^{-1/\alpha} I )$. Since the lattice spacing of $\mathcal{L}_n$ is
of order $n^{-1-1/\alpha}$,  the interval $n^{-1/\alpha} I$ contains $O ( n)$ lattice points of $\cL_n$. Theorem~\ref{thm:SLCLT} shows that each such lattice point is associated with probability $O ( n^{-1-1/\alpha})$. So we get $\Pr (G_n \in I ) = O (n^{-1/\alpha})$, which is summable for $\alpha \in (0,1)$. Hence the Borel--Cantelli lemma implies that $\liminf_{n \to \infty} | G_n| \geq x$, a.s., and since $x$ was arbitrary the result follows.
\end{proof}

\chapter{Examples and applications}
\label{ch:egcom}

In this chapter, we will give a few examples to illustrate the complication on the lattice distribution. In order to apply the local limit theorems in Chapter 7, one has to carefully find the maximal span $h$ to interpret the theorem in the right lattice. This is not always immediate even for some classical random walks. 

Recall the notation that $\varphi (\bt) := \E [ \re^{i \bt^\tra X} ]$ for the characteristic function of $X$. Also, we set 
$U := \{ \bt \in \R^d : | \varphi (\bt ) | = 1 \}$, and
given an  invertible $d$ by $d$ matrix $H$, set $S_H := 2 \pi ( H^\tra)^{-1} \Z^d$.

\section{Lazy simple symmetric random walk}

It is remarkable that the trivial choice of lattice distribution for simple symmetric random walk, i.e. $\bb = \0$ and $H = I$,  the $d$ by $d$ identity matrix, is not the right one as the span $h$ in this case is not maximal. The right choice is actually quite a hassle to obtain. We will discuss that in the next section. Before that, there is actually an elementary walk that has the trivial choice as the right choice. It is the lazy simple symmetric random walk.

\begin{examplex}[Lazy SSRW on $\Z^d$]
Let $\be_1, \ldots, \be_d$ be the standard orthonormal basis vectors of $\R^d$, and suppose that
$\Pr ( X = \be_i ) = \Pr ( X = - \be_i ) = \frac{1}{4d}$ for all $i$, and $\Pr ( X = \0) = \frac{1}{2}$.
Then for $\bb = \0$ and $H = I$,  the $d$ by $d$ identity matrix, we have $\Pr ( X \in \Z^d ) =1$.
To verify that $L = \Z^d$ is minimal, it is sufficient (see Lemma~\ref{lem:equivalent}) to check
that $U = S_H = 2 \pi \Z^d$.
If $\bt = ( t_j ) \in \R^d$,
\[ \varphi (\bt) = \frac{1}{2} + \frac{1}{4d} \sum_{j=1}^d \left( \re^{it_j} + \re^{-it_j} \right) = \frac{1}{2} + \frac{1}{2d} \sum_{j=1}^d \cos  t_j .\]
Thus $\bt \in U$ if and only if $\cos t_j = 1$ for all $j$, i.e., $U = 2 \pi \Z^d = S_H$, as required.
Note that we could alternatively see the upper bound and apply Lemma~\ref{lem:h-bounded} to check that $h=1$ is maximal.
\end{examplex}

Maybe this walk is just too lazy to bother with a complicated choice of a lattice distribution.

\section{Simple symmetric random walk}

Without delay, we will show a right choice of lattice distribution for simple symmetric random walk.
\begin{examplex}[SSRW on $\Z^d$]
Suppose that $\Pr ( X = \be_i ) = \Pr ( X = - \be_i ) = \frac{1}{2d}$ for all $i$.
For SSRW the construction of $H$ for which~\eqref{ass:basicd2} holds is non-trivial.
For $d=1$, we take $b = -1$ and $h = 2$. In general $d \geq 2$, we take $H = ( h_{ij} )$ and $\bb = (b_i)$ defined as follows.
If $d = 2n - 1$ for $n \ge 2, n \in \Z$, we take
\begin{align*}
b_i &= -1 \quad \text{for all } i = 1, 2, \ldots , d; \\
h_{ij} &=
\begin{cases}
1 & \text{if } i-j \equiv 0 \text{ or } n \pmod{2n-1} ,\\
0 & \text{otherwise}.
\end{cases}
\end{align*}
If $d = 2n $ for $n \ge 1, n \in \Z$, we take
\begin{align*}
b_i &=
\begin{cases}
0 & \text{if } i=2n ,\\
-1 & \text{otherwise};
\end{cases}
\\
h_{ij} &=
\begin{cases}
-1 & \text{if } (i,j)= (2n, 1), \\
1 & \text{if } j-i \equiv 0 \text{ or } 1 \pmod{2n} \text{ and } (i,j) \neq (2n, 1),\\
0 & \text{otherwise}.
\end{cases}
\end{align*}
For example, for $d=2$ we have
\begin{equation*}
\bb =
\begin{pmatrix}
-1 \\
0
\end{pmatrix} 
\quad \text{and} \quad H=
\begin{pmatrix}
1 & 1 \\
-1 & 1
\end{pmatrix}.
\end{equation*}
For $d=3$, we have
\begin{equation*}
\bb =
\begin{pmatrix}
-1 \\
-1 \\
-1 
\end{pmatrix} 
\quad \text{and} \quad H=
\begin{pmatrix}
1 & 1 & 0 \\
0 & 1 & 1 \\
1 & 0 & 1
\end{pmatrix}.
\end{equation*}
For $d=4$, we have 
\begin{equation*}
\bb =
\begin{pmatrix}
-1 \\
-1 \\
-1 \\
0
\end{pmatrix} 
\quad \text{and} \quad H=
\begin{pmatrix}
1 & 1 & 0 & 0 \\
0 & 1 & 1 & 0 \\
0 & 0 & 1 & 1 \\
-1 & 0 & 0 & 1
\end{pmatrix}.
\end{equation*}
For $d=5$, we have
\begin{equation*}
\bb =
\begin{pmatrix}
-1 \\
-1 \\
-1 \\
-1 \\
-1
\end{pmatrix} 
\quad \text{and} \quad H=
\begin{pmatrix}
1 & 0 & 1 & 0 & 0 \\
0 & 1 & 0 & 1 & 0 \\
0 & 0 & 1 & 0 & 1 \\
1 & 0 & 0 & 1 & 0 \\
0 & 1 & 0 & 0 & 1
\end{pmatrix},
\end{equation*}
and so on. Note that $h=2$ for all such $H$. It is elementary to verify that
$\Pr ( X \in \bb + H \{ 0, 1 \}^d = 1)$. It suffices to check
that $H^{-1} ( \bx - \bb ) \in \{0,1\}^d$ for any $\bx = \pm \be_i$.
For example, in the case $d=2n-1$ we have that 
$H^{-1}$ has elements $h_{ij}^{-1}$ given by
\[ h^{-1}_{ij} = \begin{cases} \frac{1}{2} & \text{if } i - j = 0, 1, \ldots, n-1 \pmod{2n-1} , \\
- \frac{1}{2} &\text{otherwise} , \end{cases} \]
and then one checks that, for example, $H^{-1} ( \be_i - \bb) = \ba$ where $\ba$ has all components zero apart from
$a_i = \cdots = a_{i+n-1} = 1$ (for $i < n-1$). The other cases are similar. 

We show that~\eqref{ass:basicd2} holds for SSRW with this choice of $H$,
by checking  (see Lemma~\ref{lem:equivalent}) that $U = S_H$. Since Lemma~\ref{lem:property} shows that
$S_H \subseteq U$, it suffices to show that $U \subseteq S_H$.
For SSRW on $\Z^d$, 
if $\bt = ( t_j ) \in \R^d$,
\[ \varphi (t) = \frac{1}{2d} \sum_{j=1}^d \left( \re^{it_j} + \re^{-it_j} \right) = \frac{1}{d} \sum_{j=1}^d \cos  t_j .\]
So $\bt \in U$ if and only if
$| \sum_{j=1}^d \cos t_j | = d$, which occurs if and only if either (i)
$\cos t_j = 1$ for all $j$, or (ii) $\cos t_j = -1$ for all $j$.
Case (i) is equivalent to 
$\bt \in 2 \pi \Z^d$ and case (ii) is equivalent to 
$\bt \in \pi \vo + 2 \pi \Z^d$, where $\vo$ is the vector of all $1$s.
Hence
\[ U = ( 2 \pi \Z^d ) \cup (\pi \vo + 2 \pi \Z^d ) .\]

Consider $\bx \in U$. Then for some $\ba \in \Z^d$, either
 (i) $\bx = 2 \pi \ba$, or (ii) $\bx = \pi \vo + 2 \pi \ba$.
In case (i), let $\bz = H^\tra \ba$; then since all entries in $H$
are integers, we have
$\bz \in \Z^d$ and
$2\pi ( H^\tra )^{-1} \bz = 2 \pi \ba = \bx$, so $\bx \in S_H$.
In case (ii), let $\bz = H^\tra ( \frac{1}{2} \vo + \ba )$.
Note that if $d$ is odd then $\frac{1}{2} H^\tra \vo = \vo$ while if
$d$ is even, $\frac{1}{2} H^\tra \vo = (0,1,1,\ldots,1)^\tra$;
in any case it follows that $\bz \in \Z^d$.
Then  
$2 \pi ( H^\tra )^{-1} \bz = \pi \vo + 2 \pi \ba = \bx$, so $\bx \in S_H$.
Thus $U \subseteq S_H$.
\end{examplex}

In general, it is very difficult to find a right lattice distribution for a specific random walk which has the maximal span. It requires onerous effort for trial and error. Even if we can find a possible candidate, it is tricky to prove that indeed it is the right choice. The author feels like it is more an unsolved discrete geometry problem, rather than anything in the field of probability, and we should proceed to our next part of our thesis, leaving the pleasure for the reader to discover more with lattice distribution.

\chapter*{Part III: \\ Conclusions}
\addcontentsline{toc}{chapter}{Part III:  Conclusions}

\chapter{Open problems and conjectures}
\label{ch:opc}

In this chapter, we are going to give some open problems and conjectures for further research topics.

\section{Half strip}

There are a few models related to the continuum analogue of the half strip model. One of the popular model is called `Markov-modulated diffusions' \cite{MN}. Consider $(X_t, \eta_t) \in \R \times S$, where $S$ is a finite set. The reason why taking the whole real lines instead of the half lines is because we need to define the movement of the walk using stochastic differential equation. Assume that $\eta_t$ is a continuous-time irreducible Markov chain on $S$ with stationary distribution $\pi$. Let $\mu: S \to \R$ and $\sigma : S \to (0, \infty)$, then the movement of the horizontal direction can be determined by the following stochastic differential equation
\begin{equation}
\ud X_t = \mu(\eta_t) \ud t + \sigma (\eta_t) \ud W_t,
\end{equation}
where $W_t$ is a Brownian motion on $\R$.

For a full continuum analogue for the (half) strip model, define $\mathcal{A}$ as a manifold in $\R^d$ without boundary. Then our point of interest is $(X_t, \eta_t) \in \R \times \mathcal{A}$, a diffusion process. 

On $\R \times \mathcal{A}$, we assume $\eta_t$ is e.g. a Brownian motion on $\mathcal{A} = \mathcal{S}^{d-1}$ and also assume there exist a stationary distribution $\pi$ such that $\eta_t \to \pi$ on $\mathcal{A}$. Again we have the following stochastic differential equation,
\begin{equation}
\ud X_t = \mu(\eta_t) \ud t + \sigma (\eta_t) \ud W_t,
\end{equation}
where $W_t$ is a Brownian motion on $\R$. We also assume $\mu: \mathcal{A} \to \R$ is bounded and continuous and $\sigma$ bounded uniformly from $\{0, \infty \}$, e.g. $\sigma =1$. We conjecture the following recurrence classification, similar to our discrete model.

\begin{conjecture} 
If $\int_\mathcal{A} \pi(\ud x) \mu(x) > 0$, then $X_t \to \infty$. If $\int_\mathcal{A} \pi(\ud x) \mu(x) < 0$, then $X_t \to -\infty$.
\end{conjecture}

\section{Centre of mass}
\label{s:opccom}
In $d = 1$, we believe that for any zero drift random walk, $G_n$ is recurrent:
\begin{conjecture}
Suppose $d=1$. If $\E X =0$, then $G_n$ is recurrent.\end{conjecture}

For $d \geq 2$, we believe that $G_n$ is always `at least as transient' as the situation in Theorem~\ref{thm:classification2}:

\begin{conjecture}
Suppose that $\supp X$ is not contained in a one-dimensional subspace of $\R^d$. Then
\[ \liminf_{n \to \infty} \frac{ \log \| G_n \|}{\log n} \geq \frac{1}{2} \qquad \as \]
\end{conjecture}

Also, is there an analogue of Chung-Fuchs Theorem, i.e., criterion for recurrence using characteristic function of $X$?

There is also a continuum analogue for the centre of mass process. Suppose we start with $B_t$, a Brownian motion on $\R^d$. Consider the following stochastic differential equation,
\begin{equation}
\ud X_t = B_t dt.
\end{equation}
The process $X_t$ is also known as the integrated Brownian motion, which can also be written as
\begin{equation}
 X_t = \int_0^t B_s ds.
\end{equation}
The process $(B_t, X_t)$ is the famous Kolmogorov diffusion process \cite{ANK2}, which satisfies the Markov property. The joint distribution of the continuum analogue for the random walk and its centre of mass process, $(B_t, \frac{1}{t} X_t)$, however, behaves very differently. It is expected that $\frac{1}{t} X_t$ is recurrent if and only if $d=1$.

%
%


\end{document}